\documentclass[12pt,leqno]{amsart}
\usepackage{latexsym,amsmath,amssymb,amscd,xcolor,mathrsfs,marvosym,mathtools}
\usepackage{tikz}
\usepackage[mathscr]{eucal}
\usepackage[T1]{fontenc}
\usepackage[utf8]{inputenc}
\usepackage{graphicx,hyperref}
\usepackage[shortlabels]{enumitem}

plus 6pt minus 12pt
\newtheorem{theorem}{Theorem}
\newtheorem*{theorem*}{Theorem}
\newtheorem{lemma}[theorem]{Lemma}
\newtheorem{corollary}[theorem]{Corollary}
\newtheorem{proposition}[theorem]{Proposition}

\theoremstyle{definition}
\newtheorem{remark}[theorem]{Remark}
\newtheorem{definition}[theorem]{Definition}

\newcommand{\overbarint}{
\rule[.036in]{.12in}{.009in}\kern-.16in \displaystyle\int }

\newcommand{\overbarcal}{\mbox{$ \rule[.036in]{.11in}{.007in}\kern-.128in\int $}}

\newcommand{\bbbn}{\mathbb N}
\newcommand{\bbbz}{\mathbb Z}
\newcommand{\bbbr}{\mathbb R}
\newcommand{\eps}{\varepsilon}

\newcommand{\bbbb}{\mathbb B}

\def\5{\text{\Saturn}}

\def\diam{\operatorname{diam}}

\def\id{{\rm id\, }}

\newcommand{\overbar}[1]{\mkern 1.7mu\overline{\mkern-1.7mu#1\mkern-1.5mu}\mkern 1.5mu}

\def\mvint_#1{\mathchoice
          {\mathop{\vrule width 6pt height 3 pt depth -2.5pt
                  \kern -8pt \intop}\nolimits_{\kern -3pt #1}}%
          {\mathop{\vrule width 5pt height 3 pt depth -2.6pt
                  \kern -6pt \intop}\nolimits_{#1}}%
          {\mathop{\vrule width 5pt height 3 pt depth -2.6pt
                  \kern -6pt \intop}\nolimits_{#1}}%
          {\mathop{\vrule width 5pt height 3 pt depth -2.6pt
                  \kern -6pt \intop}\nolimits_{#1}}}

\numberwithin{theorem}{section} \numberwithin{equation}{section}

\title[Homeomorphisms with prescribed derivative]{Constructing diffeomorphisms and homeomorphisms with prescribed derivative}
\author[P. Goldstein]{Pawe\l{}  Goldstein}
\address{Pawe\l{} Goldstein, Institute of Mathematics, Faculty of Mathematics, Informatics and Mechanics, University of Warsaw, Banacha 2, 02-097 Warsaw, Poland} \email{P.Goldstein@mimuw.edu.pl}
\author[Z. Grochulska]{Zofia Grochulska}
\address{Zofia Grochulska, Doctoral School of Exact and Natural Sciences, University of Warsaw, Banacha 2, 02-097 Warsaw, Poland} \email{z.grochulska@uw.edu.pl}

\thanks{P.G. and Z.G. were supported by NCN grant no 2019/35/B/ST1/02030}

\author[P. Haj\l{}asz]{Piotr Haj\l{}asz}
\address{Piotr Haj\l{}asz, Department of Mathematics, University of Pittsburgh, Pittsburgh, PA 15260, USA}
\email{hajlasz@pitt.edu}
\thanks{P.H. was supported by NSF grant  DMS-2055171 and
by Simons Foundation grant 917582.}

\begin{document}

\begin{abstract}
We prove that for any measurable mapping $T$ into the space of matrices with positive determinant, there is a diffeomorphism whose derivative equals $T$ outside a set of measure less than $\varepsilon$. We use this fact to prove that for any measurable mapping $T$ into the space of matrices with non-zero determinant (with no sign restriction), there is an almost everywhere approximately differentiable homeomorphism whose derivative equals $T$ almost everywhere.
\end{abstract}

\maketitle
\section{Introduction}
There are three main results in the paper, Theorem~\ref{T31}, Theorem~\ref{T36} and Theorem~\ref{T47}. Theorem~\ref{T5} is also new.
Throughout the paper we assume that $n\geq 2$.

Alberti \cite[Theorem~1]{Alberti} proved the following result:
\begin{theorem}[Alberti]
\label{T30}
Let $\Omega\subset\bbbr^n$ be a bounded domain and let $T:\Omega\to\bbbr^n$ be a measurable function. Then, for every $\eps>0$, there is a function $\phi\in C_c^1(\Omega)$ and a compact set $K\subset\Omega$ such that $|\Omega\setminus K|<\eps$ and $D\phi(x)=T(x)$ for all $x\in K$.
\end{theorem}
Because of a certain analogy to the classical Lusin theorem, the above result is known as the {\em Lusin property for gradients}.
Applying the result to components of a matrix-valued measurable mapping $T:\Omega\to M^{n\times n}$, we obtain a mapping $\Phi\in C_c^1(\Omega,\bbbr^n)$ and a compact set $K\subset\Omega$ such that $|\Omega\setminus K|<\eps$ and $D\Phi(x)=T(x)$ for all $x\in K$. However, the result does not provide any information about geometric properties of the mapping $\Phi:\Omega\to\bbbr^n$. Our first main result addresses this issue and we prove that if $\det T>0$ a.e., then we can actually construct a $C^1$-diffeomorphism such that $D\Phi=T$ outside an open set of a small measure.
By $GL(n)$ we denote the space of all real $n\times n$ matrices with non-zero determinant, while by
$GL(n)^+$ we denote the space of real $n\times n$ matrices with positive determinant.
\begin{theorem}
\label{T31}
Let $\Omega \subset \bbbr^n$ be a~bounded domain and $F: \Omega \to \bbbr^n$ an orientation preserving diffeomorphism onto the bounded image $F(\Omega)$. Suppose that $T: \Omega \to GL(n)^+$ is a measurable function such that $\int_\Omega \det T  \leq |F(\Omega)|$. Then, for any $\eps > 0$, there exists a $C^1$-diffeomorphism $\Phi: \Omega \to F(\Omega)$ with the following properties:
\begin{enumerate}[(a)]
\item $\Phi(x) = F(x)$ near $\partial \Omega$;
\item there exists a~compact set $K \subset \Omega$ such that for every $x \in K$, $D\Phi (x) = T (x)$ and $|\Omega \setminus K|< \eps$.
\end{enumerate}
\end{theorem}
The condition $\int_\Omega\det T\leq |F(\Omega)|$ is necessary. Indeed, if $\int_\Omega\det T>|F(\Omega)|$,
then by the change of variables formula, for $\eps$ sufficiently small, we have
$$
|F(\Omega)|=|\Phi(\Omega)|=\int_\Omega\det D\Phi>\int_K\det T> |F(\Omega)|,
$$
which is a contradiction. The result shows that the upper bound for the integral of $\det T$ is in fact necessary and sufficient for the existence of a diffeomorphism $\Phi$ as in Theorem~\ref{T31}.

While we use Alberti's result in the proof of Theorem~\ref{T31}, it is just  a tip of the iceberg as the main difficulty lies in making sure that the map is a diffeomorphism. To this end, we use a theorem of Dacorogna and Moser \cite{DacorognaMoser} about smooth diffeomorphisms with prescribed Jacobians (Lemma~\ref{T11} below), some topological arguments of Munkres \cite{munkres60} (see Lemma~\ref{T14} and Section~\ref{S23}), and a lot of explicit constructions of diffeomorphisms with prescribed properties (see Section~\ref{S23}). All diffeomorphisms are constructed rigorously.

While Theorem~\ref{T31} tells us about diffeomorphisms with prescribed derivative outside a set of small measure, the next result, Theorem~\ref{T36}, is about homeomorphisms with almost everywhere prescribed derivative.

We say that a measurable function $f:\Omega\to\bbbr^m$, $\Omega\subset\mathbb{R}^n$, is approximately
differentiable at $x\in \Omega$ if there is a measurable set $E_x\subset \Omega$
and a linear function $D_{\rm a}f(x):\bbbr^n\to\bbbr^m$ such that $x$ is a density point of $E_x$ and
$$
\lim_{E_x\ni y\to x} \frac{|f(y)-f(x)-D_{\rm a}f(x)(y-x)|}{|y-x|} = 0.
$$
See Section~\ref{AD} for more information about approximately differentiable functions.

One of the reasons why the class of mappings that are approximately differentiable a.\,e. is important is
the following general form of the change of variables formula that was essentially proved by Federer \cite{Federer2}  (see \cite{hajlasz} for this particular statement and a~detailed proof). By $N(\Phi,y)$ we will denote the number of points (cardinality) of the preimage $\Phi^{-1}(y)$.

We say that a measurable mapping $f:\Omega\to\mathbb{R}^n$ defined on an open set $\Omega\subset\bbbr^n$ satisfies the {\em Lusin condition~(N)} if it maps sets of measure zero to sets of measure zero.
\begin{theorem}[Federer]
\label{T42}
Let $\Phi:\Omega\to\mathbb{R}^n$ be a measurable mapping defined on an open set $\Omega\subset\bbbr^n$. Assume that it is approximately differentiable a.e. If $\Phi$ satisfies the Lusin condition (N), then for any measurable function $f:\bbbr^n\to\bbbr$ we have
\begin{equation}
\label{eq84}
\int_\Omega (f\circ\Phi)(x)|\det D_{\rm a}\Phi(x)|\, dx = \int_{\Phi(\Omega)} f(y)N(\Phi,y)\, dy.
\end{equation}
If $\Phi$ does not satisfy the condition (N), then we can redefine $\Phi$ on a set of measure zero so that after the redefinition, $\Phi$ satisfies the condition (N) and hence \eqref{eq84}.
\end{theorem}
To be more precise, \eqref{eq84} means that the function on the left hand side is integrable if and only if the function on the right hand side is integrable and then we have equality.

In particular, if $\Phi:\Omega\to\bbbr^n$ is a homeomorphism that is approximately differentiable a.e. and satisfies condition~(N), then
applying \eqref{eq84} to $f=\chi_{\Phi(E)}$, where $E$ is measurable, we get
\begin{equation}
\label{eq115}
\int_E |\det D_{\rm a}\Phi(x)|\, dx = |\Phi(E)|
\end{equation}
($\Phi(E)$ is measurable, because $E$ is a union of a Borel set and a set of measure zero and homeomorphisms preserve Borel sets). In particular,
\begin{equation}
\label{eq86}
\int_\Omega |\det D_{\rm a}\Phi(x)|\, dx = |\Phi(\Omega)|.
\end{equation}
If the homeomorphism $\Phi$ is Lipschitz, then $D\Phi$ must satisfy several conditions, for example: $\det D\Phi$ cannot change sign and $\partial^2\Phi/\partial x_i\partial x_j=\partial^2\Phi/\partial x_j\partial x_i$ (in the distributional sense).
Now assume that $\Phi$ is a homeomorphism satisfying assumptions of Theorem~\ref{T42}, i.\,e., $\Phi$ is approximately differentiable a.\,e.\ and satisfies the condition (N).
{\em What conditions must $D_{\rm a}\Phi$ satisfy?} It is an important question, because Sobolev and $BV$ homeomorphisms are approximately differentiable a.\,e.\ and they are abundant in
non-linear elasticity, geometric problems in calculus of variations, and in the theory of quasiconformal and quasiregular mappings. Because of the nature of the problems in which they appear, understanding the relation between geometric and topological properties of the mapping and properties of the derivative is crucial. Our second main result,
Theorem~\ref{T36}, answers the above question: {\em Basically none}. Indeed, Theorem~\ref{T36} shows that we can construct a homeomorphism $\Phi$ with arbitrarily prescribed approximate derivative. This answers a conjecture from \cite{GoldsteinHajlasz17} in the positive. For simplicity, we assume that $\Omega$ is the interior of the unit cube.
\begin{theorem}
\label{T36}
Let $\mathscr{Q}=[0,1]^n$.
For any measurable map $T: \mathscr{Q} \to GL(n)$ that satisfies
\begin{equation}
\label{eq85}
\int_\mathscr{Q} |\mathrm{det}\, T(x)|\, dx = 1,
\end{equation}
there exists an a.\,e.\ approximately differentiable homeomorphism $\Phi:\mathscr{Q}\to \mathscr{Q}$ such that $\Phi|_{\partial \mathscr{Q}} = \id$ and $ D_{\rm{a}}\Phi = T$ a.e.
Moreover,
\begin{itemize}
\item[(a)] $\Phi^{-1}$ is approximately differentiable a.\,e.\ and $D_{\rm a}\Phi^{-1}(y)=T^{-1}(\Phi^{-1}(y))$ for almost all $y\in \mathscr{Q}$;
\item[(b)] $\Phi$ preserves the sets of measure zero, i.e., for any $A\subset \mathscr{Q}$,
$$
|A|=0
\qquad
\text{if and only if}
\qquad
|\Phi(A)|=0.
$$
\item[(c)] $\Phi$ is a limit of $C^\infty$-diffeomorphisms $\Phi_k:\mathscr{Q}\to \mathscr{Q}$, $\Phi_k=\id$ in a neighborhood of $\partial\mathscr{Q}$,
in the uniform metric, i.e., $\Vert\Phi-\Phi_k\Vert_\infty+\Vert\Phi^{-1}-\Phi_k^{-1}\Vert_\infty\to 0$ as $k\to\infty$.
\end{itemize}
\end{theorem}
Note that $\Phi$ cannot be differentiable in the classical sense if $\det T<0$ on a set of positive measure, because $\Phi$ is orientation preserving and hence it follows from the degree theory that if $\Phi$ is differentiable at $x$ and $\det D\Phi(x)\neq 0$, then the classical derivative must satisfy $\det D\Phi(x)>0$, see e.g. the proof of \cite[Theorem~5.22]{HenclKoskela}.

It follows from \eqref{eq86} that if $T:\mathscr{Q}\to GL(n)$ is measurable, and $\Phi:\mathscr{Q}\to \mathscr{Q}$ is a homeomorphism that is approximately differentiable, satisfies the Lusin condition (N), and $D_{\rm a}\Phi(x)=T(x)$ a.\,e., then $T$ must satisfy \eqref{eq85}. As Theorem~\ref{T36} shows, this is actually the only constraint for the mapping $T$.

Some related results have been obtained earlier. The following  was the main result in \cite{GoldsteinHajlasz17} (cf.\ \cite{GoldsteinHajlasz18} and Theorem~\ref{T5} below).
\begin{theorem}
\label{T33}
There is a homeomorphism $\Phi:\mathscr{Q}=[0,1]^n\to \mathscr{Q}$, $\Phi|_{\partial \mathscr{Q}}=\id$, such that
$$
D_{\rm a} \Phi=
\left[
\begin{array}{ccccc}
1       &      0       &   \ldots   &   0      &     0      \\
0       &      1       &   \ldots   &   0      &     0      \\
\vdots  &   \vdots     &   \ddots   &  \vdots  &    \vdots  \\
0       &      0       &   \ldots   &   1      &     0      \\
0       &      0       &   \ldots   &   0      &    -1      \\
\end{array}
\right]\ \
 \text{ a.e.}
$$
and $\Phi$ is a limit of volume preserving $C^\infty$-diffeomorphisms $\Phi_k:\mathscr{Q}\to \mathscr{Q}$, $\Phi_k=\id$ in a neighborhood of $\partial\mathscr{Q}$,
in the uniform metric.\footnote{The statement of \cite[Theorem~1.4]{GoldsteinHajlasz17} says $\Phi_k|_{\partial\mathscr{Q}}=\id$, but the proof shows that $\Phi_k=\id$ in a neighborhood of $\partial{\mathscr{Q}}$.}
\end{theorem}

Thus despite the fact that $\Phi$ is an orientation preserving homeomorphism, it can have negative Jacobian.
In the case of Sobolev  homeomorphisms the problem of the sign of the Jacobian had been studied in \cite{CDH,CHT,GoldsteinHajlasz19,henclm1,henclv} and the example constructed in \cite{CHT} gave a negative answer to a Ball-Evans question \cite{ball} about approximation of Sobolev homeomorphisms by diffeomorphisms or piecewise linear homeomorphisms.
While the papers mentioned above provide a large number of examples of homeomorphisms with unexpected properties of the derivative, these are just examples. On the other hand, Theorem~\ref{T36} provides a general answer.

There has been substantial interest in finding diffeomorphisms and homeomorphisms with prescribed Jacobian. The first result in this line is due to Oxtoby and Ulam \cite{OxtobyUlam41} followed by many other papers \cite{AlpernPrasad,Banyaga,BuragoK,CsatoDacorognaKneuss,DacorognaMoser,goffman,GuerraKL,McMullen,Moser,RiviereYe}. It is possible if the Jacobian is positive and sufficiently regular, but if $f\in L^p$, $p>1$, then, in general, there are no solutions to $\det D\Phi=f$ with expected Sobolev regularity $\Phi\in W^{1,np}$, \cite{GuerraKL} (see also \cite{BuragoK,McMullen}). On the other hand, Theorem~\ref{T36} shows that with weaker regularity one can find a homeomorphism with not only prescribed Jacobian, but with arbitrarily prescribed derivative. To the best of our knowledge this is the first result of this kind.

Here is another reason why the class of functions that are approximately differentiable a.\,e.\ is natural.
In the class of measurable functions $f:E\to\bbbr$ defined on a measurable set $E\subset\bbbr^n$ of finite measure (we allow here $n=1$) we define the {\em Lusin metric}
$$
d_L(f,g)=|\{x\in E:\, f(x)\neq g(x)\}|.
$$
This is a metric if we identify functions that are equal a.e. It is easy to see that the space of measurable functions is complete with respect to the metric $d_L$, see Lemma~\ref{T48}, and according to Lusin's theorem, continuous functions are dense in that space.
Now, if $\Omega\subset\bbbr^n$ is a bounded domain, Whitney's theorem (Lemma~\ref{T2} below) asserts that the closure of $C^1(\Omega)$ in the metric $d_L$ is precisely the class of functions that are a.\,e.\ approximately differentiable. Therefore, the class of functions that are a.\,e.\ approximately differentiable is the counterpart of the class of measurable functions when in Lusin's theorem we replace continuous functions with $C^1$ functions.

In particular, if $\mathscr{Q}=[0,1]^n$ and $\Phi_k:\mathscr{Q}\to\mathscr{Q}$ is a sequence of surjective, orientation preserving $C^1$-diffeomorphisms that converge to a surjective homeomorphism $\Phi:\mathscr{Q}\to\mathscr{Q}$ in the Lusin metric, then $\Phi$ is approximately differentiable a.\,e.\ and it is easy to see (we leave details to the reader) that
$$
\det D_{\rm a} \Phi(x)>0 \text{ a.\,e.}
\quad
\text{and}
\quad
\int_{\mathscr{Q}} \det D_{\rm a}\Phi(x)\, dx\leq 1.
$$
We may again ask the same question as before: {\em What other conditions must $D_{\rm a}\Phi$ satisfy?} The answer is similar to that given before: {\em Basically none.}
\begin{theorem}
\label{T47}
Let $\mathscr{Q}=[0,1]^n$. For any measurable map $T:\mathscr{Q}\to GL(n)^+$ that satisfies
$$
\int_\mathscr{Q} \det T(x)\, dx=1,
$$
there exists a sequence of $C^1$-diffeomorphisms $\Phi_k:\mathscr{Q}\to\mathscr{Q}$, $\Phi_k=\id$ in a neighborhood of $\partial\mathscr{Q}$, that converges both in the uniform metric $d$ and the Lusin metric $d_L$ to a homeomorphism $\Phi:\mathscr{Q}\to\mathscr{Q}$, $\Phi|_{\partial\mathscr{Q}}=\id$, that is a.\,e.\ approximately differentiable and satisfies $D_{\rm a}\Phi=T$ a.\,e.
\end{theorem}
This result is a straightforward consequence of the proof of Theorem~\ref{T36}. Indeed, in the case $\det T>0$ a.\,e.\  we constructed in Section~\ref{gen} a sequence of $C^1$-diffeomorphisms $\Phi_k$ with properties (i)-(v) which imply that the sequence $\Phi_k$ converges in both metrics $d$ and $d_L$ to a homeomorphism $\Phi$ with the properties listed in Theorem~\ref{T47}.

The paper is structured as follows. In Section~\ref{GN} we fix the notation used throughout the paper. Technical results needed in the proof of Theorem~\ref{T31} are collected in Section~\ref{S1}. This section contains a lot of explicit constructions of diffeomorphisms, including results about gluing of diffeomorphisms. We believe that these results can be used in other situations. Section~\ref{trzy} is devoted to the proof of Theorem~\ref{T31}.
In Section~\ref{S3} we discuss diffeomorphic partitions of cubes and constructions of diffeomorphisms that transfer measures between the cells of partitions. The material is motivated by the celebrated work of Oxtoby and Ulam \cite{OxtobyUlam41}, but the constructions are far more sophisticated due to required regularity and  additional constraints. These constructions are needed in the proof of Theorem~\ref{T36}. In Section~\ref{AD} we review basic properties of approximately differentiable functions. However, Theorem~\ref{T5}, a generalization of Theorem~\ref{T33}, is new. The final Section~\ref{szesc} is devoted to the proof of Theorem~\ref{T36}. In Appendix, we prove Lemma~\ref{T48}.

Except for Theorems~\ref{T30}, \ref{T42} and~\ref{T33} recalled in the Introduction, we use a convention that well-known results needed in our proofs and results of technical character are called ``Lemma'', ``Proposition'' or ``Corollary''.

\subsection*{Acknowledgements}
Piotr Hajłasz appreciates the hospitality of the University of Warsaw, where part of this work was conducted. His stay in Warsaw received funding
from the University of Warsaw via the  IDUB project (Excellence Initiative Research University) as part of the Thematic Research Programme \emph{Analysis and Geometry}.

Paweł Goldstein enjoyed the hospitality of the University of Pittsburgh during his several visits.

\section{General notation}
\label{GN}

In this section we explain most of the notation used throughout the paper. Also some of the notions introduced here are recalled later in the paper to make reading of this long paper easier.

Throughout the paper we assume that $n\geq 2$. The only exception is Lemma~\ref{T48}.

$\bbbn$ and $\bbbz$ will denote the sets of positive and all integers, respectively.

The space of real $n\times n$ matrices, invertible matrices and matrices with positive determinant will be denoted by $M^{n\times n}$, $GL(n)$, and $GL(n)^+$, respectively. The identity matrix will be denoted by $\mathcal{I}$. The operator and the Hilbert-Schmidt norms of $A\in M^{n\times n}$ will be denoted by $\Vert A\Vert$ and $|A|$, respectively. It is easy to see that $\Vert A\Vert\leq |A|$.

The tensor product of vectors $u,v\in\bbbr^n$ is the matrix
$u\otimes v=[u_iv_j]_{i,j=1}^n\in M^{n\times n}$. It is easy to see that $\Vert u\otimes v\Vert=|u\otimes v|=|u|\, |v|$.
Note that if $U\subset\bbbr^n$ is open and $F:U\to\bbbr^n$, $\eta:U\to\bbbr$ are differentiable, then
\begin{equation}
\label{eq54}
D(\eta F)=F\otimes D\eta+\eta DF
\quad
\text{so}
\quad
\Vert D(\eta F)\Vert\leq |D\eta|\, |F|+|\eta|\,\Vert DF\Vert.
\end{equation}

Symmetric difference of sets $A$ and $B$ is $A\vartriangle B=(A\setminus B)\cup (B\setminus A)$.

The interior and the closure of a set $A$ will be denoted by $\mathring{A}$ and $\overbar{A}$. Boundary of the set will be denoted by $\partial A$.
We write $A\Subset B$ if $\overbar{A}$ is a compact subset of $\mathring{B}$.

We say that a map between topological spaces is {\em proper} if preimages of compact sets are compact.

Open balls in $\mathbb{R}^n$ will be denoted by $B(x,r)$. $\Omega$ will always denote an open subset of~$\mathbb{R}^n$. By a domain we mean an open and connected set.

$\mathscr{Q}$ will always denote the closed unit cube $\mathscr{Q}=[0,1]^n$.
If $x=(x_1,\ldots,x_n)\in\bbbr^n$, then we define $\Vert x\Vert_\infty=\max_i|x_i|$.
$\mathring{Q}(p, r) = \{x: \, \Vert x - p\Vert_\infty <r\}$ and $Q(p, r) = \{x: \, \Vert x - p\Vert_\infty \leq r\}$ denote the open and the closed cube centered at $p$ of side-length $2r$. We will denote with $\mathsf{L}(Q)$ the side-length of cube $Q$.

The space of $k$-times continuously differentiable functions with compact support will be denoted by $C^k_c(\Omega)$. Approximate derivative will be denoted by $D_{\rm a}f(x)$. The Jacobian of a mapping $\Phi:\Omega\to\bbbr^n$, $\Omega\subset\bbbr^n$, will be denoted by $\det D\Phi(x)$ or by $J_\Phi(x)$.

If $\Omega\subset\bbbr^n$ is open, then by a diffeomorphism (homeomorphism) $\Phi:\Omega\to\bbbr^n$ we mean a diffeomorphism (homeomorphism) onto the image, i.e., diffeomorphism (homeomorphism) between $\Omega$ and $\Phi(\Omega)$.

A diffeomorphism of a closed set $A$ is a mapping that extends to a diffeomorphism of a neighborhood of $A$. We will use this notion mostly to discuss diffeomorphisms  of closed cubes.

A set $P$ is a \emph{diffeomorphic closed cube} if there is a diffeomorphism $\Theta$ defined on a neighborhood of $P$ such that $\Theta(P)=\mathscr{Q}$. We say that $\mathscr{P}=\{P_i\}_{i=1}^N$ is \emph{a partition of $P$} if $P=\bigcup_{i=1}^N P_i$ and $P_i$ are diffeomorphic closed cubes with pairwise disjoint interiors.

Important examples of partitions of $\mathscr{Q}$ are the \emph{dyadic partitions} into $2^{nk}$ identical cubes of edge length $2^{-k}$, for $k=0,1,2,\ldots$

A partition $\mathscr{P}=\{ P_i\}_{i=1}^N$ of $P$ is a \emph{diffeomorphic dyadic partition} if there is a diffeomorphism $\Theta$ defined in a neighborhood of $P$ and such that $\{\Theta(P_i)\}_{i=1}^N$ forms a dyadic partition of $\mathscr{Q}$. More generally, a partition $\mathscr P$ of $P$ is \emph{diffeomorphic} to a partition $\mathscr{P}'$ of $P'$ if  $\mathscr{P}'=\{\Theta(P_i)~~:~~P_i\in \mathscr{P}\}$ for some diffeomorphism $\Theta$ defined in a neighborhood of $P$.

The space of homeomorphisms of the unit cube $\mathscr{Q}$ is equipped with the {\em uniform metric}
$$
d(\Phi,\Psi)=\sup_{x\in\mathscr{Q}} |\Phi(x)-\Psi(x)|+\sup_{x\in \mathscr{Q}}|\Phi^{-1}(x)-\Psi^{-1}(x)|,
$$
see Section~\ref{S26} for more details.

Lebesgue measure of a set $E\subset\bbbr^n$ will be denoted by $|E|$. If $E\subset\bbbr^n$ is measurable, we say that $x\in\bbbr^n$ is a {\em density point} of $E$ if $|B(x,r)\cap E|/|B(x,r)|\to 1$ as $r\to 0^+$. According to the Lebesgue differentiation theorem, almost all points $x\in E$ are density points of $E$.

We say that a mapping $f:\Omega\to\bbbr^n$, $\Omega\subset\bbbr^n$, satisfies the {\em Lusin (N) condition} if it maps sets of Lebesgue measure zero to sets of Lebesgue measure zero.

\section{Preliminaries for Theorem \ref{T31}}
\label{S1}

\subsection{Linear algebra}
\label{S2}

\begin{lemma}
\label{T16}
If $A\in GL(n)$ and $\Vert A-B\Vert<\Vert A^{-1}\Vert^{-1}$, then $B\in GL(n)$.
\end{lemma}
\begin{proof}
Under the given assumptions $\Vert A^{-1}B-\mathcal{I}\Vert<1$ and hence $A^{-1}B$ is invertible (the inverse can be written as an absolutely convergent power series).
\end{proof}
\begin{lemma}
\label{T15}
Assume $T \in GL(n)^+$. Then for any $\eps > 0$, there exists a~finite family of matrices $\{A_i\}_{i=1}^{M_\eps} \subset GL(n)^+$ such that $\|A_i - \mathcal{I}\| < \eps$ for each $i = 1, \dots, M_\eps$, and
\begin{equation*}
T = A_1 \cdot \ldots \cdot A_{M_\eps}.
\end{equation*}
\end{lemma}
\begin{proof}
$GL(n)^+$ is a Lie group with respect to the matrix multiplication. According to \cite[Theorem~3.68]{warner}, $GL(n)^+$ is connected. Then the result follows from \cite[Proposition~3.18]{warner} which says that if $U$ is a neighborhood of the identity element in a connected Lie group $G$, then any element $g\in G$ can be represented as $g=u_1\cdot\ldots\cdot u_k$ for some $k$ and ${u_1,\ldots,u_k\in U}$.
\end{proof}

\subsection{Local to global homeomorphisms}
Recall that a map is {\em proper} if preimages of compact sets are compact.
A {\em local homeomorphism} $f:X\to Y$ is a map that is a homeomorphism in a neighborhood of each point $x\in X$.
The following result is due to Ho \cite{Ho81,Ho75}, see also \cite[Chapter~4, Section~2.4]{prasolov}.
\begin{lemma}
\label{T13}
Suppose that $X$ and $Y$ are path-connected Hausdorff spaces, where $Y$ is simply connected. Then a local homeomorphism $f:X\to Y$ is a global homeomorphism of $X$ onto $Y$ if and only if $f$ is a proper map.
\end{lemma}
Note that in general, a local homeomorphism need not be surjective, and surjectivity is a part of the lemma. The main idea of the proof is to show that a proper local homeomorphism between path connected Hausdorff spaces is a covering map and then the result follows from general facts about covering spaces, see also \cite[Lemma~3.1]{KauranenLT}.

If $f:\partial\bbbb^n\to\bbbr^n$ is an embedding, then according to the Jordan-Brouwer separation theorem $f(\partial\bbbb^n)$ separates $\bbbr^n$ into two domains, bounded and unbounded, and $f(\partial\bbbb^n)$ is their common boundary. The example of the inward Alexander horned sphere shows that in general, the bounded component of $\bbbr^n\setminus f(\partial\bbbb^n)$ need not be simply connected.

The next result is a version of Corollary~8.2 from \cite{munkres60}, but our proof is different and more elementary.
\begin{lemma}
\label{T14}
Let $f: \overbar{\bbbb}^n \to \bbbr^n$ be a continuous function such that $f|_{\partial\bbbb^n}$ is one-to-one, $f|_{\bbbb^n}$ is a local homeomorphism, and the bounded component of  $\bbbr^n\setminus f(\partial\bbbb^n)$  is simply connected. Then $f: \overbar{\bbbb}^n \to \bbbr^n$ is a homeomorphism of $\overbar{\bbbb}^n$ onto $f(\overbar{\bbbb}^n)$.
\end{lemma}
\begin{remark}
The above result is true even if we do not assume that the bounded component of $\bbbr^n\setminus f(\partial\bbbb^n)$ is simply connected, but the proof requires the theory of Eilenberg-MacLane spaces from algebraic topology, see \cite[Theorem~1.2]{KauranenLT}.
\end{remark}
\begin{proof}
Denote the bounded and the unbounded components of $\bbbr^n\setminus f(\partial\bbbb^n)$ by $D$ and $U$, respectively. According to our assumptions, $D$ is simply connected.

Since $f(\overbar{\bbbb}^n)$ is compact, it follows that $\partial f(\overbar{\bbbb}^n)\subset f(\overbar{\bbbb}^n)$. On the other hand, $f(\bbbb^n)$ is open (because $f$ is a local homeomorphism in $\bbbb^n)$ and hence $f(\bbbb^n)\cap \partial f(\overbar{\bbbb}^n)=\varnothing$, so $\partial f(\overbar{\bbbb}^n)\subset f(\partial\bbbb^n)=\partial U=\partial D$.

We claim that $f(\overbar{\bbbb}^n)\subset\overbar{D}$. Suppose to the contrary that $f(x)\in U$ for some $x\in\overbar{\bbbb}^n$. Since $U$ is unbounded and connected, there is a curve connecting $f(x)$ to infinity inside $U$. Since $f(\overbar{\bbbb}^n)$ is bounded, the curve must intersect with the boundary of that set and hence a point in $U$ belongs to  $\partial f(\overbar{\bbbb}^n)$, which is a contradiction, because $\partial f(\overbar{\bbbb}^n)\subset \partial U$.

Since $f(\bbbb^n)\subset\overbar{D}$ is an open subset of $\bbbr^n$, it follows that $f(\bbbb^n)\subset D$.
We claim that the mapping $f:\bbbb^n\to D$ is proper. Indeed, if $K\subset D$ is compact, then $K$ is a closed subset of $\bbbr^n$, so $f^{-1}(K)\cap\overbar{\bbbb}^n$ is closed and hence compact. On the other hand, $f^{-1}(K)\cap\partial\bbbb^n=\varnothing$, because $f(\partial\bbbb^n)\cap D=\varnothing$, so $f^{-1}(K)$ is a compact subset of $\bbbb^n$. This proves that $f:\bbbb^n\to D$ is proper. Now, Lemma~\ref{T13} yields that $f:\bbbb^n\to D$ is a homeomorphism onto $D$. That also implies that $f$ is one-to-one on $\overbar{\bbbb}^n$ because $f(\partial\bbbb^n)\cap D=\varnothing$. Since $\overbar{\bbbb}^n$ is compact, it follows that $f:\overbar{\bbbb}^n\to\bbbr^n$ is a homeomorphism onto its image.
\end{proof}

\subsection{Gluing homeomorphisms together}

Throughout the paper we repeatedly use the following observation (c.f. \cite[Lemma 3.7]{GoldsteinHajlasz17}) and its corollary.
\begin{lemma}
\label{lempg1}
Let $\Omega \subset \bbbr^n$ be a~bounded domain and let $F, G: \Omega \to \bbbr^n$ be homeomorphisms onto their respective images. Assume moreover that $F = G$ near $\partial \Omega$. Then $F(\Omega) = G(\Omega)$.
\end{lemma}
\begin{proof}
Assume otherwise; without loss of generality we may assume that there is $y\in\Omega$ such that $G(y)\not \in F(\Omega)$. We can find subdomains $\Omega_1 \Subset \Omega_2 \Subset \Omega$ such that $F = G$ on $\Omega \setminus \Omega_1$. This means that $y$ needs to lie in $\Omega_1 \Subset \Omega_2$ and $G(y) \notin F(\Omega_2)$. However, as $F = G$ on $\Omega_2 \setminus \Omega_1$, there is an $x \in \Omega_2$ with $G(x) = F(x)$.

Since $\Omega_2$ is path-connected, we may connect $x$ and $y$ with a~curve $\gamma$ lying entirely in $\Omega_2$. Recall that by a well known consequence of Brouwer's theorem on invariance of domain, homeomorphisms $F$ and $G$ map interior points to interior points and boundary points to boundary points. The curve $G(\gamma)$ connects $G(x) = F(x)$, which is an interior point of $F(\Omega_2)$ with $G(y)$, which lies outside $F(\Omega_2)$ and thus $G(\gamma)$ must intersect $\partial F(\Omega_2)=F(\partial\Omega_2)$. Since $F = G$ near $\partial \Omega_2$, $G(\gamma)$ intersects $G(\partial \Omega_2)$. This leads to a~contradiction,  because $\gamma \subset \Omega_2$.
\end{proof}

\begin{corollary}
\label{T45}
Assume $\Omega_2\subset\Omega_1\subset\mathbb{R}^n$ are bounded domains and let $F:\Omega_1\to\mathbb{R}^n$, $G:\Omega_2\to\mathbb{R}^n$ be homeomorphisms onto their respective images. Assume moreover that  for all $x\in\Omega_2$ in some neighborhood of $\partial\Omega_2$ we have $F(x)=G(x)$. Then $\widetilde{F}:\Omega_1\to \mathbb{R}^n$ given by
$$
\widetilde{F}(x)=
\begin{cases}
F(x)&\text{ for }x\in \Omega_1\setminus \Omega_2,\\
G(x)&\text{ for }x\in\Omega_2 \end{cases}
$$
is a homeomorphism and $F(\Omega_1)=\widetilde{F}(\Omega_1)$.
\end{corollary}
\begin{proof}
$\widetilde{F}$ is a local homeomorphism since $F$ and $G$ are homeomorphisms and $F=G$ near $\partial\Omega_2$. Since by
Lemma~\ref{lempg1} $\widetilde{F}$ is injective and $F(\Omega_1) = \widetilde{F}(\Omega_1)$, it follows that $\widetilde{F}$ is a homeomorphism of $\Omega_1$ onto $F(\Omega_1)$.
\end{proof}

\subsection{Gluing diffeomorphisms together}
\label{S23}

Lemma~\ref{T17} below is similar to Lemma~3.8 in \cite{GoldsteinHajlasz17}. However, the proof in \cite{GoldsteinHajlasz17} required the diffeomorphism to be at least of class $C^2$. We managed to prove the results for $C^1$-diffeomorphisms by using a topological argument due to Munkres \cite{munkres60} (Lemma~\ref{T14} above). He used it in a similar context.

\begin{lemma}
\label{T17}
Suppose that $\Phi:U\to \bbbr^n$ is a $C^k$-diffeomorphism, $k\in\bbbn\cup\{\infty\}$, defined on an open set $U\subset\bbbr^n$ and $\lambda\in(0,1)$ is given. Then for any $x_o\in U$ there is $r_{x_o}>0$ such that $B(x_o,r_{x_o})\Subset U$ and that for any $r\in(0,r_{x_o}]$ it is possible find diffeomorphisms $H_1,H_2:U\to\bbbr^n$ of class $C^k$ satisfying
\begin{equation}
\label{eq24}
H_1(x)=
\begin{cases}
\Phi(x_o)+D\Phi(x_o)(x-x_o) &\text{ for }x\in \overbar{B}(x_o,\lambda r),\\
\Phi(x) &\text{ for }x\in U\setminus B(x_o,r).
\end{cases}
\end{equation}
\begin{equation}
\label{eq55}
H_2(x)=
\begin{cases}
\Phi(x) &\text{ for }x\in \overbar{B}(x_o,\lambda r),\\
\Phi(x_o)+D\Phi(x_o)(x-x_o) &\text{ for }x\in U\setminus B(x_o,r).
\end{cases}
\end{equation}
\end{lemma}
\begin{proof}
Let $S(x):=\Phi(x_o)+D\Phi(x_o)(x-x_o)$ and let $\beta=\|(D\Phi(x_o))^{-1}\|^{-1}$. Since $\Phi$ is continuously differentiable, there is $r_{x_o}>0$ such that $B(x_o,r_{x_o})\Subset U$,
$$
\Vert D\Phi(x_o)-D\Phi(x)\Vert<\frac{\beta}{4}
\quad
\text{and}
\quad
\frac{|\Phi(x)-S(x)|}{|x-x_o|}<\frac{(1-\lambda)\beta}{4}
\quad
\text{for } x\in \overbar{B}(x_o,r_{x_o}).
$$
For $r\in (0,r_{x_o}]$, let $\eta\in C^\infty(\bbbr^n)$, $0\leq\eta\leq 1$, be a cut-off function such that
$$
\eta=
\begin{cases}
1 &\text{ on } B(x_o,\lambda r),\\
0&\text{ on } \bbbr^n\setminus B(x_o,r),
\end{cases}
\quad
\text{and}
\quad
\Vert D\eta \Vert_\infty\leq \frac{2}{r(1-\lambda)}\, .
$$
Here $\Vert D\eta\Vert_\infty$ stands for the supremum norm of $|D\eta|$.

Now, we define
$$
H_1(x):=\Phi(x)+\eta(x) (S(x)-\Phi(x))
\quad
\text{and}
\quad
H_2(x):=S(x)+\eta(x) (\Phi(x)-S(x)).
$$
Clearly, $H_{1,2}\in C^k(U;\bbbr^n)$ satisfy \eqref{eq24} and \eqref{eq55}. It remains to show that $H_{1,2}$ are diffeomorphisms. To this end, it suffices to show that $DH_{1,2}(x)$ is invertible for all $x\in U$ and that $H_{1,2}$ are homeomorphisms.

The matrices $DH_{1,2}(x)$ are invertible for $x \in U \setminus \overbar{B}(x_o, r)$, because $H_{1,2}$ are diffeomorphisms in $U \setminus \overbar{B}(x_o, r)$.
To show invertibility of $DH_{1,2}(x)$ for $x\in \overbar{B}(x_o,r)$, it suffices to show that (see Lemma~\ref{T16}):
\begin{equation}
\label{eq56}
\Vert DH_{1,2}(x)-D\Phi(x_o)\Vert<\beta
\quad
\text{for } x\in \overbar{B}(x_o,r).
\end{equation}
Note that $|\eta|\leq 1$ and hence \eqref{eq54} yields
$
\Vert D(\eta(S-\Phi))\Vert\leq |D\eta|\, |S-\Phi|+\Vert DS-D\Phi\Vert.
$
Bearing in mind that $DS(x)=D\Phi(x_o)$, we have
\begin{equation}
\label{eq57}
\begin{split}
&\Vert DH_1(x)-D\Phi(x_o)\Vert \\
&\leq \Vert D\Phi(x)-D\Phi(x_o)\Vert +|D\eta(x)|\, |S(x)-\Phi(x)| + \Vert DS(x)-D\Phi(x)\Vert    \\
&=
|D\eta(x)|\, |\Phi(x)-S(x)|+2 \Vert D\Phi(x)-D\Phi(x_o)\Vert.
\end{split}
\end{equation}
Similarly,
\begin{equation}
\label{eq58}
\Vert DH_2(x)-D\Phi(x_o)\Vert
\leq
|D\eta(x)|\, |\Phi(x)-S(x)|+ \Vert D\Phi(x)-D\Phi(x_o)\Vert.
\end{equation}
For $x\in \overbar{B}(x_o,r)$ we have $|x-x_o|\leq r$ and hence
$$
|D\eta(x)|\, |\Phi(x)-S(x)| +2\Vert D\Phi(x)-D\Phi(x_o)\Vert
\leq
\frac{2}{1-\lambda}\frac{|\Phi(x)-S(x)|}{|x-x_o|} +
2\cdot\frac{\beta}{4}<\beta,
$$
which together with \eqref{eq57} and \eqref{eq58} proves \eqref{eq56}.

It remains to show that $H_{1,2}:U\to\bbbr^n$ are homeomorphisms of $U$ onto their respective images.

Since the surfaces $H_1(\partial B(x_o,r))=\Phi(\partial B(x_o,r))=\partial\Phi(B(x_o,r))$ and $H_2(\partial B(x_o,r))=\partial S(B(x_o,r))$ bound simply connected domains $\Phi({B}(x_o,r))$ and $S({B}(x_o,r))$, and $H_{1,2}$ are local homeomorphisms (because $\det DH_{1,2}\neq 0$), it follows from Lemma~\ref{T14} that $H_1$ and $H_2$ are homeomorphisms of $B(x_o,r)$ onto $\Phi(B(x_o,r))$ and $S(B(x_o,r))$, respectively, and hence $H_1$ and $H_2$ are homeomorphisms of $U$ onto $\Phi(U)$ and $S(U)$, respectively.
The proof is complete.
\end{proof}
The next result shows how to connect linear maps in $GL(n)^+$ in a diffeomorphic way.
\begin{proposition}
\label{T21}
Fix $r>0$ and $\theta\in(0,1)$, and let $A_1, A_2 \in GL(n)^+$. If
$$
A_1 \left(B(0,\theta r) \right) \Subset A_2 \left(B(0, r) \right),
$$
then there exists a $C^\infty$-diffeomorphism $H:\bbbr^n\to\bbbr^n$ which coincides with $x \mapsto A_1 x$ on $B(0, \theta r)$ and with $x \mapsto A_2 x$ on $\bbbr^n \setminus B(0, r)$.
\end{proposition}
In the proof we will need the following special case of the result (c.f.\ \cite[Lemma 8.1]{munkres60}):
\begin{lemma}
\label{T20}
Let $A\in GL(n)^+$. Then for any $r>0$ there is $\varrho \in (0, r)$ and a $C^\infty$-diffeomorphism $H:\bbbr^n\to\bbbr^n$ such that
\begin{equation}
\label{eq26}
H(x)=
\begin{cases}
Ax & \text{for } x\in B(0,\varrho),\\
x  & \text{for } x\in\bbbr^n\setminus B(0,r).
\end{cases}
\end{equation}
\end{lemma}
\begin{proof}
Let $\eta\in C_c^\infty(B(0,1))$, $\eta=1$ on $B(1,1/2)$, and let $M:=\Vert\eta\Vert_\infty+\Vert D\eta\Vert_\infty$. For $r>0$ and
$L\in M^{n\times n}$ define
$f(x):=\eta(x/r)Lx$. We have (see \eqref{eq54}):
$$
Df(x)=\frac{1}{r}(Lx)\otimes D\eta\left(\frac{x}{r}\right) +\eta\left(\frac{x}{r}\right)L,
$$
and hence
$$
|Df(x)|\leq\frac{1}{r}|Lx|\left|D\eta\left(\frac{x}{r}\right)\right| +\left|\eta\left(\frac{x}{r}\right)\right|\, |L|
\leq \frac{1}{r}|L|\, |x| M\chi_{B(0,r)}(x)+M|L|
\leq 2M|L|.
$$
In particular,
$$
|f(x)-f(y)|\leq |x-y|\int_0^1|Df(y+t(x-y))|\, dt
\leq 2M|L|\, |x-y|.
$$
First we will prove the lemma under the assumption that $|A-\mathcal{I}|<(2M)^{-1}$. Let $L:=A-\mathcal{I}$ and define
$$
H_A(x):=x+\eta\left(\frac{x}{r}\right)(A-\mathcal{I})x=x+f(x).
$$
We have
$$
\Vert DH_A(x)-\mathcal{I}\Vert\leq |DH_A(x)-\mathcal{I}|=|Df(x)|\leq 2M|A-\mathcal{I}|<1,
$$
so $DH_A(x)$ is invertible by Lemma~\ref{T16}. $H_A$ is also one-to-one, because for $x\neq y$ we have
$$
|H_A(x)-H_A(y)|\geq |x-y|-|f(x)-f(y)|\geq |x-y|-2M|A-\mathcal{I}|\, |x-y|>0.
$$
Therefore, $H_A$ is a diffeomorphism and \eqref{eq26} is true with $\varrho=r/2$.

Now assume that $A\in GL(n)^+$ is an arbitrary matrix. According to Lemma~\ref{T15}, we can write
$
A=A_1\cdot\ldots\cdot A_k,
$
where
$
|A_i-\mathcal{I}|<(2M)^{-1}.
$
Then
$$
H:=H_{A_1}\circ\ldots\circ H_{A_k}:\bbbr^n\to\bbbr^n
$$
is a diffeomorphism that satisfies $H(x)=x$ for $|x|\geq r$. Using simple induction and the fact that the diffeomorphisms $H_{A_i}$ satisfy \eqref{eq26} with $A_i$ and $\varrho=r/2$, one can easily check that
$$
H(x)=A_1(A_2(\ldots(A_k(x)\ldots))=Ax
$$
if $x$ belongs to the set
$$
B(0,\tfrac{r}{2})\cap A_k^{-1}(B(0,\tfrac{r}{2}))\cap
(A_{k-1}\circ A_k)^{-1}(B(0,\tfrac{r}{2}))\cap\ldots\cap
(A_{2}\circ\ldots\circ A_k)^{-1}(B(0,\tfrac{r}{2})).
$$
Since this set is an open neighborhood of $0$, it contains a ball $B(0,\varrho)$ for some $0<\varrho<r$ and hence $H$ satisfies \eqref{eq26}.
\end{proof}
\begin{proof}[Proof of Proposition~\ref{T21}]
By composing with $A_2^{-1}$ and scaling if necessary,
we can assume that $A_2 = \id$ and $r = 1$. Using Lemma~\ref{T20} we can construct a diffeomorphism $h_2$ of $\bbbr^n$ such that $h_2(x) = A_1 x$ on $B(0, \varrho)$ for some $\varrho\in (0,1)$ and $h_2(x)=x$ on $\bbbr^n \setminus B(0, 1)$.
This diffeomorphism would have desired properties if we could take $\varrho=\theta$, but it may happen that $\varrho$ is smaller than $\theta$.
Thus assume that $\varrho<\theta$.
To correct $h_2$ we take a~radial diffeomorphism $h_1$ of $\bbbr^n$, which equals $h_1(x)=\varrho\,\theta^{-1} x$ on $B(0,\theta)$ and is identity outside $B(0, 1)$. As a~result, $h_2(h_1(x)) = \varrho\,\theta^{-1}A_1 x$ on $B(0, \theta)$ and $h_2(h_1(x)) = x$ on $\bbbr^n \setminus B(0,1)$.
This diffeomorphism has all the required properties, except that it equals $\varrho\,\theta^{-1} A_1 x$ on $B(0,\theta)$ instead of required $A_1x$. To correct it, we
take a~radial diffeomorphism $h_3$, which equals $h_3(x)=\theta\,\varrho^{-1} x$ on $A_1(B(0, \varrho))$ and $h_3(x)=x$ on $\bbbr^n\setminus B(0,1)$. Such a~diffeomorphism exists, because
$A_1(B(0,\varrho))\Subset B(0,1)$ and
$h_3(A_1(B(0, \varrho)))=A_1(B(0,\theta))\Subset B(0,1)$.
The map $H = h_3 \circ h_2 \circ h_1$ is the desired diffeomorphism.
\end{proof}

For the proof of Proposition~\ref{T38} we need two technical lemmata, Lemma~\ref{T37} and Lemma~\ref{T39}.

\begin{lemma}[Storage lemma]
\label{T37}
Fix $\ell \in \mathbb{Z}$. Assume that $\mathscr{V}$ is a~finite family of closed cubes $V \subset \bbbr^n$ with pairwise disjoint interiors, and each cube $V \in \mathscr{V}$ has the same side-length $\mathsf{L}(V) = 2^{-\ell}$. Assume that $\mathscr{W}$ is another finite family of closed cubes $W \subset \bbbr^n$ such that
\begin{equation}
\label{eq60}
\sum_{W \in \mathscr{W}} |W| \leq \sum_{V \in \mathscr{V}} |V|
\end{equation}
and that for each $W \in \mathscr{W}$, $\mathsf{L}(W) = 2^{-k}$ for some $k \in \mathbb{Z}$, $k \geq \ell$.
Then, for each $W \in \mathscr{W}$, there is an isometric closed cube $\widetilde{W}$ i.\,e., $\mathsf{L}(\widetilde{W}) = \mathsf{L}(W)$, such that the cubes $\{ \widetilde{W}\}_{W \in \mathscr{W}}$ have pairwise disjoint interiors, and
$$
\bigcup \nolimits_{W \in \mathscr{W}} \widetilde{W} \subset \bigcup \nolimits_{V \in \mathscr{V}} V.
$$

\end{lemma}
\begin{remark}
The lemma has a practical interpretation. You can place dyadic boxes $W\in\mathscr{W}$ in the storage containers $V\in\mathscr{V}$ (identical and dyadic) if and only if the total volume of the boxes $W$ does not exceed the total volume of the storage, and no box $W$ is larger than a storage container.
\end{remark}
\begin{proof}
Let $\mathscr{V}_\ell := \mathscr{V}$. Divide the family $\mathscr{W}$ into subfamilies according to the side-length: $\mathscr{W} = \bigcup_{i=\ell}^N \mathscr{W}_i$, where $\mathscr{W}_i = \{ W \in \mathscr{W}: \, \mathsf{L}(W) = 2^{-i}\}$. Clearly, \eqref{eq60} can be rewritten as
\begin{equation}
\label{eq61}
\sum_{i=\ell}^N \sum_{W \in \mathscr{W}_i} |W| \leq \sum_{V \in \mathscr{V}_\ell} |V|.
\end{equation}
It follows that the number of cubes in $\mathscr{W}_\ell$ is less than or equal to the number of cubes in $\mathscr{V}_\ell$. Thus, for each $W \in \mathscr{W}_\ell$, we can find $\widetilde{W} \in \mathscr{V}_\ell$ so that the cubes $\widetilde{W}$ have pairwise disjoint interiors. Clearly, $W$ and $\widetilde{W}$ are isometric.

Divide each of the cubes in the remaining family
\begin{equation}
\label{eq62}
\mathscr{V}_\ell \setminus \{ \widetilde{W}: W \in W_\ell \}
\end{equation}
into $2^n$ dyadic closed cubes of side-length $2^{-(\ell + 1)}$. Denote the resulting family of cubes by $\mathscr{V}_{\ell + 1}$. That is, each of the cubes in $\mathscr{V}_{\ell + 1}$ has side-length $2^{-(\ell + 1)}$ and the number of cubes in $\mathscr{V}_{\ell + 1}$ equals $2^{n}$ times the number of the cubes in \eqref{eq62}. Clearly, \eqref{eq61} implies that
$$
\sum_{i = \ell +1}^N \sum_{W \in \mathscr{W}_i} |W| \leq \sum_{V \in \mathscr{V}_{\ell + 1}} |V|,
$$
because by removing cubes $W \in \mathscr{W}_\ell$ from $\mathscr{W}$ and cubes $\widetilde{W} \in \mathscr{V}_\ell$, from $\mathscr{V}_\ell$, we removed equal volumes from both sides of \eqref{eq61}. Now, we can repeat the procedure described above and match each $W \in \mathscr{W}_{\ell + 1}$ with a~suitable cube $\widetilde{W} \in \mathscr{V}_{\ell + 1}$. We repeat the procedure by induction. The required family $\{ \widetilde{W}\}_{W \in \mathscr{W}}$ will be constructed after a~finite number of steps.
\end{proof}

It is a~known fact that given any two points $p, q$ lying in the interior of a~smooth, connected manifold $M$, one can find a~diffeomorphism of $M$ onto itself that carries $p$ into $q$ and is isotopic to identity, see \cite[Chapter 4]{Milnor}. We need a~slightly stronger folklore result, stating that given a~finite family of points in a~Euclidean domain, we can rearrange them in a~diffeomorphic manner so that neighborhoods of these points are mapped by translation.
\begin{lemma}
\label{T39}
Let $\{p_i\}_{i=1}^N$ and $\{q_i\}_{i=1}^N$ be given points in $U$, a domain in $\bbbr^n$, with $p_i \neq p_j$ and $q_i \neq q_j$ for $i \neq j$. Then, there exists an~$\eps > 0$ and a $C^\infty$-diffeomorphism $H: U \to U$, identity near the boundary, such that
$$
H(x) = x + (q_i - p_i) \text{ for } x \in B(p_i, \eps),
$$
i.e., $H$ maps by translation each ball $B(p_i, \eps)$ onto $B(q_i, \eps)$ with $H(p_i) = q_i$.
\end{lemma}
\begin{proof}[Proof]
Firstly, let us consider the case $N=1$ and assume for simplicity that $p:=p_1$ and $q:=q_1$ can be connected by a~segment $\gamma$ contained in $U$. We choose $\eps >0$ so that the $2\eps$-neighborhood of $\gamma$ is contained in $U$ and
find a~smooth vector field $X$, satisfying
$$
X = \begin{cases}
    0 & \text{ on the set } \{x: \, \mathrm{dist}(x, \gamma) \geq 2\eps \},\\
    q - p & \text{ on the set } \{x: \, \mathrm{dist}(x, \gamma) \leq \eps\}.
\end{cases}
$$
If $\Phi_t:U\to U$ is the one-parameter family of diffeomorphisms generated by $X$, then $H_{p,q}^\eps := \Phi_1$ is a~diffeomorphism which acts as the translation by $q - p$ on the ball $B(p,\eps)$ and maps it onto the~ball $B(q,\eps)$. Moreover, $H_{p,q}^\eps$ equals identity outside the $2\eps$-neighborhood of $\gamma$.

In view of path-connectedness of Euclidean domains, any two points $p, q$ can be connected with a~piecewise linear curve $\gamma \subset U$ with vertices $a_0=p, a_1, \ldots, a_m=q$.
Choose $\eps>0$ so that the $2\eps$-neighborhood of $\gamma$ is contained in $U$  and apply the construction from previous paragraph to each pair of points $a_i, a_{i+1}$ for $i=0, \ldots, m-1$ to construct diffeomorphisms $H_{a_i, a_{i+1}}^\eps$, identity outside the $2\eps$-neighborhood of $\gamma$, such that
$$
H_{a_i, a_{i+1}}^\eps (x) = x + (a_{i+1} - a_i) \text{ for } x \in B(a_i, \eps).
$$
Then, $H^\eps_{p,q} := H_{a_{m-1}, q}^\eps \circ \ldots \circ H_{p, a_1}^\eps$ is the desired diffeomorphism when $N=1$.

For $N>1$, consider firstly the case when $\{p_i\}_{i=1}^N \cap \{q_i\}_{i=1}^N = \varnothing$. We can then find $N$ distinct piecewise linear curves $\gamma_i$, $i=1, \ldots, N$, connecting $p_i$ with $q_i$, and an $\eps > 0$ so that the $2\eps$-neighborhoods of $\gamma_i$ are pairwise disjoint and contained in $U$ and construct diffeomorphisms $H_{p_i, q_i}^\eps$ from the previous paragraph. Diffeomorphism $H = H_{p_N, q_N}^\eps \circ \ldots \circ H_{p_1, q_1}^\eps$ is the desired map.

If $\{p_i\}_{i=1}^N \cap \{q_i\}_{i=1}^N \neq \varnothing$, then we find a set of distinct points  $\{s_i\}_{i=1}^N\subset U$, such that $\{p_i\}_{i=1}^N\cap\{s_i\}_{i=1}^N=\varnothing$ and
$\{s_i\}_{i=1}^N\cap\{q_i\}_{i=1}^N=\varnothing$. From what we already proved, there is a diffeomorphism $H_1$ that translates neighborhoods of $p_i$'s onto neighborhoods of $s_i$'s and a diffeomorphism $H_2$ that translates neighborhoods of $s_i$'s onto neighborhoods of $q_i$'s. Then $H=H_2\circ H_1$ satisfies the claim of the lemma.
\end{proof}

The next result shows in particular that if $A_1, A_2 \in GL(n)^+$ and $\det A_1=\det A_2$, then it is possible to find a~diffeomorphism of a ball $B$ onto $A_2(B)$ which on a~large part of $B$ acts like a piecewise affine map with $A_1$ as its linear part.
\begin{proposition}
\label{T38}
Let $G \subset B$ be a~measurable subset of an~open ball $B \subset \bbbr^n$ centered at the origin, and let $r: G \to (0, \infty)$ be any function. Let $A_1, A_2 \in GL(n)^+$ satisfy
\begin{equation}
\label{eq46}
\det A_2 > \beta \det A_1 \text{ for some } \beta \in (0,1).
\end{equation}
Then, there is a~finite family of pairwise disjoint closed balls $\overbar{B}(p_j, r_j) \subset B$ such that
\begin{equation}
\label{eq48}
p_j \in G, \qquad r_j < r(p_j),  \qquad \big|G\cap\bigcup \nolimits_j \overbar{B}(p_j, r_j)\big| > \beta |G|,
\end{equation}
and a~diffeomorphism $F: B \to A_2(B)$ which agrees with $A_2$ in a~neighborhood of $\partial B$ and
\begin{equation}
\label{eq49}
F(x) = A_1 x + v_j \text{ for all } x \in \overbar{B}(p_j, r_j) \text{ and some } v_j \in \bbbr^n.
\end{equation}
\end{proposition}
\begin{remark}
\label{R1}
If $B=B(p,R)$ is a ball not necessarily centered at the origin, and $q\in\bbbr^n$, but all other assumptions remain the same, then we can find balls $\overbar{B}(p_j,r_j)\subset B$ satisfying \eqref{eq48} and a diffeomorphism $F:B\to\bbbr^n$ such that $F(x)=A_2(x-p)+q$ in a neighborhood of $\partial B$ and $F$ satisfies \eqref{eq49} in each of the balls $\overbar{B}(p_j,r_j)$.

Indeed, such a diffeomorphism is obtained from Proposition~\ref{T38} by composing with the translations $x\mapsto x-p$ in the domain and $y\mapsto y+q$ in the target. Then in each of the balls $\overbar{B}(p_j,r_j)$ we have
$F(x)=A_1(x-p)+v_j+q=A_1x+w_j$, for some $w_j\in\bbbr^n$.
\end{remark}
\begin{proof}
Assume first that $A_1 = \id$. Let $U = A_2(B)$.

Before proceeding to details, let us describe the main idea of the proof.
We begin by finding a finite family of disjoint closed cubes $Q_i\subset B$ satisfying
$|G\cap\bigcup_i Q_i|>\beta |G|$. Working with cubes allows us to apply Lemma~\ref{T37} and to find translated cubes $Q_i+w_i\subset A_2(B)$ with pairwise disjoint interiors. Then, for each $i$ we find a finite family of pairwise disjoint balls $\overbar{B}(p_{ik},r_{ik})\subset\mathring{Q}_i$. Clearly, the balls $\overbar{B}(p_{ik},r_{ik})+w_i\subset A_2(B)$ are pairwise disjoint. After re-enumeration, we can write
$$
\{\overbar{B}(p_j,r_j)\}_j:=\{\overbar{B}(p_{ik},r_{ik})\}_{i,k},
\quad
v_j:=w_i \ \text{ if } \ p_j=p_{ik}.
$$
Choosing the balls carefully, we can guarantee \eqref{eq48}. Note that the balls $\overbar{B}(p_j,r_j)+v_j\subset A_2(B)$ are pairwise disjoint. Then we construct a diffeomorphism $F$ that equals $A_2$ near $\partial B$ and satisfies $F(x)=x+v_j$ for $x\in\overbar{B}(p_j,r_j)$, which is \eqref{eq49} in the case when $A_1=\id$.

\noindent
\underline{Step 1. Finding cubes.} Since $|U| > \beta |B|$ by \eqref{eq46}, we also have $|U| > \alpha |B|$ for some $\alpha \in (\beta, 1)$.
Let $\mathscr{F}_\ell$ be the family of all closed dyadic cubes (i.\,e., with vertices at points of $2^{-\ell} \mathbb{Z}^n$) of side-length $2^{-\ell}$ that are contained in $U$. Clearly, the family $\mathscr{F}_\ell$ is finite. We choose $\ell$ large enough to guarantee that
\begin{equation}
\label{eq63}
\sum_{Q \in \mathscr{F}_\ell} |Q| > \frac{\beta}{\alpha} |U|.
\end{equation}
Let $V\subset\bbbr^n$ be an open set such that
$$
G\subset V\subset B
\quad
\text{and}
\quad
|V\setminus G|<\frac{1}{2}\beta(\alpha^{-1}|U|-|G|)
$$
(note that $\alpha^{-1}|U|-|G|>0$).

For each $q\in G$ consider the family $\mathscr{G}_q$ of all closed cubes $Q = Q(q, 2^{-k})$, $k \in \mathbb{Z}$, centered at $q$, that satisfy
\begin{equation}
\label{eq64}
Q \subset V, \quad \mathsf{L}(Q) \leq 2^{-\ell}, \quad |Q| < \frac{1}{2}\beta (\alpha^{-1} |U| - |G|).
\end{equation}
The family
$
\widetilde{\mathscr{G}} := \bigcup \nolimits_{q \in G} \mathscr{G}_q
$
is a~Vitali covering of $G$ and Vitali's covering theorem yields a~finite sub-family $\mathscr{G}'=\{Q_i\}_{i=1}^{N'} \subset \widetilde{\mathscr{G}}$ of pairwise disjoint cubes such that
\begin{equation}
\label{eq65}
\sum_{i=1}^{N'}|Q_i|\geq \sum_{i=1}^{N'}|Q_i\cap G| > \beta |G|.
\end{equation}
By removing some cubes from the family $\mathscr{G}'$, we can obtain a family $\mathscr{G}=\{Q_i\}_{i=1}^N$ (so $N\leq N'$) such that
\begin{equation}
\label{eq68}
\frac{\beta}{\alpha}|U|\geq \sum_{i=1}^{N}|Q_i|\geq \sum_{i=1}^{N}|Q_i\cap G| > \beta |G|.
\end{equation}
Indeed, to show this, it suffices to prove for any $m$ the implication
\begin{equation}
\label{eq69}
\sum_{i=1}^{m+1}|Q_i|>\frac{\beta}{\alpha}|U|
\quad
\Longrightarrow
\quad
\sum_{i=1}^m |Q_i\cap G|>\beta |G|.
\end{equation}
Note that
$$
\sum_{i=1}^m|Q_i|-\sum_{i=1}^m |Q_i\cap G|
=\sum_{i=1}^m |Q_i\setminus G|\leq |V\setminus G|
<\frac{1}{2}\beta(\alpha^{-1}|U|-|G|).
$$
On the other hand, the hypothesis in \eqref{eq69} and the upper estimate for $|Q|$ in \eqref{eq64} yield
$$
\sum_{i=1}^m |Q_i|>\frac{\beta}{\alpha}|U|-|Q_{m+1}|> \frac{\beta}{\alpha}|U| - \frac{1}{2}\beta(\alpha^{-1}|U| - |G|) =
\beta |G|+\frac{1}{2}\beta(\alpha^{-1}|U|-|G|),
$$
and hence
$$
\sum_{i=1}^m|Q_i\cap G|=
\sum_{i=1}^m |Q_i|-\Big(\sum_{i=1}^m |Q_i|-\sum_{i=1}^m|Q_i\cap G|\Big)>\beta |G|.
$$
This proves the implication \eqref{eq69} and hence proves the existence of $\mathscr{G}=\{Q_i\}_{i=1}^N$ satisfying \eqref{eq68}.
Now, \eqref{eq63} and \eqref{eq68} yield
$$
\sum_{Q\in\mathscr{F}_\ell} |Q|\geq \sum_{i=1}^N |Q_i|.
$$
Since $\mathsf{L}(Q_i)=2^{-k}$, $k\geq\ell$, it follows from Lemma~\ref{T37}
that there are vectors $w_i \in \bbbr^n$ such that the cubes $\widetilde{Q}_i := Q_i + w_i$ have pairwise disjoint interiors and
$$
\bigcup_{i=1}^N \widetilde{Q}_i \subset \bigcup \nolimits_{Q \in \mathscr{F}_\ell} Q \subset U.
$$
The cubes $Q_i$ are pairwise disjoint, but the cubes $\widetilde{Q}_i$ need not be.

To summarize, we constructed a family of pairwise disjoint closed cubes $\{Q_i\}_{i=1}^N$, $Q_i\subset B$, such that
\begin{equation}
\label{eq75}
\left|G\cap\bigcup\nolimits_{i} \mathring{Q}_i\right|=\left|G\cap\bigcup\nolimits_{i}  {Q}_i\right|>\beta|G|,
\end{equation}
and we constructed vectors $\{w_i\}_{i=1}^N$ such that the cubes $\widetilde{Q}_i=Q_i+w_i\subset U=A_2(B)$ have pairwise disjoint interiors.

\noindent
\underline{Step 2. Finding balls.} We will now find a finite family of closed, pairwise disjoint balls $\overbar{B}(p_j, r_j)$, satisfying \eqref{eq48}. For each $i=1,2,\ldots,N$, let
$$
\mathscr{B}_i = \left\{ \overbar{B}(p,r):\, p \in G\cap \mathring{Q}_i,\
\overbar{B}(p,r)\subset\mathring{Q}_i,\ r<r(p)\right\}.
$$
Vitali's covering theorem and \eqref{eq75} yield a~finite sub-family of disjoint closed balls $\overbar{B}(p_{ik}, r_{ik})$, $k=1,\ldots,N_i$, so that
$$
p_{ik}\in G,
\qquad
r_{ik}<r(p_{ik}),
\qquad
\overbar{B}(p_{ik}, r_{ik}) \subset \mathring{Q}_i,
$$
and
$$
\Big| G\cap\bigcup\nolimits_{ik} \overbar{B}(p_{ik},r_{ik})\Big| > \beta |G|.
$$
For each $i$, the balls $\overbar{B}(p_{ik}, r_{ik}) + w_i$ are pairwise disjoint and contained in the interior of $\widetilde{Q}_i$. Since the interiors of the cubes $\widetilde{Q}_i$ are pairwise disjoint, the balls in the family $\{\overbar{B}(p_{ik}, r_{ik}) + w_i\}_{i,k}$ are pairwise disjoint as well. After re-enumerating, we get a family of balls $\overbar{B}_j := \overbar{B}(p_j, r_j)$, $j=1,2,\ldots,M$
that satisfy \eqref{eq48} and vectors $v_j\in\bbbr^n$ such that the balls
$\overbar{B}'_j := \overbar{B}(p_j, r_j) + v_j\subset U = A_2(B)$ are pairwise disjoint.

\noindent
\underline{Step 3. Finding the diffeomorphism $F$.} To complete the proof in the case $A_1=\id$, it remains to construct a diffeomorphism $F:B\to\bbbr^n$ which equals $A_2$ near $\partial B$ and
$F(x)=x+v_j$ for $x\in \overline{B}_j$ for $j=1,2,\ldots,M$.

Fix $R_1>0$ such that $B_{R_1}:=B(0,R_1)\Subset B\cap A_2(B)$. Let $B_{R_2}:=B(0,R_2)$ be a ball such that
$$
\bigcup_{j=1}^M\overbar{B}_j\subset B_{R_2}\Subset B.
$$
Let $H_1:\bbbr^n\to\bbbr^n$ be a radial diffeomorphism such that $H_1(x)=x$ near $\partial B$ and $H_1(x)=R_1R_2^{-1}x$ on $B_{R_2}$.
Proposition~\ref{T21} yields a diffeomorphism $H_2:\bbbr^n\to\bbbr^n$ such that $H_2=A_2$ near $\partial B$ and $H_2=\id$ on $\overbar{B}_{R_1}$. Then, $H_2\circ H_1:B\to A_2(B)$ equals $A_2$ near $\partial B$ and maps the balls $\overbar{B}_j$ by scaling (with factor $R_1R_2^{-1}$) onto balls in $B_{R_1}\Subset A_2(B)$.

Let $a_j:=H_2(H_1(p_j))$ be the centers of the balls $H_2(H_1(\overbar{B}_j))$ and let $b_j:=p_j+v_j$ be the centers of the balls $\overbar{B}_j'=\overbar{B}_i+v_j$. Both families $\{ a_j\}_{j=1}^M$ and
$\{b_j\}_{j=1}^M$ are contained in $U=A_2(B)$, and Lemma~\ref{T39} gives $\eps>0$ and a diffeomorphism $\Theta:U\to U$ that equals identity near $\partial U$ and satisfies
$$
\Theta(x) = x + b_j - a_j \text{ for } x \in B(a_j, \eps).
$$
Clearly, we can assume that
$$
\eps<\min_j R_1 R_2^{-1}r_j.
$$
Since the balls
\begin{equation}
\label{eq76}
H_2(H_1(\overbar{B}_j))=\overbar{B}(a_j,R_1R_2^{-1}r_j)\subset U
\end{equation}
are pairwise disjoint, there is $\delta>0$ such that the balls
\begin{equation}
\label{eq77}
\overbar{B}(a_j,R_1R_2^{-1}r_j+\delta)\subset B
\end{equation}
are pairwise disjoint.

For each $j=1,2,\ldots,M$, we find a diffeomorphism $H_1^j:\bbbr^n\to\bbbr^n$ that is similar to $H_1$. It is a radial diffeomorphism centered at $a_j$, it is identity outside the ball \eqref{eq77} and it maps the ball \eqref{eq76} onto $\overbar{B}(a_j,\eps)$ by scaling centered at $a_j$ with factor $\eps R_1^{-1}R_2 r_j^{-1}<1$. Clearly, the diffeomorphism
$$
H_3:=H_1^1\circ H_1^2\circ\ldots\circ H_1^M:U\to U
$$
is identity near $\partial U$ and maps each of the balls \eqref{eq76} onto $\overbar{B}(a_j,\eps)$ by scaling (centered at $a_j$). Now
$\Theta\circ H_3\circ H_2\circ H_1:B\to U$ equals $A_2$ near $\partial U$ and maps the balls $B_j=B_j(p_j,r_j)$ onto the balls $\overbar{B}(b_j,\eps)$ by affine maps whose linear part is scaling by factor $\eps r_j^{-1}$.

Finally, if $H_4:\bbbr^n\to\bbbr^n$ is a diffeomorphism similar to $H_3$ that equals identity near $\partial U$ and expands the balls $\overbar{B}(b_j,\eps)$ to $\overbar{B}(b_j,r_j)=\overbar{B}_j'$ by scaling, then the diffeomorphism
$$
F:H_4\circ\Theta\circ H_3\circ H_2\circ H_1:B\to A_2(B)
$$
is the required diffeomorphism satisfying \eqref{eq49} for $A_1 = \id$.

\noindent
\underline{Step 4. The general case.} Finally, suppose that $A_1$ and $A_2$ are arbitrary $GL(n)^+$ mappings satisfying \eqref{eq46}.
Then $\widetilde{A}_1:=\id$ and $\widetilde{A}_2:=A_1^{-1}\circ A_2$ satisfy $\det \widetilde{A}_2>\beta\det\widetilde{A}_1$ and the construction from Step~3 yields a family of balls $\overbar{B}_j$ satisfying \eqref{eq48} and a diffeomorphism $\widetilde{F}:B\to\widetilde{A}_2(B)$ such that
$$
\widetilde{F}(x) =x + \widetilde{v}_j \text{ for all } x \in \overbar{B}_j \text{ and some } \widetilde{v}_j \in \bbbr^n.
$$
Setting $F = A_1 \circ \widetilde{F}$ yields the desired diffeomorphism satisfying \eqref{eq49}.
\end{proof}

\subsection{The Dacorogna-Moser theorem}
The following lemma is a special case of a theorem of Dacorogna and Moser (\cite[Theorem~5]{DacorognaMoser}, see also \cite[Theorem 10.11]{CsatoDacorognaKneuss}), who generalized earlier results of Moser \cite{Moser} and Banyaga~\cite{Banyaga}.
\begin{lemma}
\label{T11}
Let  $\Omega\subset\bbbr^n$ be a bounded domain and let $f\in C^{\infty}(\Omega)$ be a positive function equal $1$ in a neighborhood of $\partial\Omega$ such that
$$
\int_\Omega f(x) \, dx=|\Omega|.
$$
Then there exists a $C^{\infty}$-diffeomorphism $\Psi$ of $\Omega$ onto itself that is identity on a neighborhood of $\partial\Omega$ and satisfies
$$
J_{\Psi}(x)=f(x)
\quad
\text{for all } x\in \Omega.
$$
\end{lemma}
Although the proofs in \cite{CsatoDacorognaKneuss} and \cite{DacorognaMoser} are written only for $f$ and $\Psi\in C^{k}(\Omega)$ for some $k\in\bbbn$, they clearly work for $f$ and $\Psi\in C^{\infty}(\Omega)$; for a proof using Moser's \emph{flow method} with $f,~\Psi\in C^{\infty}$ see e.g. \cite[Appendix, Lemma 2.3]{DacorognaDirect} (the first edition of the book).

The next result shows that if $f$ is only measurable, then on a large set we can uniformly approximate $f$ by the Jacobian of a smooth diffeomorphism.
\begin{lemma}
\label{T12}
Let $\Omega\subset\bbbr^n$ be a bounded domain and let $f\in L^1(\Omega)$, $f>0$ a.e., be such that
$$
\int_\Omega f(x)\, dx =|\Omega|.
$$
Then for any $\eps>0$ there exists a compact set $K\subset\Omega$ with $|\Omega\setminus K|<\eps$, and
a $C^{\infty}$-diffeomorphism $\Psi$ of $\Omega$ onto itself that is identity on a neighborhood of $\partial\Omega$ and satisfies
\begin{equation}
\label{eq21}
|J_{\Psi}(x)-f(x)|<\eps f(x)
\qquad
\text{for all } x\in K.
\end{equation}
\end{lemma}
\begin{proof}
By Lusin's theorem, we can find a compact set $K\subset\Omega$, such that $|\Omega\setminus K|<\eps$ and $f$ is continuous and strictly positive in $K$. Let
$$
m:=\inf_K f
\qquad
\text{and}
\qquad
M:=\sup_K f\,.
$$
Clearly,
$$
\int_K f(x)\, dx=|\Omega|-2M\delta
\quad
\text{for some } \delta>0.
$$
Let $\Omega'$ and $\Omega''$ be open sets such that
$$
K\subset\Omega'\Subset\Omega''\Subset\Omega,
\quad
\text{and}
\quad
|\Omega'\setminus K|<\delta,
$$
We can extend $f$ from $K$ to a continuous function $0\leq f_1\leq M$ that is compactly supported in $\Omega'$, so
$$
\int_\Omega f_1(x)\, dx = \int_{\Omega'\setminus K} f_1(x)\, dx +
\int_{K} f(x)\, dx<
|\Omega'\setminus K|\cdot M+(|\Omega|-2M\delta)<|\Omega|-M\delta.
$$
Using a standard approximation of $f_1$ by convolution, we find $f_2\in C_c^\infty(\Omega')$, $f_2\geq 0$, such that
\begin{equation}
\label{eq20}
|f(x)-f_2(x)|=|f_1(x)-f_2(x)|<\frac{\eps m}{2}
\quad
\text{for all } x\in K.
\end{equation}
Since the approximation by convolution preserves the $L^1$ norm of a non-negative function (by Fubini's theorem), we have
\begin{equation}
\label{eq17}
\int_\Omega f_2(x)\, dx=\int_\Omega f_1(x)\, dx<|\Omega|-M\delta.
\end{equation}
It is easy to see that there is
${f}_3\in C^\infty(\Omega)$ that it is strictly positive in $\Omega$, ${f}_3=1$ in a neighborhood of $\partial\Omega$, ${f}_3<\eps m/2$ in $K$, and
\begin{equation}
\label{eq18}
\int_\Omega {f}_3(x)\, dx<M\delta.
\end{equation}
Integrals in \eqref{eq17} and \eqref{eq18} add to a number less than $|\Omega|$ and we can find a~function $f_4\in C_c^\infty(\Omega''\setminus \overbar{\Omega}')$, $f_4\geq 0$, such that
$$
\int_\Omega f_2(x)+f_3(x)+f_4(x)\, dx = |\Omega|.
$$
Observe that the function $f_\eps:=f_2+f_3+f_4\in C^\infty(\Omega)$ equals $f_3=1$ in a neighborhood of $\partial\Omega$ and equals $f_2+f_3$ in $K$, so
$$
|f(x)-f_\eps(x)|\leq |f(x)-f_2(x)|+f_3(x)<\eps m\leq \eps f(x)
\quad
\text{for all } x\in K.
$$
These conditions and Lemma~\ref{T11} imply the existence of a~smooth diffeomorphism $\Psi:\Omega\to\Omega$ that is identity near the boundary and satisfies $J_\Psi=f_\eps$, so \eqref{eq21} is satisfied.
\end{proof}

\begin{corollary}
\label{T29}
Let $\Omega, \Omega' \subset \bbbr^n$ be bounded domains and let $F: \Omega \to F(\Omega)=\Omega'$ be an orientation preserving diffeomorphism between $\Omega$ and $\Omega'$. Suppose that $f \in L^{1}(\Omega)$, $f > 0$ a.e. and
$$
    \int_\Omega f(x) \, dx = |\Omega'|.
$$
Then for any $\eps>0$ there exists a compact set $K\subset\Omega$ with $|\Omega\setminus K|<\eps$, and
a $C^{\infty}$-diffeomorphism $F'$ of $\Omega$ onto $\Omega'$, that equals $F$ on a neighborhood of $\partial\Omega$ and satisfies
\begin{equation}
\label{eq36}
|J_{F'}(x)-f(x)|<\eps f(x)
\qquad
\text{for all } x\in K.
\end{equation}
\end{corollary}
\begin{remark}
The condition that $F$ is orientation preserving is necessary. Indeed, in view of \eqref{eq36}, $J_{F'}(x) > 0$ for $x \in K$, so $J_{F'}>0$ on $\Omega$, and hence $J_F>0$ on $\Omega$, because $F=F'$ on an open set.
\end{remark}
\begin{proof}
Let
$$
g(y)=\frac{f(F^{-1}(y))}{J_F(F^{-1}{(y))}},
\qquad
\text{so}
\qquad
g(F(x))J_F(x)=f(x).
$$
It follows from the change of variables formula that $g\in L^1(\Omega')$, $g>0$ a.e, and
$$
\int_{\Omega'} g(y)\, dy = \int_{\Omega}f(x)\, dx =|\Omega'|.
$$
Therefore, for any $\eps'>0$, Lemma~\ref{T12} yields a compact set $K'\subset\Omega'$, $|\Omega'\setminus K'|<\eps'$, and a diffeomorphism $G$ of $\Omega'$ onto itself that is identity in a neighborhood of $\partial\Omega'$ and satisfies
\begin{equation}
\label{eq59}
|J_G(y)-g(y)|<\eps'g(y)
\quad
\text{for all } y\in K'.
\end{equation}
Let $K=F^{-1}(K')$.
By taking $\eps'\leq\eps$ sufficiently small, we can guarantee that $|\Omega\setminus K|<\eps$. Now the diffeomorphism $F'=G\circ F:\Omega\to\Omega'$ satisfies the claim of the corollary. Indeed, $F'=F$ near $\partial\Omega$ and \eqref{eq59} yields
$$
|J_{F'}(x)-f(x)|=
|J_G(F(x))J_F(x)-g(F(x))J_F(x)|<
\eps'g(F(x))J_F(x)=\eps'f(x)\leq \eps f(x),
$$
whenever $x\in K$. The proof is complete.
\end{proof}

\subsection{Uniform metric}
\label{S26}
Let us denote by
$$
d(\Phi,\Psi):=\sup_{x\in \mathscr{Q}}|\Phi(x)-\Psi(x)|+\sup_{x\in \mathscr{Q}}|\Phi^{-1}(x)-\Psi^{-1}(x)|
$$
the {\em uniform metric} in the space of homeomorphisms of the unit cube $\mathscr{Q}=[0,1]^n$ onto itself. It is known that the space of homeomorphisms is a complete metric space with respect to the metric $d$. More precisely, we have:
\begin{lemma}
\label{T6}
Let $\Phi_{k}:\mathscr{Q}\to \mathscr{Q}$, $k=1,2,\ldots$ be a Cauchy sequence of surjective homeomorphisms in the uniform metric $d$. Then $\Phi_{k}$ converges uniformly to a homeomorphism $\Phi:\mathscr{Q}\to \mathscr{Q}$, and
$\Phi_{k}^{-1}$ converges uniformly, and the limit is equal to $\Phi^{-1}$.
\end{lemma}
\begin{proof}
Obviously $\Phi_{k}$ and $\Phi_{k}^{-1}$ are Cauchy sequences in the space of continuous mappings $C(\mathscr{Q},\mathscr{Q})$, thus they converge (uniformly)
to some $\Phi$ and $\Psi\in C(\mathscr{Q},\mathscr{Q})$, respectively.
To see that $\Psi=\Phi^{-1}$, fix a point $x\in \mathscr{Q}$ and pass with $k$ to the limit in the equality $\Phi_{k}(\Phi_{k}^{-1}(x))=x$ to prove that $\Phi(\Psi(x))=x$.
We show that $\Psi(\Phi(x))=x$ in an analogous way.
\end{proof}
\begin{lemma}
\label{T46}
Assume that $\Phi:\mathscr{Q}\to\mathscr{Q}$ is a $C^1$-diffeomorphism such that $\Phi=\id$ in a neighborhood of $\partial\mathscr{Q}$. Then $\Phi$ can be approximated in the uniform metric $d$ by a sequence of $C^\infty$-diffeomorphisms $\Phi_k\overset{d}{\to}\Phi$ such that $\Phi_k=\id$ in a neighborhood of $\partial\mathscr{Q}$.
\end{lemma}
\begin{proof}
Approximating $\Phi$ by convolution with a standard symmetric mollifier $\psi_\eps$ we obtain smooth maps $\Phi_\eps=\Phi*\psi_\eps$ that are identity near $\partial\mathscr{Q}$ and converge uniformly to $\Phi$ on $\mathscr{Q}$. Since $\det\Phi_\eps\to\det\Phi$ uniformly, we see that $\det\Phi_\eps>0$ in $\mathscr{Q}$, provided $\eps>0$ is sufficiently small. This implies that $\Phi_\eps$ is a local diffeomorphism and according to Lemma~\ref{T14} it is a global diffeomorphism of $\mathscr{Q}$ onto itself. It easily follows that $\Phi_\eps\to\Phi$ in the uniform metric $d$.
\end{proof}
\section{Proof of Theorem~\ref{T31}}
\label{trzy}

\begin{proof}[Proof of Theorem~\ref{T31}]
First, we prove the theorem under the assumption that
\begin{equation}
\label{eq83}
\int_\Omega\det T(x)\, dx = |F(\Omega)|.
\end{equation}
The general case will then easily follow from this one.

Let $0<\eps<|\Omega|$ be given and fix $\beta$ such that
$(1-\eps|\Omega|^{-1})^{1/8}< \beta <1$.

Corollary~\ref{T29} yields a~$C^\infty$-diffeomorphism $\Psi: \Omega \to F(\Omega)$ and a~compact set $K_1 \subset \Omega$ with $|\Omega \setminus K_1| < \tfrac{1}{2}(1-\beta)|\Omega|$ such that
$\Psi = F$ near $\partial \Omega$ and
\begin{equation}
\label{eq29}
|\det D\Psi(x) - \det T(x)| < (1 - \beta) \det T(x) \text{ for } x \in K_1.
\end{equation}

On the other hand, Theorem~\ref{T30} gives us a mapping $\mathcal{A} \in C^1_c(\Omega, \bbbr^n)$ and a compact set $K_2\subset\Omega$ with
$|\Omega\setminus K_2| <\tfrac{1}{2}(1-\beta)|\Omega|$ such that
\begin{equation}
\label{eq78}
D\mathcal{A}(x) = T(x)  \text{ for } x \in K_2.
\end{equation}

Let $G\subset K_1\cap K_2$ be the set of density points of $K_1\cap K_2$ that belong to $K_1\cap K_2$ and observe that
$$
|G| = |K_1 \cap K_2| >\beta |\Omega|.
$$
It follows from \eqref{eq29} and \eqref{eq78} that
\begin{equation}
\label{eq79}
D\mathcal{A}(x) = T(x)
\quad
\text{and}
\quad
\det D\Psi(x)>\beta \det D\mathcal{A}(x)>0 \text{ for all } x \in G.
\end{equation}
Let us interrupt the proof for a moment and explain its main idea. The idea is to find a finite family of balls
--- in the proof it will be the family $\{B(p_{ij},\beta^{2/n}r_{ij})\}_{i,j}$ --- such that we can replace $\Psi$ with $\mathcal{A}(x)+\tau_{ij}$ for some $\tau_{ij}\in\bbbr^n$ on each of the balls and the resulting map $\Phi$ will be a diffeomorphism. Then $\Phi=\Psi=F$ near $\partial\Omega$ and
\begin{equation}
\label{eq82}
D\Phi(x) = D\mathcal{A}(x)=T(x)
\text{ for } x\in G\cap\bigcup\nolimits_{ij} B(p_{ij},\beta^{2/n}r_{ij}).
\end{equation}
Moreover, the family of balls will be constructed in such a way that the measure of the complement of the set in \eqref{eq82} (reproduced in the proof as \eqref{eq81}) will be less than $\eps$. This will complete the proof. We will replace $\Psi$
with $\mathcal{A}(x)+\tau_{ij}$ by a sequence of diffeomorphic gluing. We will glue $\Psi$ with its affine approximation, then we will glue the affine approximation of $\Psi$ with the affine approximation of $\mathcal{A}(x)+\tau_{ij}$, which we will glue with $\mathcal{A}(x)+\tau_{ij}$. To this end we will use \eqref{eq24} and \eqref{eq55} in Lemma~\ref{T17}, Proposition~\ref{T21} and Proposition~\ref{T38}. Now, we shall return to the proof.

For any $x_o \in G$, there is $r_{x_o} > 0$ such that for all $r \leq r_{x_o}$, the following conditions hold
\begin{enumerate}[(a)]
\item $B(x_o, r) \Subset \Omega$;
\item $|B(x_o, r) \cap G| \geq \beta |B(x_o, r)|$;
\item $\mathcal{A}$ is a~diffeomorphism on $B(x_o, r)$;
\item $D\mathcal{A}(x)$ is close to $D\mathcal{A}(x_o)$ for $x \in B(x_o, r)$ in the sense that
\begin{equation}
\label{eq70}
\sup_{x \in B(x_o,r)} \|D\mathcal{A}(x) - D\mathcal{A}(x_o)\| < \left(\beta^{-1/(2n)} - 1 \right)
\Vert (D\mathcal{A}(x_o))^{-1} \Vert^{-1};
\end{equation}
\item there exists a diffeomorphism $\Psi'_{x_o, r}:\Omega\to F(\Omega)$ such that
\begin{equation}
\label{eq30}
\Psi'_{x_o, r}(x) =
\begin{cases}
\Psi(x_o) + D\Psi(x_o)(x - x_o) & \text{ for } x \in \overbar{B}(x_o, \beta^{1/n} r),\\
\Psi(x) & \text{ for } x \in \Omega \setminus B(x_o, r).
\end{cases}
\end{equation}
\end{enumerate}
Property (b) follows from the fact that $x_o$ is a~density point of $G$. Property (c) follows from \eqref{eq79}.
Property (d) is a consequence of continuity of $D\mathcal{A}$.
Finally, (e) follows from Lemma~\ref{T17}.

The family of balls
$$
\mathfrak{B} = \{\overbar{B}(x_o, r): \, x_o \in G,\, r \leq r_{x_o}\}
$$
is a Vitali covering of $G$,
so by Vitali's covering theorem, we can choose a~finite subfamily of pairwise disjoint balls $\{\overbar{B}_{i}\}_{i=1}^{N}$ so that the measure of these balls satisfies
\begin{equation}
\label{eq32}
\Big| \bigcup_{i=1}^{N} \overbar{B}_{i} \Big| \geq \beta |G| > \beta^2 |\Omega|.
\end{equation}
We replace $\Psi$ with $\Psi_{x_i,r_i}'$ in each of the balls $\overbar{B}_i=\overbar{B}(x_i,r_i)$. The resulting diffeomorphism satisfies
$$
\Psi'(x)=
\begin{cases}
\Psi(x) & \text{on\ \  } \Omega\setminus\bigcup_{i=1}^N B_i,\\
\Psi(x_i)+D\Psi(x_i)(x-x_i) & \text{on\ \  } B_i':=B(x_i,\beta^{1/n}r_i).
\end{cases}
$$
In particular, $\Psi'=F$ near $\partial\Omega$.

Let $r:G\to (0,\infty)$ be such that for any $p\in G$ and any $0<r<r(p)$,
\begin{equation}
\label{eq72}
|G \cap B(p, r)| > \beta |B(p,r)|
\end{equation}
and there is a diffeomorphism on $\bbbr^n$ that equals
\begin{equation}
\label{eq80}
\Upsilon(x)=
\begin{cases}
\mathcal{A}(x) & \text{on } \overbar{B}(p,\beta^{2/n}r),\\
\mathcal{A}(p)+D\mathcal{A}(p)(x-p) &\text{on } \bbbr^n\setminus \overbar{B}(p,\beta^{3/(2n)}r).
\end{cases}
\end{equation}
Existence of such a function $r$ is guaranteed by the fact that $G$ consists of density points and by
formula \eqref{eq55} in Lemma~\ref{T17} (note that $\mathcal{A}$ is a diffeomorphism in a neighborhood of $p$, by (c) above).

Since $x_i\in G$, \eqref{eq79} yields that $D\Psi(x_i), D\mathcal{A}(x_i)\in GL(n)^+$ and
$$
\det D\Psi(x_i)> \beta \det D\mathcal{A}(x_i).
$$
This is condition \eqref{eq46} from Proposition~\ref{T38} which we want to use to modify $\Psi'$ on each of the balls $B_i'$.
Proposition~\ref{T38} and Remark~\ref{R1} give a finite family of pairwise disjoint closed balls $\overbar{B}_{ij}=\overbar{B}(p_{ij},r_{ij})\subset B_i'$ such that
\begin{equation}
\label{eq33}
p_{ij}\in {G}\cap B_i',
\quad
r_{ij}<r(p_{ij}),
\quad
\Big|{G} \cap \bigcup_{j} \overbar{B}_{ij}\Big|>
\beta |{G} \cap B'_i| > \beta^2 |B'_i|
\end{equation}
(the last inequality follows from (b)),
and a diffeomorphism $F_i:B_i'\to\bbbr^n$ such that
$$
F_i(x)=\Psi(x_i)+D\Psi(x_i)(x-x_i)
\quad
\text{in a neighborhood of } \partial B_i',
$$
and
$$
F_i(x) = D\mathcal{A}(x_i)x + v_{ij} \text{ for } x \in \overbar{B}_{ij}
\text{ and some } v_{ij}\in\bbbr^n.
$$
Note that $F_i=\Psi'$ in a neighborhood of $\partial B_i'$. Hence, if we replace $\Psi'$ with $F_i$ on each of the balls $B_i'$, we will obtain a diffeomorphism $F':\Omega\to F(\Omega)$ which agrees with $F$ near $\partial\Omega$ and satisfies
\begin{equation}
\label{eq71}
F'(x)=D\mathcal{A}(x_i)x + v_{ij} \text{ for } x \in \overbar{B}_{ij}
\text{ and some } v_{ij}\in\bbbr^n.
\end{equation}

We will now replace the affine map \eqref{eq71} with a diffeomorphism $\mathcal{A}(x)+\tau_{ij}$, $\tau_{ij}\in\bbbr^n$, in two steps. In the first step, we will replace \eqref{eq71} with an affine map $D\mathcal{A}(p_{ij})x + w_{ij}$ (on a smaller ball) using Proposition~\ref{T21}. Then we will replace this new affine map (on an even smaller ball) with a diffeomorphism $\mathcal{A}(x)+\tau_{ij}$ using formula \eqref{eq55} from Lemma~\ref{T17}.

Let us fix $A_1 := D\mathcal{A}(p_{ij})$ and $A_2 := D\mathcal{A}(x_i)$ and observe that in view of~\eqref{eq70},
$$
\| A_1 - A_2\| < (\beta^{-1/(2n)} - 1)  \Vert A_2^{-1}\Vert^{-1},
$$
which by triangle inequality and sublinearity of operator norm implies that
\begin{equation*}
\| A_2^{-1}  A_1 \| = \| A_2^{-1} \left( A_1 - A_2 \right) + \mathcal{I} \|
 \leq \| A_2^{-1} \| \, \| A_1 - A_2\| + 1
< \beta^{-1/(2n)}.
\end{equation*}
Therefore,
$$
A_2^{-1}A_1(B(0,1))\Subset B(0,\beta^{-1/(2n)})
\quad
\text{so}
\quad
A_1(B(0, \beta^{1/n}r_{ij})) \Subset A_2(B(0, \beta^{1/(2n)}r_{ij})).
$$
Applying Proposition~\ref{T21} yields a~diffeomorphism $\Theta_{ij}: \bbbr^n \to \bbbr^n$, which equals $A_1$ on $B(0, \beta^{1/n}r_{ij})$ and $A_2$ on $\bbbr^n \setminus B(0,\beta^{1/(2n)}r_{ij})$.
Set
$$
H_{ij}(x) := \Theta_{ij}(x - p_{ij}) + A_2 p_{ij} + v_{ij},
$$
where $v_{ij}$ were defined in \eqref{eq71}. We have
$$
H_{ij}(x) =
\begin{cases}
    D\mathcal{A}(p_{ij})x \underbrace{-D\mathcal{A}(p_{ij}) p_{ij} + D\mathcal{A}(x_i) p_{ij} + v_{ij}}_{\omega_{ij}} & \text{ on }  B(p_{ij}, \beta^{1/n} r_{ij}), \\
    D\mathcal{A}(x_i)x + v_{ij} & \text{ on }  \bbbr^n\setminus B(p_{ij},\beta^{1/(2n)}r_{ij}).
\end{cases}
$$
Observe that according to the definition of $F'$ in \eqref{eq71},
$$
H_{ij} = F'
\text{ on }
B(p_{ij}, r_{ij})\setminus B(p_{ij},\beta^{1/(2n)}r_{ij}), \text{ for all $i$ and $j$}.
$$
Hence, if we replace $F'$ with $H_{ij}$ on each of the balls $\overbar{B}_{ij}$, we will obtain a diffeomorphism $F'':\Omega\to F(\Omega)$ which agrees with $F$ near $\partial\Omega$ and satisfies
$$
F''(x)=D\mathcal{A}(p_{ij})x+\omega_{ij}
\text{ for } x\in \overbar{B}(p_{ij},\beta^{1/n}r_{ij})
\text{ and some } \omega_{ij}\in\bbbr^n.
$$
Since $r_{ij}<r(p_{ij})$, \eqref{eq80} gives a~diffeomorphism $\Upsilon_{ij}$ such that
$$
\Upsilon_{ij} (x) =
\begin{cases}
    \mathcal{A}(x) & \text{ for } x \in \overbar{B}(p_{ij}, \beta^{2/n}r_{ij}),\\
    \mathcal{A}(p_{ij}) + D\mathcal{A}(p_{ij})(x - p_{ij}) & \text{ for } x \in \bbbr^n\setminus \overbar{B}(p_{ij}, \beta^{3/(2n)}r_{ij}).
\end{cases}
$$
The translated diffeomorphism
$$
\Upsilon_{ij}'(x):= \Upsilon_{ij}(x) + \omega_{ij} - \mathcal{A}(p_{ij}) + D\mathcal{A}(p_{ij})p_{ij}
=:\Upsilon_{ij}(x)+\tau_{ij}
$$
satisfies
$$
\Upsilon'_{ij}(x) =
\begin{cases}
\mathcal{A}(x)+\tau_{ij} & \text{ for } x\in \overbar{B}(p_{ij}, \beta^{2/n}r_{ij}),\\
D\mathcal{A}(p_{ij})x + \omega_{ij}  & \text{ for } x  \in \bbbr^n\setminus \overbar{B}(p_{ij}, \beta^{3/(2n)}r_{ij}).
\end{cases}
$$
Observe that
$$
\Upsilon_{ij}'=F''
\text{ on }
B(p_{ij},\beta^{1/n}r_{ij})\setminus B(p_{ij},\beta^{3/(2n)}r_{ij}).
$$
Hence if we replace $F''$ with $\Upsilon_{ij}'$ on $\overbar{B}(p_{ij},\beta^{1/n}r_{ij})$, we obtain a diffeomorphism $\Phi:\Omega\to F(\Omega)$ which agrees with $F$ near $\partial\Omega$
and satisfies
$$
\Phi(x)= \mathcal{A}(x)+\tau_{ij}
\text{ for } x\in \overbar{B}(p_{ij},\beta^{2/n}r_{ij})
\text{ and some } \tau_{ij}\in\bbbr^n.
$$

Clearly, $D\Phi=D\mathcal{A}$ on ${B}(p_{ij},\beta^{2/n}r_{ij})$. Since $D\mathcal{A}=T$  on $G$ by \eqref{eq79}, we have that
\begin{equation}
\label{eq81}
D\Phi(x)=T(x)
\text{ for } x\in G\cap\bigcup\nolimits_{i,j} B(p_{ij},\beta^{2/n}r_{ij})
\end{equation}
and it remains to show that the complement of this set has measure less that $\eps$.

Since $p_{ij}\in G$ and $r_{ij}<r(p_{ij})$, \eqref{eq72} implies that
$$
|G\cap B(p_{ij},\beta^{2/n}r_{ij})|>
\beta |B(p_{ij},\beta^{2/n}r_{ij})|=\beta^3|B_{ij}|.
$$
Therefore, the fact that the balls $B_{ij}$ are pairwise disjoint, \eqref{eq33} and \eqref{eq32} give
$$
\Big|G\cap\bigcup\nolimits_{ij} B(p_{ij},\beta^{2/n}r_{ij})\Big|
>\beta^3\sum_{i,j}|B_{ij}|>
\beta^5\sum_i|B_i'|
 =
\beta^6\sum_i|B_i|>
\beta^8|\Omega|.
$$
Clearly, we can find a compact set $K$ contained in the set \eqref{eq81} so that
$|K|>\beta^8|\Omega|$ and hence
$|\Omega\setminus K|<(1-\beta^8)|\Omega|<\eps$.
This completes the proof under the assumption \eqref{eq83}.

Now we will prove the result in the general case, when $\int_\Omega\det T\leq |F(\Omega)|$. It is easy to see that there is a measurable map $\widetilde{T}:\Omega\to GL(n)^+$ such that $\int_\Omega\det\widetilde{T}=|F(\Omega)|$ and $|\{\widetilde{T}\neq T\}|<\eps/2$.
We proved that in this situation there is a $C^1$-diffeomorphism $\Phi:\Omega\to F(\Omega)$ that agrees with $F$ near $\partial\Omega$ and satisfies
$|\{ D\Phi\neq\widetilde{T}\}|<\eps/2$. Clearly, $|\{ D\Phi\neq T\}|<\eps$ and we can find a compact set $K\subset\Omega$ such that $D\Phi=T$ on $K$ and $|\Omega\setminus K|<\eps$.
The proof is complete.
\end{proof}

\section{Preliminaries for Theorem~\ref{T36}}
\label{S3}

This section consists of a series of similar lemmata of increasing complexity, whose aim is to prove Proposition~\ref{T32}. In this proposition we construct a particular partition of a diffeomorphic closed cube, which is instrumental to carry out the iteration in the proof of Theorem~\ref{T36}.

Essentially, the results in this section are far-reaching modifications of the construction of the mapping in the proof of the homeomorphic measures theorem by Oxtoby and Ulam, see \cite[Theorem 2]{OxtobyUlam41} for the original paper or \cite[Section A2.2]{AlpernPrasad}, \cite[Chapter 7]{goffman} for a~concise treatment.

The original construction of Oxtoby and Ulam, \cite{OxtobyUlam41,AlpernPrasad,goffman}, does not lead to any differentiability properties of the homeomorphism, even if the measure is absolutely continuous with respect to the Lebesgue measure (which is the case considered by us). Therefore, in order to prove the a.e. approximate differentiability claimed in Theorem~\ref{T36}, we need essential modifications of the argument of Oxtoby and Ulam.

\begin{lemma}
\label{T8}
Let $P$ be a rectangular box
$$
P=P_1\cup\ldots\cup P_k=[0,a]\times [0,1]^{n-1},
\quad
k\geq 2,
$$
represented as the union of adjacent boxes
$$
P_i=[a_{i-1},a_i]\times[0,1]^{n-1},
\quad
0=a_0<a_1<\ldots<a_k=a.
$$
If functions $f,g\in L^1(P)$, $f,g>0$ a.e., are such that
\begin{equation}
\label{eq12}
\int_P f(x)\, dx=\int_P g(x)\, dx,
\end{equation}
then there is a diffeomorphism $\Phi:P\to P$ that is identity in a neighborhood of $\partial P$, and such that
$$
\int_{P_i} f(x)\, dx=\int_{\Phi(P_i)} g(x)\, dx
\quad
\text{for } i=1,2,\ldots,k.
$$
\end{lemma}
\begin{remark}
The lemma has a simple geometric interpretation. The functions $f$ and $g$ are densities of absolutely continuous measures $\mu_f$ and $\mu_g$, and \eqref{eq12} means that $\mu_f(P)=\mu_g(P)$. Then, the lemma says that given the partition $P_1,\ldots,P_k$ of $P$, we can find a diffeomorphic  partition $\Phi(P_1),\ldots,\Phi(P_k)$ of $P$, such that the corresponding cells have equal measures $\mu_f(P_i)=\mu_g(\Phi(P_i))$.
\end{remark}
\begin{proof}[Proof of Lemma~\ref{T8}]
The proof will be by induction with respect to $k$. Thus, first assume that $k=2$. Since
$\int_{P_1}f+\int_{P_2}f=\int_{P_1}g+\int_{P_2}g$, without loss of generality, we may assume that $\int_{P_1}f\geq \int_{P_1}g$.
Let $K$ be a compact rectangular box in the interior of $P$, with edges parallel to the coordinate axes, such that
$$
\int_{P_1} f(x)\, dx< \int_Kg(x)\, dx.
$$
By taking $K$ sufficiently large, we may assume that the common face of $P_1$ and $P_2$ intersects $K$.
Let $X$ be a smooth vector field parallel to the $x_1$ coordinate axis, non-zero in a neighborhood of $K$, zero in a neighborhood of $\partial P$, and such that $X$ points in the positive direction of the $x_1$-axis whenever $X\neq 0$. If $\Phi_t$ is the one-parameter family of diffeomorphisms generated by $X$, then $\Phi_0(P_1)=P_1$, so
\begin{equation}
\label{eq13}
\int_{P_1}f(x)\, dx\geq \int_{\Phi_0(P_1)}g(x)\, dx.
\end{equation}

Since the common face of $P_1$ and $P_2$ intersects with $K$, $X$ is non-zero in a neighborhood of that intersection. Now, from a~standard compactness and open covering argument recalled below, we see that there is $t_o$ such that $K\subset\Phi_t(P_1)$ for all $t>t_o$, so
\begin{equation}
\label{eq14}
\int_{P_1} f(x)\, dx<\int_K g(x)\, dx\leq\int_{\Phi_t(P_1)} g(x)\, dx
\quad
\text{for all } t>t_o.
\end{equation}
Indeed, since $\Phi_t$ depends continuously on the parameter $t$, for any $x \in K$, there is a~$t_x>0$ and $\eps_x>0$ such that for $t > t_x$, $\Phi_{-t}(B(x, \eps_x)) \subset P_1$. Balls $B(x, \eps_x)$ for $x \in K$ form an open covering of $K$ and by compactness of $K$, we can choose a~finite subcovering of $K$, balls $B(x_i, \eps_{x_i})$ for $i = 1, \ldots, N$. Setting $t_o:= \max_i t_{x_i}$, we see that for $t > t_o$ for all $x \in K$, $\Phi_{-t}(x) \in P_1$, i.e., $x \in \Phi_t (P_1)$. This implies that $K \subset \Phi_t(P_1)$ for $t > t_o$.

Since the function $t\mapsto\int_{\Phi_t(P_1)}g$ is continuous, it follows from \eqref{eq13} and \eqref{eq14} that there is $t\in [0,t_o]$, such that $\Phi:=\Phi_t$ satisfies
$$
\int_{P_1} f(x)\, dx=\int_{\Phi(P_1)} g(x)\, dx
\quad
\text{and hence}
\quad
\int_{P_2} f(x)\, dx = \int_{\Phi(P_2)} g(x)\, dx.
$$
Observe that the diffeomorphism $\Phi$ equals identity in a neighborhood of $\partial P$, on the set where $X=0$.

We completed the proof in the case of $k=2$. Suppose now that the claim is true for all integers in $\{2,\ldots,k\}$ and we will prove it for $k+1$.

Let us write
$$
P=\underbrace{P_1\cup\ldots\cup P_k}_{\widetilde{P}}\cup P_{k+1}
=\widetilde{P}\cup P_{k+1}.
$$
Applying the claim for two boxes, we can find a diffeomorphism
$\Phi_1:P\to P$, that is identity in a neighborhood of $\partial P$ and such that
\begin{equation}
\label{eq15}
\int_{\widetilde{P}} f(x)\, dx = \int_{\Phi_1(\widetilde{P})} g(x)\, dx
\quad
\text{and}
\quad
\int_{P_{k+1}} f(x)\, dx = \int_{\Phi_1(P_{k+1})} g(x)\, dx.
\end{equation}
Note that the second equality in \eqref{eq15} is as desired and our diffeomorphism $\Phi$ will be equal $\Phi_1$ in $P_{k+1}$. However, we have to modify it in $\widetilde{P}$. The diffeomorphism $\Phi_1$ is orientation preserving and hence its Jacobian $J_{\Phi_1} = \det D\Phi_1$ is positive in $\widetilde{P}$. Let
$$
\widetilde{g}(x)=(g\circ\Phi_1)(x)J_{\Phi_1}(x)
\quad
\text{for } x\in\widetilde{P}.
$$
The change of variables formula yields
$$
\int_{\widetilde{P}} \widetilde{g}(x)\, dx =
\int_{\widetilde{P}}(g\circ\Phi_1)(x)J_{\Phi_1}(x)\, dx =
\int_{\Phi_1(\widetilde{P})} g(x)\, dx = \int_{\widetilde{P}}f(x)\, dx,
$$
and hence the pair of functions $f$ and $\widetilde{g}$ satisfies the  assumption \eqref{eq12} on $\widetilde{P}=P_1\cup\ldots\cup P_k$. Thus, the induction hypothesis yields a diffeomorphism $\Phi_2:\widetilde{P}\to\widetilde{P}$ that is identity near $\partial\widetilde{P}$ and such that for $i=1,2,\ldots,k$, we have
$$
\int_{P_i} f(x)\, dx = \int_{\Phi_2(P_i)}\widetilde{g}(x)\, dx=
\int_{\Phi_2(P_i)}(g\circ\Phi_1)(x)J_{\Phi_1}(x)\, dx=
\int_{(\Phi_1\circ\Phi_2)(P_i)}g(x)\, dx.
$$
Therefore, the diffeomorphism $\Phi:P\to P$ defined by
$$
\begin{cases}
\Phi_1\circ\Phi_2(x) & \text{if } x\in \widetilde{P},\\
\Phi_1(x)            & \text{if } x\in P_{k+1},
\end{cases}
$$
satisfies the claim for $k+1$. To see that $\Phi$ is a well defined diffeomorphism, note that $\Phi_2$ is identity near $\partial\widetilde{P}$ and hence $\Phi_1\circ\Phi_2=\Phi_1$ near the common face of the boxes $\widetilde{P}$ and $P_{k+1}$. Also, $\Phi$ is identity near the boundary of $P$. The proof is complete.
\end{proof}
The following corollary shows that Lemma~\ref{T8} can be applied to diffeomorphic closed cubes and their partitions that are diffeomorphic to the partition in Lemma \ref{T8}.

\begin{corollary}
\label{T10}
Let $P$ and its partition be as in Lemma~\ref{T8}. Let $\Theta:P\to  \bbbr^n$ be a diffeomorphism of the closed rectangular box $P$. Denote by $\widetilde{P}$ and $\widetilde{P}_i$ the images of $P$ and $P_i$ under $\Theta$.
If functions $f,g\in L^1(\widetilde{P})$, $f,g>0$ a.e., are such that
$$
\int_{\widetilde{P}} f(x)\, dx=\int_{\widetilde{P}} g(x)\, dx,
$$
then there is a diffeomorphism $\Phi:\widetilde{P}\to \widetilde{P}$, that is identity in a neighborhood of $\partial \widetilde{P}$, and such that
$$
\int_{\widetilde{P}_i} f(x)\, dx=\int_{\Phi(\widetilde{P}_i)} g(x)\, dx
\quad
\text{for } i=1,2,\ldots,k.
$$
\end{corollary}
\begin{proof}
Using $\Theta$ as a change of variables we can reduce the problem to Lemma~\ref{T8}. The induced functions on $P$ will be
$$
(f\circ\Theta)(x)|J_\Theta(x)|
\qquad
\text{and}
\qquad
(g\circ\Theta)(x)|J_\Theta(x)|.
$$
Actually, we used a very similar argument in the proof of Lemma~\ref{T8} and we leave easy details to the reader.
\end{proof}

Naturally, given a~compact set $K$ in a~rectangular box $P$, it cannot be expected that $\overbar{P \setminus K}$ is  diffeomorphic to a cube. However, in Lemma~\ref{T24} we show that $K$ can be replaced by another compact set $K'$, with small measure of the symmetric difference $|K \vartriangle K'|$, so that $\overbar{P \setminus K'}$ is a diffeomorphic closed cube. We do it by approximating $K$ with small balls and smoothly connecting them with thin tubes which start from one face of $P$, see Figure~\ref{fig:1}. We will need it in Lemma~\ref{T25} to construct a~diffeomorphism as in Lemma~\ref{T8} which additionally is identity on a~large part of a~given compact set.

\begin{lemma}
\label{T24}
Let $a > 0$, $P = [0,a] \times [0,1]^{n-1}$ be a~rectangular box and
$F=(0,a) \times (0,1)^{n-2} \times \{0\}$ a~fixed open face of $P$.
If $K$ is a~compact subset of $P$, then for any $\eps >0$, it is possible to find a~compact set $K' \subset P$ and a~diffeomorphism $\Psi: \bbbr^n \to \bbbr^n$ such that
$$
K'\cap (\partial P\setminus F)=\varnothing,
\quad
|K \vartriangle K'| < \eps,
\quad
\Psi (P) = \overbar{P \setminus K'},
$$
and  $\Psi = \id$ outside an arbitrarily small neighborhood of $K'$.
\end{lemma}

\begin{proof}
\begin{figure}

\begin{tikzpicture}
\node at (0,0) {\includegraphics[scale=0.8]{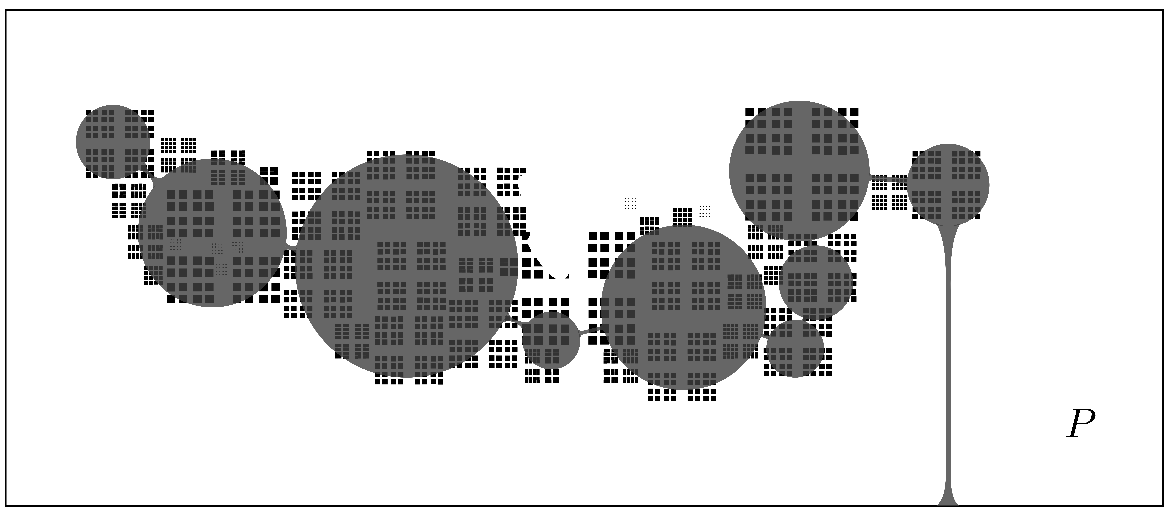}};
\node[scale=1.4] at (-4.5,-2.3) {$F$};
\draw[thick,->] (-4.2,-2.3) to[out=-15,in=110] (-2,-3.2);
\end{tikzpicture}
\caption{The compact set $K\subset P$ (black) is  approximated by the set $K'$ (grey) which consists of a finite number of balls, connected smoothly by thin tubes to the side $F$ of $P$.}
\label{fig:1}
\end{figure}
We find a~finite number of disjoint closed balls $B_{j}$ contained in $\mathring{P}$ so that $|K \vartriangle \bigcup_j B_{j}| < \eps/2$ (easy exercise)
and a~smooth curve $\gamma$ without self-intersections which starts at a~point of the open face $F$ and connects all balls: the curve enters each ball once and leaves it at another point, except the last ball that it does not leave. We can easily guarantee that $\gamma$ intersects $\partial P$ only at the starting point in $F$.

We thicken slightly the curve $\gamma$ to get a~tube $\gamma_\eps$ so that the measure of $\gamma_\eps$ does not exceed $\eps/2$. In doing so, we can guarantee that $\gamma_\eps$ connects the balls and touches $\partial P$ at the face $F$ only, in a~smooth manner, see Figure \ref{fig:1}.
The desired compact set $K'$ consists of $\bigcup_j B_{j}$ and $\gamma_\eps$. Clearly, $|K \vartriangle K'| < \eps$.

Since $K'$ is diffeomorphic to a~ball connected smoothly with a~thin tube to the face $F$ of $P$, the set $\overbar{P \setminus K'}$ is diffeomorphic to a~rectangular box. Intuitively speaking, diffeomorphism $\Psi: \bbbr^n \to \bbbr^n$, which maps $P$ onto $\overbar{P \setminus K'}$ pushes a~smooth cylinder glued to $F$ outside $P$ inside this rectangular box, transforming $P$ into $\overbar{P \setminus K'}$. It can be guaranteed that $\Psi$ fixes points outside an arbitrarily small neighborhood of $K'$.
\end{proof}

The next lemma is an enhanced version of Lemma~\ref{T8}. We find a~diffeomorphism of a~rectangular box onto itself which transforms measures of cells accordingly and, additionally, equals identity on a~large part of a~given compact subset $K$ of $P$.
\begin{lemma}
\label{T25}
Let $P$, $P_i$, $f$ and $g$ be as in Lemma~\ref{T8}. Assume that $K$ is a~compact subset of $P$ and that the functions $f$ and $g$ satisfy the additional condition that
\begin{equation}
\label{eq28}
f = g \text{ a.\,e.\ on } K.
\end{equation}
Then, for any $\eps > 0$, there is a~diffeomorphism $\Phi: P \to P$ and a~compact set $\widetilde{K}\subset K$ such that $\Phi$ equals identity in a neighborhood of $\widetilde{K}\cup\partial P$, $|K\setminus \widetilde{K}| < \eps$ and
$$
\int_{P_i} f(x)\, dx=\int_{\Phi(P_i)} g(x)\, dx \quad \text{for } i=1,2,\ldots,k.
$$
\end{lemma}
\begin{proof}
The main idea of the proof is applying Lemma~\ref{T24} to each $P_i$ to eventually remove a~set including a~large part of set $K$ and to obtain a~diffeomorphic closed cube. We can then use Corollary~\ref{T10} to construct the desired diffeomorphism which transforms the measures of cells and equals identity near the boundary.

By removing from $K$ a subset $A$ with $|A|<\eps/2$ we can assume that $\hat{K}:=K\setminus A$ is compact and $K_i:=\hat{K}\cap P_i\neq P_i$ for all $i=1,2,\ldots, k$.

Choose a sufficiently small $0<\eta<\eps/(2k)$ so that
\begin{equation}
\label{eq67}
|E|<\eta
\quad
\Longrightarrow
\quad
\int_E g(x)\,dx<\min_i \int_{P_i\setminus K_i} f(x)\,dx.
\end{equation}
(Since $K_i\neq P_i$, the minimum in \eqref{eq67} is positive.)
Lemma~\ref{T24} applied to each $P_i$, the compact sets $K_i$ and the chosen $\eta$ yields diffeomorphisms $\Psi_i:\bbbr^n\to\bbbr^n$ and compact sets $K_i'$ for $i=1,2,\ldots,k$ such that $\Psi_i(P_i)=\overbar{P_i\setminus K_i'}$, $|K_i\vartriangle K_i'|<\eta$, and $\Psi_i=\id$ outside an arbitrarily small neighborhood $U_i$ of $K_i'$. Also, the sets $K_i'$ do not intersect any of the common faces of the partition of $P$ and thus they are pairwise disjoint.

Since the sets $K_i'$ are compact and pairwise disjoint, we may guarantee that the neighborhoods $U_i$ are pairwise disjoint and hence the diffeomorphisms $\Psi_i$ can be glued together to a diffeomorphism $\Psi:\bbbr^n\to\bbbr^n$ such that $\Psi=\Psi_i$ on each $P_i$.

Consequently, if $K' = \bigcup_{i=1}^k K'_i$, then $\overline{P \setminus K'}$ is diffeomorphic to $P$ by this diffeomorphism $\Psi$ and hence it satisfies the assumptions of Corollary~\ref{T10} for
$$
\widetilde{P} = \overline{P \setminus K'}, \quad \widetilde{P}_i = \overline{P_i \setminus K'_i}.
$$

At this point, observe that
$|P\setminus K'|=|\overline{P\setminus K'}|=|\widetilde{P}|$,
because the boundary of $K'$ is piecewise smooth. Also, it is obvious that
\begin{equation}
\label{eq109}
   P = \widetilde{P} \cup K' \quad \text{and} \quad P_i = \widetilde{P}_i \cup K'_i \quad \text{for } i =1, \ldots, k.
\end{equation}
We still need to define suitable functions to use in place of $f$ and $g$ in Corollary~\ref{T10}. Observe that $|K_i \vartriangle K'_i| < \eta$ for each $i = 1, \ldots, k$, yields
\begin{align*}
\int_{K'_i} g(x) \, dx & \leq \int_{K_i} g(x) \, dx + \int_{K'_i \setminus K_i} g(x) \, dx  \\
&\overset{\eqref{eq67}}{<}
 \int_{K_i} g(x) \, dx + \int_{P_i \setminus K_i} f(x) \, dx
 \overset{\eqref{eq28}}{=} \int_{P_i} f(x) \, dx.
\end{align*}
Consequently, the function
\begin{equation}
\label{eq116}
\hat{f}(x) = \sum_{i=1}^k   \left( \int_{P_i} f - \int_{K'_i} g \right)\,\frac{\chi_{\widetilde{P}_i}(x)}{|\widetilde{P}_i|}
\end{equation}
is positive a.\,e.\ on $\widetilde{P}$ and satisfies
\begin{equation}
\label{eq108}
\int_{\widetilde{P}} \hat{f}(x) \, dx=\int_{\widetilde{P}} g(x) \, dx.
\end{equation}
Indeed, since $|P\setminus K'|=|\widetilde{P}|$, \eqref{eq108} follows from
$$
\int_{\widetilde{P}} \hat{f}(x) \, dx=\sum_{i=1}^k\Big(\int_{P_i}f - \int_{K_i'} g\Big)=
\int_P f(x)\, dx-\int_{K'}g(x)\, dx \overset{\eqref{eq12}}{=} \int_{P\setminus K'}g(x)\, dx.
$$

All in all, we have checked the assumptions of Corollary~\ref{T10} for functions $\hat{f}$ and $g$. This corollary provides us with a~diffeomorphism $\Phi': \widetilde{P} \to \widetilde{P}$, $ \Phi' = \mathrm{id}$ in a~neighborhood of $\partial \widetilde{P}$, such that
\begin{equation} \label{eq104}
\int_{\widetilde{P}_i} \hat{f}(x) \, dx
     = \int_{\Phi'(\widetilde{P}_i)} g(x) \, dx.
\end{equation}
Set $\widetilde{K} = K \cap K'$ and
$$
\Phi(x) = \begin{cases}
            x & \text{ for } x \in P \setminus \widetilde{P},\\
            \Phi'(x) & \text{ for } x \in \widetilde{P}.
        \end{cases}
$$

The set $\widetilde{K}$ is clearly compact whereas $\Phi$ is a~well defined diffeomorphism of $P$ onto itself since $\Phi'(x) = x$ near $\partial \widetilde{P}$. We shall now check that $\Phi$ and $\widetilde{K}$ satisfy all the desired properties. Writing $\partial P = (\partial P \setminus \partial \widetilde{P}) \, \cup \, (\partial P \cap \partial \widetilde{P})$, we see that $\Phi = \id$ near $\partial P$. It follows immediately from the definition of $\Phi$ and $\overline{P\setminus\widetilde{P}}=K'$, that $\Phi = \id$ on $K'$. We can say even more: since $\partial K' \subset \partial P \cup \partial \widetilde{P}$, $\Phi = \id$ in a~neighborhood of $K'$ and, consequently, in a~neighborhood of $\widetilde{K}$.

Since $\Phi=\Phi'$ on $\widetilde{P}_i$, \eqref{eq104} yields
\begin{equation}
\label{eq105}
\int_{\widetilde{P}_i} \hat{f}(x) \, dx = \int_{\Phi(\widetilde{P}_i)} g(x) \, dx.
\end{equation}
By~\eqref{eq109}, $\Phi(P_i) = \Phi(\widetilde{P}_i) \cup \Phi(K'_i)$
and the two sets $\Phi(\widetilde{P}_i)$ and $\Phi(K'_i)$ overlap on a set of measure zero (the image of a part of the piecewise smooth boundary of $K_i'$). Hence \eqref{eq105} yields
\begin{align*}
\int_{\Phi(P_i)} g(x) \, dx &= \int_{\Phi(\widetilde{P}_i)} g(x) \, dx + \int_{\Phi(K'_i)} g(x) \, dx = \int_{\widetilde{P}_i} \hat{f}(x) \, dx + \int_{K'_i} g(x) \, dx \\
&\overset{\eqref{eq116}}{=} \int_{P_i} f(x) \, dx - \int_{K'_i} g(x) \, dx + \int_{K'_i} g(x) \, dx  = \int_{P_i} f(x) \, dx,
\end{align*}
as required. At last, we compute
$$
|K\setminus\widetilde{K}|=|A\setminus K'|+|\widehat{K}\setminus K'|\leq
|A|+\sum_{i=1}^k |K_i\vartriangle K_i'|<\frac{\eps}{2}+k\eta<\eps
$$
which finishes the proof.
\end{proof}

\begin{lemma}
\label{T9}
Let $\mathscr{Q}=\bigcup_{j=1}^{2^{nk}} Q_j$ be the dyadic partition of the unit cube $\mathscr{Q}=[0,1]^n$ into  cubes $Q_j$ of side-length $2^{-k}$. Let $K$ be a~compact subset of $\mathscr{Q}$ and $f,g\in L^1(\mathscr{Q})$, $f,g>0$ a.\,e.\ be such that
$$
\int_\mathscr{Q} f(x)\, dx = \int_\mathscr{Q} g(x)\, dx \text{ and } f(x) = g(x) \text{ for a.\,e.\ } x \in K.
$$
Then, for any $\eps > 0$, there is a~diffeomorphism $\Phi:\mathscr{Q}\to \mathscr{Q}$ and a~compact set $\widetilde{K} \subset K$ such that $\Phi = \mathrm{id}$ in a neighborhood of  $\widetilde{K}\cup\partial \mathscr{Q}$, $|K \setminus \widetilde{K}| < \eps$  and
$$
\int_{Q_j} f(x)\, dx = \int_{\Phi(Q_j)} g(x)\, dx
\quad
\text{for } j=1,2,\ldots,2^{nk}.
$$
\end{lemma}

\begin{remark}
The statement of the lemma is very similar to that of Lemma~\ref{T25}. The main difference is that the boxes in Lemma~\ref{T25} were arranged into a single line. This arrangement played an important role in the proof and it is not obvious how to modify the proof of Lemma~\ref{T25} to cover the situation described in Lemma~\ref{T9}.
\end{remark}

\begin{proof}[Proof of Lemma \ref{T9}]
For $1\leq \ell\leq n$ and $j_1,\ldots,j_\ell\in \{1,\ldots,2^k\}$ we shall denote
\begin{align*}
P_{j_1\ldots j_\ell}&=[(j_1-1) 2^{-k},j_1 2^{-k}]\times [(j_2-1) 2^{-k},j_{2} 2^{-k}]\times\cdots\times [(j_\ell-1) 2^{-k},j_\ell 2^{-k}],\\
L_{j_1\ldots j_\ell}&=P_{j_1\ldots j_\ell}\times [0,1]^{n-\ell}.
\end{align*}
The sets $P_{j_1\ldots j_\ell}$ are cubes in the dyadic partition of $[0,1]^{n-\ell}$ and
$L_{j_1\ldots j_\ell}$ are ``towers'' over cubes $P_{j_1\ldots j_\ell}$ covering $[0,1]^n$. In particular $L_1,\ldots, L_{2^k}$ result from slicing of $\mathscr{Q}$ along the first coordinate like a toast bread into $2^k$ sandwiches. If $\ell=n$, $L_{j_1\ldots j_n}=P_{j_1\ldots j_n}$ are exactly the dyadic cubes $Q_j$.

Our aim is to prove (by finite induction on $\ell$) the following claim:
\begin{quote}
For any $\ell\in\{1,2,\ldots,n\}$ there exists a diffeomorphism $\Psi_\ell:\mathscr{Q}\to \mathscr{Q}$ and a~compact set $K_\ell \subset K$ such that
\begin{equation}
\label{eq106}
\Psi_\ell = \id \text{ in a neighborhood of } K_\ell \cup \partial \mathscr{Q},
\end{equation}
\begin{equation}
\label{eq10}
\int_{L_{j_1\ldots j_\ell}} f(x) dx = \int_{\Psi_\ell(L_{j_1\ldots j_\ell})} g(x) dx \quad \text{ for all }j_1,\ldots,j_\ell\in \{1,\ldots,2^k\},
\end{equation}
\begin{equation}
\label{eq66}
|K \setminus K_\ell| < 2^{\ell - n} \eps.
\end{equation}
\end{quote}

Note that for $\ell=n$ we obtain $\Phi :=\Psi_n$ and $\widetilde{K} := K_n$ that satisfy the conditions of the lemma (because $L_{j_1\ldots j_n}$ are exactly the cubes $Q_j$), so proving the above claim suffices to prove the lemma.

Let us first consider the case $\ell=1$. Then
$$
\mathscr{Q}=L_1\cup L_2\cup\cdots\cup L_{2^k}=\left(\left[0,\frac{1}{2^k}\right]\cup \left[\frac{1}{2^k},\frac{2}{ 2^k}\right]\cup\cdots\cup\left[\frac{2^k-1}{2^k},1\right]\right)\times [0,1]^{n-1}
$$
is a decomposition of a rectangular box into a union of  adjacent boxes, exactly as in Lemma~\ref{T8}. By Lemma \ref{T25} we can find a diffeomorphism $\Psi_1:\mathscr{Q}\to \mathscr{Q}$ and a~compact set $K_1 \subset K$ such that $\Psi_1 = \id$ in a neighborhood of $K_1 \cup \partial \mathscr{Q}$,
$$
|K \setminus K_1| < 2^{1-n} \eps \text{ and } \int_{L_{j}} f(x) dx = \int_{\Psi_1(L_{j})} g(x) dx \quad \text{ for all }j\in \{1,\ldots,2^k\}.
$$

Assume now that the claim  holds for some $\ell$, $1\leq \ell<n$, so that there exists a~diffeomorphism $\Psi_\ell:\mathscr{Q}\to \mathscr{Q}$ and a~compact set $K_\ell \subset K$ satisfying \eqref{eq106}, \eqref{eq10}, \eqref{eq66}.

For any $j_1,\ldots,j_\ell$
$$
\int_{L_{j_1\ldots j_\ell}} f(x) dx =\int_{\Psi_\ell(L_{j_1\ldots j_\ell})} g(x) dx=\int_{L_{j_1\ldots j_\ell}}(g\circ \Psi_\ell)(x) J_{\Psi_\ell}(x) dx=\int_{L_{j_1\ldots j_\ell}}\tilde{g}_\ell(x) dx
$$
for $\tilde{g}_\ell(x)=(g\circ \Psi_\ell)(x) J_{\Psi_\ell}(x)$.
Since $\Psi_\ell=\id$ near $\partial \mathscr{Q}$, the diffeomorphism $\Psi_\ell$ is orientation preserving and hence $J_{\Psi_\ell}>0$.
Since $\Psi_\ell=\id$ near $K_\ell$, $D\Psi_\ell=\mathcal{I}$ on $K_\ell$, and hence
$f(x)=g(x)=\tilde{g}_\ell(x)$ for a.\,e.\ $x \in K_\ell$.

Now, let us fix $j_1\ldots j_\ell$. Note that
\begin{equation*}
\begin{split}
L_{j_1\ldots j_{\ell }}&=L_{j_1\ldots j_\ell\, 1}\cup L_{j_1\ldots j_\ell\, 2}\cup\cdots\cup L_{j_1\ldots j_\ell\, 2^k}\\
&=P_{j_1\ldots j_\ell}\times \left(\left[0,\frac{1}{2^k}\right]\cup \left[\frac{1}{2^k},\frac{2}{ 2^k}\right]\cup\cdots\cup\left[\frac{2^k-1}{2^k},1\right]\right)\times [0,1]^{n-\ell-1}
\end{split}
\end{equation*}
is again a decomposition of the rectangular box $L_{j_1\ldots j_\ell}$ into a union of adjacent boxes isometric to that in Lemma~\ref{T8}, satisfying assumptions of Lemma~\ref{T25} for the functions $f$, $\tilde{g}_\ell$, and the compact set $K_\ell \cap L_{j_1 \ldots j_\ell}$. Therefore, we can find a diffeomorphism $\Upsilon_{j_1\ldots j_\ell}:L_{j_1\ldots j_\ell}\to L_{j_1\ldots j_\ell}$ and a~compact set $\widetilde{K}_{j_1 \ldots j_\ell}\subset K_\ell \cap L_{j_1 \ldots j_\ell}$ such that $\Upsilon_{j_1\ldots j_\ell} = \id$ in a~neighborhood of $\widetilde{K}_{j_1 \ldots j_\ell} \cup \partial L_{j_1\ldots j_\ell}$,
$$
\int_{L_{j_1\ldots j_\ell\,j_{\ell+1}}} f(x) dx = \int_{\Upsilon_{j_1\ldots j_\ell}(L_{j_1\ldots j_\ell\,j_{\ell+1}})} \tilde{g}_\ell(x) dx \quad \text{ for all }j_{\ell + 1}\in \{1,\ldots,2^k\}
$$
and
\begin{equation} \label{eq73}
| \left( K_\ell \cap L_{j_1 \ldots j_\ell} \right) \setminus \widetilde{K}_{j_1 \ldots j_\ell}| < 2^{-k\ell}\,2^{\ell - n} \, \eps.
\end{equation}

Since the diffeomorphisms $\Upsilon_{j_1\ldots j_\ell}$ are identity near $\partial L_{j_1\ldots j_\ell}$, they agree near the boundary of adjacent boxes $L_{j_1\ldots j_\ell}$ and thus we can glue them to a diffeomorphism $\Upsilon:\mathscr{Q}\to \mathscr{Q}$, identity near $\partial \mathscr{Q}$. Setting $K_{\ell+1} = \bigcup_{j_1 \ldots j_\ell} \widetilde{K}_{j_1 \ldots j_\ell}$, we see that $K_{\ell+1}\subset K_\ell$, that $\Upsilon(x) = x$ in a neighborhood of $K_{\ell + 1}$ and that for any $j_1,\ldots,j_{\ell+1}$
\[
\begin{split}
\int_{L_{j_1\ldots j_\ell\,j_{\ell+1}}} f(x) dx
&=
\int_{\Upsilon(L_{j_1\ldots j_\ell\,j_{\ell+1}})} \tilde{g}_\ell(x) dx=\int_{\Upsilon(L_{j_1\ldots j_\ell\,j_{\ell+1}})} (g\circ \Psi_\ell)(x)J_{\Psi_\ell}(x) dx\\
&=
\int_{(\Psi_\ell\circ\Upsilon)(L_{j_1\ldots j_\ell\,j_{\ell+1}})} g(x) dx.
\end{split}
\]
Moreover, by \eqref{eq73} we can see that
$$
|K_\ell \setminus K_{\ell + 1}| = \sum_{j_1 \ldots j_\ell} |\left( K_\ell \cap L_{j_1 \ldots j_\ell} \right) \setminus \widetilde{K}_{j_1 \ldots j_\ell}| < 2^{\ell - n} \eps,
$$
which implies that
$$
|K \setminus K_{\ell+1}| = |K \setminus K_\ell | + |K_\ell \setminus K_{\ell + 1}| < 2^{\ell + 1 - n} \eps.
$$

Therefore, we set $\Psi_{\ell+1}=\Psi_\ell\circ\Upsilon$, which is identity in a neighborhood of  $K_{\ell + 1}\cup\partial \mathscr{Q}$ and consequently satisfies the claim for $\ell + 1$ in place of $\ell$, which completes the inductive step and the proof.
\end{proof}

Moreover, we immediately see that the diameters of the cubes $Q_j$ can be made arbitrarily small by taking large $k$. However, we have no control over the diameters of $\Phi(Q_j)$. The next proposition corrects that.

\begin{proposition}
\label{T43}
Let $\mathscr{Q}=[0,1]^n$ and $K$ be a~compact subset of $\mathscr{Q}$. Assume that $f, g \in L^1(\mathscr{Q})$, $f,g>0$ a.\,e.\ satisfy
\begin{equation} \label{eq88}
\int_{\mathscr{Q}} f(x) \, dx = \int_{\mathscr{Q}} g(x) \, dx \text{ and } f(x) = g(x) \text{ for a.\,e.\ } x \in K.
\end{equation}
Then for any $\eps >0 $, $\eta > 0$, there exist a diffeomorphic dyadic partition $\mathscr{Q}=\bigcup_{j=1}^{2^{nN}} P_j$, a diffeomorphism $\Psi:\mathscr{Q}\to \mathscr{Q}$ and a~compact set $\widetilde{K} \subset K$,  with $|K \setminus \widetilde{K}| < \eta$ such that $\Psi=\id$ in a neighborhood of $\widetilde{K}\cup \partial \mathscr{Q}$, $\diam P_j<\eps$,  $\diam \Psi(P_j)<\eps$,
and
$$
\displaystyle\int_{P_j} f(x)\, dx = \int_{\Psi(P_j)} g(x)\, dx,
\quad
\text{for }  j=1,2,\ldots,2^{nN}.
$$
\end{proposition}
\begin{remark}
Recall that here, the diffeomorphic dyadic partition means that there is a diffeomorphism $\Theta:\mathscr{Q}\to \mathscr{Q}$ such that $\Theta(Q_j)=P_j$, where $\mathscr{Q}=\bigcup_{j=1}^{2^{nN}}Q_j$ is the standard dyadic partition of $\mathscr{Q}$
into $2^{nN}$ identical cubes of side-length $2^{-N}$. In fact, the diffeomorphism $\Theta$ constructed in the proof will have the additional property that $\Theta=\id$ in a neighborhood of $\partial\mathscr{Q}$.
\end{remark}
\begin{remark}
The idea is to take the diffeomorphism $\Phi$ from Lemma~\ref{T9} and then apply
a version of Lemma~\ref{T9} for a diffeomorphic dyadic partition to each of the sets $\Phi(Q_j)$ and the inverse diffeomorphism $\Phi^{-1}:\Phi(Q_j)\to Q_j$. While we do not have control of the diameters of the sets $\Phi(Q_j)$, after the construction described here, we will partition $\Phi(Q_j)$ into sets of as small diameters as we wish.
\end{remark}

\begin{proof}
Choose $k\in \bbbn$ such that $2^{-k}\sqrt{n}<\eps$ and let $\mathscr{Q}=\bigcup_{j=1}^{2^{nk}} Q_j$ be the dyadic decomposition into $2^{nk}$ identical cubes of side-length $2^{-k}$, so $\diam Q_j<\eps$. Let $\Phi$ be the diffeomorphism and $K_1 \subset K$ the compact set provided by Lemma~\ref{T9} so that $\Phi = \id$ in a~neighborhood of $K_1 \cup \partial \mathscr{Q}$, $|K \setminus K_1| < \eta/2$ and
$$
\int_{Q_j} f(x)\, dx = \int_{\Phi(Q_j)} g(x)\, dx
\quad
\text{for } j=1,2,\ldots,2^{nk}.
$$

The diffeomorphism $\Phi$ is uniformly continuous in $\mathscr{Q}$; let $\delta>0$ be such that
\begin{equation}
\label{eq117}
|\Phi(x)-\Phi(y)|<\eps
\text{ whenever } |x-y|<\delta.
\end{equation}
Let $\ell\in \bbbn$ satisfy $2^{-(k+\ell)}\sqrt{n}<\delta$ and
consider the dyadic partition $\mathscr{Q}=\bigcup_{j=1}^{2^{nk}}\bigcup_{i=1}^{2^{n\ell}} \tilde{P}_{ij}$ into identical cubes of side-length $2^{-(\ell+k)}$, so that each cube $Q_j$ is partitioned into $2^{n\ell}$ identical cubes $\tilde{P}_{ij}$. Clearly, $\diam\widetilde{P}_{ij}<\delta$.

For any $j$ we have, bearing in mind that $J_{\Phi}>0$,
$$
\int_{Q_j} f(x)\, dx = \int_{\Phi(Q_j)} g(x)\, dx=\int_{Q_j} (g\circ \Phi)(x)J_{\Phi}(x)\,dx,
$$
so if we denote $\tilde{g}(x):=f(x)$, $\tilde{f}(x):=(g\circ\Phi)(x)J_{\Phi}(x)$,\footnote{$\widetilde{g}=f$ is not a typo; we reverse notation of $f$ and $g$ for a reason.} we have
$$
\int_{Q_j} \tilde{f}(x)\, dx = \int_{Q_j} \tilde{g}(x)\, dx.
$$
Observe that for a.\,e.\ $x \in K_1$, $\tilde{f}(x) = \tilde{g}(x)$. Indeed,
$$
\widetilde{f}(x)=(g\circ\Phi)(x)J_\Phi(x)=g(x)=f(x)=\widetilde{g}(x).
$$

Applying Lemma \ref{T9} with $Q_j$ in place of $\mathscr{Q}$, $\tilde{f}$ for $f$ and $\tilde{g}$ for $g$, partition $Q_j=\bigcup_{i=1}^{2^{n\ell}}\tilde{P}_{ij}$ and the compact set $K_1 \cap Q_j$ yields a~diffeomorphism $\Theta_j:Q_j\to Q_j$ and a~compact set $K_{2j} \subset K_1 \cap Q_j$ such that $\Theta_j = \id$ in a~neighborhood of $K_{2j} \cup \partial Q_j$,
\begin{equation}
\label{eq19}
|(K_1 \cap Q_j) \setminus K_{2j}| < \tfrac{\eta}{2} \, 2^{-nk} \text{ and } \int_{\tilde{P}_{ij}} \tilde{f}(x)\, dx = \int_{\Theta_j(\tilde{P}_{ij})} \tilde{g}(x)\, dx.
\end{equation}

 Since $\Theta_j$ are identity near $\partial Q_j$, they glue together to a diffeomorphism $\Theta:\mathscr{Q}\to \mathscr{Q}$, identity near $\partial \mathscr{Q}$. Set
$$
\widetilde{K} := \bigcup \nolimits_j K_{2j} \subset K_1 \subset K.
$$
By~\eqref{eq19}, $|K_1 \setminus \widetilde{K}| < \eta/2$ and consequently,
\begin{equation} \label{eq107}
|K \setminus \widetilde{K}| = |K \setminus K_1| + |K_1 \setminus \widetilde{K}| < \eta.
\end{equation}
Moreover, $\Theta = \id$ in a~neighborhood of $\widetilde{K}$. Let $P_{ij}=\Theta(\tilde{P}_{ij})$.

Then, we check that
\begin{equation*}
\begin{split}
\int_{P_{ij}} f(x)\,dx&=\int_{\Theta(\tilde{P}_{ij})} \tilde{g}(x) \, dx \overset{\eqref{eq19}}{=}\int_{\tilde{P}_{ij}} \tilde{f}(x)\, dx\\
&=\int_{\tilde{P}_{ij}} (g\circ\Phi)(x)J_{\Phi}(x)\, dx\\
&=\int_{\Phi(\tilde{P}_{ij})} g(x)\,dx=\int_{\Phi(\Theta^{-1}(P_{ij}))} g(x)\,dx,
\end{split}
\end{equation*}
so setting $\Psi=\Phi\circ\Theta^{-1}$ we get
$$
\int_{P_{ij}} f(x)\,dx=\int_{\Psi(P_{ij})} g(x)\,dx.
$$
Since $\Theta^{-1} = \id$ and $\Phi = \id$ in a~neighborhood of $\widetilde{K} \cup \partial \mathscr{Q}$, $\Psi = \id$ there, as well. We have already checked in~\eqref{eq107} that $|K \setminus \widetilde{K}| < \eta$.  Finally, $P_{ij}\subset Q_j$, so $\diam P_{ij}\leq \diam Q_j<\eps$; also $\diam \tilde{P}_{ij}=2^{-(k+\ell)}\sqrt{n}<\delta$, so  $\diam \Psi(P_{ij})=\diam \Phi(\tilde{P}_{ij})<\eps$ by \eqref{eq117}.

Note also that the partition of $\mathscr{Q}$ into the $2^{n(\ell+k)}$ sets $P_{ij}$ is diffeomorphic to the dyadic partition $\mathscr{Q}=\bigcup_{j=1}^{2^{nk}}\bigcup_{i=1}^{2^{n\ell}} \tilde{P}_{ij}$ by the diffeomorphism $\Theta:\mathscr{Q}\to \mathscr{Q}$.
\end{proof}

Eventually, we will need an analogue of the previous proposition in terms of diffeomorphic images of cubes, which is again a~consequence of the change of variables theorem.
\begin{proposition}
\label{T32}
Let $\mathscr{Q}=[0,1]^n$ and $\Theta: \mathscr{Q} \to \bbbr^n$ be a~diffeomorphism of the closed unit cube $\mathscr{Q}$. Let $\widetilde{\mathscr{Q}} = \Theta(\mathscr{Q})$ and $K$ be a~compact subset of $\widetilde{\mathscr{Q}}$. Suppose that
$f, g \in L^1({\widetilde{\mathscr{Q}}})$, $f,g>0$ a.\,e.\ satisfy
\begin{equation}
\label{eq120}
\int_{\widetilde{\mathscr{Q}}} f(x) \, dx = \int_{\widetilde{\mathscr{Q}}} g(x) \, dx \text{ and } f(x) = g(x) \text{ for a.\,e.\ } x \in K.
\end{equation}
Then, for any $\varepsilon, \eta >0$ there exists a diffeomorphic dyadic partition  $\widetilde{\mathscr{Q}}=\bigcup_{j=1}^{2^{nN}} \widetilde{P}_j$, a diffeomorphism $\Psi:\widetilde{\mathscr{Q}}\to \widetilde{\mathscr{Q}}$ and a~compact set $\widetilde{K} \subset K$, with $|K \setminus \widetilde{K}| < \eta$ such that $\Psi = \id$ in a~neighborhood of $\widetilde{K} \cup \partial \widetilde{\mathscr{Q}}$, $\diam \widetilde{P}_j<\eps$, $\diam \Psi(\widetilde{P}_j)<\eps$ and
$$
\int_{\widetilde{P}_j} f(x)\, dx = \int_{\Psi(\widetilde{P}_j)} g(x)\, dx
\quad
\text{for } j=1,2,\ldots,2^{nN}.
$$
\end{proposition}

At the end of this section, we add a technical lemma concerning connectedness of complements of compact sets.
The proof of Lemma~\ref{T40} explains the shape of the set $K$ on Figure~\ref{fig:1}.
\begin{lemma}
\label{T40}
Let $E$ be a~measurable subset of a~domain $\Omega$ in $\bbbr^n$. Then, for any $\eps >0$ there is a~compact set $K \subset E$ such that $|E \setminus K| < \eps$ and $\Omega \setminus K$ is connected.
\end{lemma}

\begin{proof}
First, let $K_1\subset E$ be a compact set such that $|E\setminus K_1|\leq \eps/2$.

Next, let $K_2\subset\Omega$ denote a finite sum of pairwise disjoint closed cubes, $K_2=\bigcup_{i=1}^N P_i$, such that $|K_1\setminus K_2|<\eps/4$.
In each of the cubes $P_i$ let $C_i$ denote the standard Cantor set of positive measure, such that $|P_i\setminus C_i|<\eps/(4N)$. Denote $C=\bigcup_{i=1}^N C_i$.
Finally, set $K=K_1\cap C$.

Then $K$ is obviously compact,
$$
E\setminus K\subset (E\setminus K_1)\cup (K_1\setminus K_2)\cup (K_2\setminus K)
$$
and $K_2\setminus K\subset \bigcup_{i=1}^N (P_i\setminus C_i)$, thus
$$
|E\setminus K|\leq |E\setminus K_1|+|K_1\setminus K_2|+\sum_{i=1}^N|P_i\setminus C_i|<\eps/2+\eps/4+N\eps/(4N)=\eps.
$$

The fact that $\Omega\setminus K$ is connected follows from construction of the Cantor sets $C_i$. Indeed, in each of the cubes $P_i$ the complement of the set $C_i$ is path-connected and contains $\partial P_i$. It follows that any point $p$ in $\Omega\setminus C$ can be connected to a given point $q$ in $C$ by a path which intersects $C$ in $q$ only. Therefore the complement of any subset of $C$ (in particular, of $K$) is path connected.
(The fact that $\Omega\setminus K$ is connected follows also from more involved results in topological dimension theory \cite[p.\,22 and p.\,48]{HW}.)
\end{proof}

\section{Approximately differentiable functions}
\label{AD}
\begin{definition}[Classical definition]
Let $f:E\to\bbbr$ be a measurable function defined on a measurable set $E\subset\bbbr^n$. We say that
$f$ is {\em approximately differentiable} at $x\in E$ if there is a linear function $L:\bbbr^n\to\bbbr$
such that for any $\eps>0$ the set
$$
\left\{ y\in E:\, \frac{|f(y)-f(x)-L(y-x)|}{|y-x|} <\eps \right\}
$$
has $x$ as a density point.
\end{definition}

The next result provides a useful characterization of approximate differentiability, see \cite[Proposition~5.2]{GoldsteinHajlasz17}.
\begin{lemma}
\label{T1}
A measurable function $f:E\to\bbbr$ defined in a measurable set $E\subset\bbbr^n$ is approximately
differentiable at $x\in E$ if and only if there is a measurable set $E_x\subset E$
and a linear function $L:\bbbr^n\to\bbbr$ such that $x$ is a density point of $E_x$ and
$$
\lim_{E_x\ni y\to x} \frac{|f(y)-f(x)-L(y-x)|}{|y-x|} = 0.
$$
\end{lemma}
If a function $f$ is approximately differentiable at $x$, $L$ is unique, and we call it approximate derivative of $f$ at $x$. The approximate derivative will be denoted by $D_{\rm a}f(x)$ or simply by $Df(x)$.
\begin{lemma}
\label{T7}
Assume that $f,g:U\to\bbbr$, $U\subset\bbbr^n$, are given measurable functions and $E\subset U$ is a measurable set. If $f$ is approximately differentiable a.e. and $f=g$ in $E$, then $g$ is approximately differentiable a.e. in $E$ and
\begin{equation}
\label{eq7}
D_{\rm a} g(x)=D_{\rm a} f(x)
\quad
\text{for almost all } x\in E.
\end{equation}
\end{lemma}
\begin{proof}
It easily follows from Lemma~\ref{T1} that \eqref{eq7} is satisfied whenever $x\in E$ is a~density point of $E$, such that $f$ is approximately differentiable at $x$.
\end{proof}
The next result was proved by Whitney \cite{whitney1951}.
\begin{lemma}
\label{T2}
Let $U\subset\bbbr^n$ be open. Then, a function $f:U\to\bbbr$ is approximately differentiable a.e. if an only if for every $\eps>0$ there is a function $f_\eps\in C^1(\bbbr^n)$ such that $|\{x\in U:\, f(x)\neq f_\eps(x)\}| <\eps$.
\end{lemma}
It is easy to see that if $x\in \{f=f_\eps\}$ is a density point of the set $\{f=f_\eps\}$, then $f$ is approximately differentiable at $x$ and $D_{\rm a}f(x)=Df_\eps(x)$.

The proof of the next result is similar to that of Lemma~13 in \cite{GoldsteinHajlasz19}.
\begin{lemma}
\label{T3}
Let $U\subset\bbbr^n$ be open.
Assume that $f:U\to\bbbr^n$ is approximately differentiable a.e. and $\det D_{\rm a}f (x)\neq 0$ a.e. Assume that $V\subset\bbbr^n$ is an open set such that
$f(U)\subset V$ and $g:V\to\bbbr^m$ is approximately differentiable a.e. Then
$g\circ f:U\to\bbbr^m$ is approximately differentiable a.e. and
\begin{equation}
\label{eq1}
D_{\rm a}(g\circ f)(x)=D_{\rm a}g(f(x))\cdot D_{\rm a}f(x)
\quad
\text{for almost all $x\in U$.}
\end{equation}
\end{lemma}
\begin{proof}
According to Lemma~\ref{T2}, there is a sequence of functions $f_k\in C^1(\bbbr^n;\bbbr^n)$,
and a sequence of closed sets $E_k\subset\{f_k=f\}$ such that $E_k$ is contained in the
set of density points of the set $\{f_k=f\}$, $|U\setminus E_k|<1/k$, and $\det D_{\rm a}f(x)\neq 0$ for $x\in E_k$.
Note that if $x\in {E}_k$, then $D_{\rm a}f(x)=Df_k(x)$ and hence $\det Df_k\neq 0$ in ${E}_k$.
It suffices to prove that \eqref{eq1} holds at almost all points $x\in E_k$ for all $k$.

The mapping $f_k$ is a diffeomorphism in a neighborhood of every point of $E_k$ and hence we can decompose $E_k=\bigcup_{i=1}^\infty W_{i}$ into countably many compact sets $W_{i}$, such that $f=f_k$ is bi-Lipschitz on $W_{i}$. Clearly, it suffices to prove that \eqref{eq1} is satisfied at almost all points of $W_{i}$.

According to Lemma~\ref{T2}, there is a~sequence of functions $g_\ell\in C^1(\bbbr^n;\bbbr^m)$,
and a~sequence of closed sets $F_\ell\subset\{g_\ell=g\}$ such that $F_\ell$ is contained in the set of density points of the set $\{g_\ell=g\}$, and $|V\setminus F_\ell|<1/\ell$.
Clearly, $D_{\rm a}g(y)=Dg_\ell(y)$ for $y\in F_\ell$.

Let $M_{i\ell}:=f(W_i)\cap F_\ell$. Since $f$ is bi-Lipschitz on $W_i$,
$$
\Big| f(W_i)\setminus\bigcup_{\ell=1}^\infty M_{i\ell}\Big|=0
\quad
\text{implies that}
\quad
\Big| W_i\setminus\bigcup_{\ell=1}^\infty f^{-1}(M_{i\ell})\Big|=0.
$$
Therefore, it suffices to prove that \eqref{eq1} is satisfied at almost all points of each of the sets $Z_{i\ell}:=W_i\cap f^{-1}(M_{i\ell})$. This is, however, obvious because
for $x\in Z_{i\ell}$, $f(x)=f_k(x)$, $D_{\rm a}f(x)=Df_k(x)$,
$g(f(x))=g_\ell(f(x))$, $D_{\rm a}g(f(x))=Dg_\ell(f(x))$ and hence
$$
D(g_\ell\circ f_k)(x)=Dg_\ell(f_k(x))Df_k(x)=D_{\rm a}g(f(x))\cdot D_{\rm a}f(x).
$$
It remains to observe that since $g_\ell\circ f_k\in C^1$, and $g_\ell\circ f_k=g\circ f$ on $Z_{i\ell}$, we have that $D(g_\ell\circ f_k)(x)=D_{\rm a}(g\circ f)(x)$ at all density points of $Z_{i\ell}$ and hence almost everywhere in $Z_{i\ell}$.
\end{proof}

\begin{lemma}
\label{T35}
Let $\Phi: \mathscr{Q} \to \mathscr{Q}$ be an a.\,e.\ approximately differentiable homeomorphism of $\mathscr{Q}$ such that $\int_\mathscr{Q} |\det D_{\rm a} \Phi| = 1$ and $\det D_{\rm a}\Phi \neq 0$ a.\,e. Then $\Phi$ preserves the sets of zero measure, i.\,e., both $\Phi$ and $\Phi^{-1}$ satisfy Lusin's (N) condition.
\end{lemma}
\begin{proof}
By Theorem \ref{T42}, there is a set $Z\subset \mathscr{Q}$ of measure zero and $\Psi:\mathscr{Q}\to\bbbr^n$ such that $\Phi=\Psi$ on $\mathscr{Q}\setminus Z$, $\Psi$ satisfies Lusin's (N) condition and \eqref{eq84} holds for any measurable $f:\bbbr^n\to \bbbr$.

Since any measurable set is contained in a Borel set of the same measure, we may assume that $Z$ is Borel, and thus $\Phi(Z)$ is Borel.

Observe that whenever $E\subset \mathscr{Q}$ is measurable, so is $\Psi(E)$. Indeed, if  $E'\subset \mathscr{Q}$ is Borel, then $\Psi(E'\setminus Z)=\Phi(E'\setminus Z)$ is Borel (because homeomorphisms preserve Borel sets) and $\Psi(E')$ differs from $\Psi(E'\setminus Z)$ by a subset of $\Psi(Z)$, which has measure zero, thus $\Psi(E')$ is measurable, and for any measurable $E\subset \mathscr{Q}$ there is a Borel $E'\supset E$, $|E'\setminus E|$=0, and thus $\Psi(E)$ differs from $\Psi(E')$ by a subset of $\Psi(E'\setminus E)$ which has measure zero.

For any measurable $E\subset \mathscr{Q}$ setting $f=\chi_{\Psi(E)}$ in \eqref{eq84} yields
\begin{equation}\label{eq113}
\int_E |\det D_{\rm a} \Psi(x)|\,dx=\int_{\Psi(E)}N(\Psi,y)\, dy
\end{equation}
Note that
\begin{itemize}
    \item on $E\setminus Z$ we have $\Psi=\Phi$, thus by Lemma \ref{T7}, $\det D_{\rm a}\Psi=\det D_{\rm a}\Phi$ a.e. in $E$;
    \item since $|\Psi(Z)|=0$ and $\Psi^{-1}(y)=\Phi^{-1}(y)$ is a singleton for $y\in \Psi(E)\setminus \Psi(Z)$, ${N(\Psi,y)=1}$ a.e. in $\Psi(E)$;
    \item $|\Psi(E)|=|\Psi(E\setminus Z)|=|\Phi(E\setminus Z)|$.
\end{itemize}
The above observations combined with \eqref{eq113} give
\begin{equation}
\label{eq114}
\int_E |\det D_{\rm a} \Phi(x)|\,dx=|\Psi(E)|=|\Phi(E\setminus Z)|.
\end{equation}
In particular, for $E=\mathscr{Q}$ we get
$$
1=\int_\mathscr{Q}|\det D_{\rm a} \Phi(x)|\,dx=|\Phi(\mathscr{Q}\setminus Z)|,
$$
therefore $|\Phi(Z)|=0$.
Note also that this and \eqref{eq114} yield that
\begin{equation}
\label{eq118}
\int_E |\det D_{\rm a} \Phi(x)|\,dx=|\Phi(E)|.
\end{equation}

Assume now $N\subset \mathscr{Q}$ is a set of zero measure. Then $\Phi(N)=\Phi(N\setminus Z)\cup\Phi(N\cap Z)=\Psi(N\setminus Z)\cup\Phi(N\cap Z)\subset \Psi(N)\cup \Phi(Z)$, and since both $\Phi(Z)$ and $\Psi(N)$ have measure zero, so has $\Phi(N)$. This proves that $\Phi$ satisfies Lusin's (N) condition.

To prove that $\Phi^{-1}$ satisfies Lusin's (N) condition, assume $A\subset \mathscr{Q}$ is a set of zero measure. We need to show that $\Phi^{-1}(A)$ has measure zero. \emph{A priori} it need not be even measurable.

Let $A'\subset \mathscr{Q}$, $A\subset A'$, be a Borel set of zero measure. Then $\Phi^{-1}(A')$ is Borel and thus measurable. Since we have just proven that $\Phi$ satisfies Lusin's (N) condition, setting $E=\Phi^{-1}(A')$ in \eqref{eq118} we get
$$
\int_{\Phi^{-1}(A')} |\det D_{\rm a} \Phi(x)|\,dx=|A'|=0,
$$
and since $\det D_{\rm a} \Phi\neq 0$ a.e., the set $\Phi^{-1}(A')$ has measure zero, and so has its subset, $\Phi^{-1}(A)$.
\end{proof}
A~somewhat different proof of the following lemma can be also found in~\cite[Corollary 5.1]{DOnofrio14}, we provide it here for convenience of the reader.
\begin{lemma}
\label{T44}
Let $U \subset \bbbr^n$ be open. Assume that $\Phi: U \to \bbbr^n$ is an a.\,e.\ approximately differentiable homeomorphism satisfying Lusin's (N) condition. Then $\Phi^{-1}$ is also a.\,e.\ approximately differentiable and
\begin{equation}
\label{eq102}
D_{\rm a} \Phi^{-1} (y) = (D_{\rm a} \Phi)^{-1} (\Phi^{-1}(y))
\quad
\text{ for a.\,e.\ } y \in \Phi(U).
\end{equation}
\end{lemma}
\begin{proof}
Since it suffices to prove \eqref{eq102} on every subdomain $U'\Subset U$, and clearly,
$$
\int_{U'}|\det D_{\rm a}\Phi(x)|\, dx=|\Phi(U')|<\infty,
$$
we can assume that $U$ is bounded and $\det D_{\rm a}\Phi\in L^1(U)$ is integrable.

Fix $\eps >0$ and choose $\eta$ so that for any $|E| < \eta$, $\int_E |\det D_{\rm a} \Phi| < \eps$. By Lemma~\ref{T2} and approximation of measurable sets with compact ones, we find $f_\eta \in C^{1}(\bbbr^n, \bbbr^n)$ and a~compact set $K_\eta$ such that
\begin{equation} \label{eq101}
f_\eta = \Phi \text{ and } D f_\eta = D_{\rm a} \Phi \text{ on } K_\eta \quad \text{and} \quad |U \setminus K_\eta| < \eta.
\end{equation}

Set
$$
Z:= \{x \in U: \, \Phi \text{ is approximately differentiable at } x \text{ and } \det D_{\rm a} \Phi = 0 \}.
$$

We aim to show that $\Phi^{-1}$ is approximately differentiable at density points of the set $\Phi(K_\eta \setminus Z)$ and that \eqref{eq102} holds there. This suffices to prove the lemma. Indeed, by \eqref{eq115} $\Phi(Z)=0$ and thus, by the choice of $\eta$,
$$
| \Phi(U) \setminus \Phi(K_\eta \setminus Z)| \leq |\Phi(U \setminus K_\eta)| + |\Phi(Z)| < \eps.
$$
Consequently, approximate differentiability of $\Phi^{-1}$ a.\,e.\ on $\Phi(U)$ follows from arbitrariness of $\eps$.

Observe that for any $y \in \Phi(K_\eta \setminus Z)$, $f_\eta(\Phi^{-1}(y))=y$, $\det Df_\eta (\Phi^{-1}(y)) \neq 0$, and the inverse function theorem implies that $f_\eta$ is a~diffeomorphism on some neighborhood of $\Phi^{-1}(y)$. Denote the inverse diffeomorphism of $f_\eta$ restricted to a neighborhood of $\Phi^{-1}(y)$ by $g_y$ so we may assume that $g_y$ is defined in $B_y:=B(y,r)$ for some $r>0$ and $f_\eta\circ g_y=\id$ on $B_y$. In particular,
$$
Dg_y(z)=(Df_\eta)^{-1}(g_y(z))
\text{ for } z\in B_y
\quad
\text{and}
\quad
g_y=\Phi^{-1} \text{ on } \Phi(K_\eta\setminus Z)\cap B_y.
$$
Take a~density point $y$ of the set $\Phi(K_\eta \setminus Z)$ and set
$$
L := (D_{\rm a} \Phi)^{-1}( \Phi^{-1} (y)) \overset{\eqref{eq101}}{=} (Df_\eta)^{-1} (g_y(y))=Dg_y(y).
$$
Since $g_y$ is differentiable at $y$,
$$
\lim_{\substack{z \to y, \\ z \in \Phi(K_\eta \setminus Z)}} \frac{|\Phi^{-1}(z) - \Phi^{-1}(y) - L (z - y)|}{|z - y|} = \lim_{\substack{z \to y, \\ z \in \Phi(K_\eta \setminus Z)}} \frac{|g_y(z) - g_y(y) - L(z -y)|}{|z-y|} = 0,
$$
as for $z$ sufficiently close to $y$, $z \in B(y, r)$, where $g_y$ is well-defined. By Lemma~\ref{T1}, this shows that $\Phi^{-1}$ is approximately differentiable at density points of $\Phi(K_\eta \setminus Z)$ and that \eqref{eq102} holds. As explained earlier, this finishes the proof.
\end{proof}

\subsection{Reflection on a measurable set}
Let
\begin{equation}
\label{eq110}
\mathcal{R}:=
\left[
\begin{array}{ccccc}
1       &      0       &   \ldots   &   0      &     0      \\
0       &      1       &   \ldots   &   0      &     0      \\
\vdots  &   \vdots     &   \ddots   &  \vdots  &    \vdots  \\
0       &      0       &   \ldots   &   1      &     0      \\
0       &      0       &   \ldots   &   0      &    -1      \\
\end{array}
\right]\, .
\end{equation}
The main result of this section is the following theorem. It is a generalization of the main result in \cite{GoldsteinHajlasz17}, which is recalled as Theorem~\ref{T33} above.
\begin{theorem}
\label{T5}
Let $\mathscr{Q}=[0,1]^n$ and let $E\subset \mathscr{Q}$ be a measurable set. Then there
exists an almost everywhere approximately differentiable homeomorphism $\Phi$ of the cube $\mathscr{Q}$ onto itself, such that $\Phi|_{\partial \mathscr{Q}}=\id$ and
\begin{equation}
\label{eq3}
D_{\rm a}\Phi(x)=
\begin{cases}
\mathcal{R} & \text{for almost all } x\in E,\\
\mathcal{I} & \text{for almost all } x\in \mathscr{Q}\setminus E.
\end{cases}
\end{equation}
Moreover, $\Phi$ is a limit, in the uniform metric $d$, of $C^\infty$ volume preserving diffeomorphisms that are identity in a neighborhood of $\partial \mathscr{Q}$.
\end{theorem}
\begin{remark}
\label{R2}
As $\det D_{\rm a}\Phi = \pm 1$ a.\,e.\ on $\mathscr{Q}$, Lemma~\ref{T35} implies that $\Phi$ satisfies Lusin's (N) condition. It then follows from \eqref{eq115} that the homeomorphism $\Phi$ is measure preserving. Moreover, by Lemma~\ref{T44}, $\Phi^{-1}$ is also approximately differentiable a.\,e.\ on $\mathscr{Q}$ and $D_{\rm a} \Phi^{-1} (y) = (D_{\rm a} \Phi)^{-1} (\Phi^{-1}(y))$ for a.\,e.\ $y \in \mathscr{Q}$. The fact that $\Phi$ is volume preserving can also be directly concluded from the fact that $\Phi$ is a limit of volume preserving diffeomorphisms in the metric $d$, see \cite[Lemma~1.2]{GoldsteinHajlasz17}.
\end{remark}

\begin{proof}
Denote the function on the right hand side of \eqref{eq3} by $\lambda$, so \eqref{eq3} reads as $D_{\rm a}\Phi=\lambda$ a.e. Note that
$$
\mathcal{R}{\lambda}=
\begin{cases}
\mathcal{I} & \text{for almost all } x\in E,\\
\mathcal{R} & \text{for almost all } x\in \mathscr{Q}\setminus E.
\end{cases}
$$
The homeomorphism $\Phi$ will be constructed as a limit of a sequence of measure preserving homeomorphisms $\Phi_j$, where at almost every point the approximate derivative of $\Phi_{j}$ will be equal $\mathcal{I}$ or $\mathcal{R}$; in other words, it will be equal $\lambda(x)$ of $\mathcal{R}{\lambda}(x)$. The main idea behind the construction of the sequence is as follows. If $D_{\rm a}\Phi_j(x)=\lambda(x)$, then $x$ is a `good' point. Otherwise we have a `bad' point, where $D_{\rm a}\Phi_j(x)=\mathcal{R}{\lambda}(x)$. We want to modify $\Phi_j$ in a way that bad points will became good.

The main step in the construction of $\Phi_{j+1}$ from $\Phi_j$ is based on the following idea.
Let $B\subset \mathscr{Q}$ be the set of bad points, i.e., $D_{\rm a}\Phi_j=\mathcal{R}\lambda$ on $B$, and let $G=\mathscr{Q}\setminus B$, so $D_{\rm a}\Phi_j=\lambda$ a.e. in $G$, i.e., almost all points of $G$ are good.

Suppose $K\subset \mathscr{Q}$ is a closed cube and assume that $|K\cap\Phi_j(B)|=(1-\eps)|K|$. Clearly, $|K\cap \Phi_j(G)|=\eps |K|$. That is, most of the cube $K$ is covered by the image of bad points.

Note that $D_{\rm a}\Phi_j=\mathcal{R}{\lambda}$ in $\Phi_j^{-1}(K)\cap B$, and $D_{\rm a}\Phi_j=\lambda$ a.e. in $\Phi_j^{-1}(K)\cap G$.
We want to change the derivative of $\Phi_j$ on the bad set $\Phi_j^{-1}(K)\cap B$ from $\mathcal{R}\lambda$ to $\lambda$.

Denote by $\Psi:\mathscr{Q}\to\mathscr{Q}$ the homeomorphism $\Phi$ from Theorem~\ref{T33}.
Let $\kappa:K\to \mathscr{Q}$ be the standard similarity, and let
$$
\Psi_K(y):=
\begin{cases}
\kappa^{-1}\circ\Psi\circ\kappa(y)& \text{ if } y\in K \\
y&\ \text{ if } y\in \mathscr{Q}\setminus K.
\end{cases}
$$
Since $\Psi$ is the identity on the boundary of $\mathscr{Q}$, the mapping $\Psi_K$ is a homeomorphism of $\mathscr{Q}$. Then, we define $\Phi_{j+1}=\Psi_K\circ\Phi_j$.

It follows from the chain rule, Lemma~\ref{T3}, that
$D_{\rm a}\Phi_{j+1}=\lambda$ a.e. in  $\Phi_j^{-1}(K)\cap B$, so the bad set becomes a good one. Unfortunately, we also have that $D_{\rm a}\Phi_{j+1}=\mathcal{R}{\lambda}$ a.e. in $\Phi_j^{-1}(K)\cap G$, so the good set is bad now.
However,
$$
|\Phi_j^{-1}(K)\cap B|=|K\cap\Phi_j(B)|=(1-\eps)|K|
\quad
\text{and}
\quad
|\Phi_j^{-1}(K)\cap G|=|K\cap\Phi_j(G)|=\eps|K|
$$
(because the transformation $\Phi_j$ is measure preserving), so
we changed the bad derivative to the good one on a set of measure $(1-\eps)|K|$ which is much larger than the measure $\eps|K|$ of the set where the good derivative turned bad.
We iterate this procedure infinitely many times in such a way that the measure of the set of good points converges to the measure of the cube $\mathscr{Q}$.

In fact, in the actual construction, we will use not only one cube $K$ to modify $\Phi_j$, but a finite family of cubes that approximates well the measure of the set $\Phi_j(B)$.

By taking sufficiently small cubes we will guarantee that $\Phi_j$ is a Cauchy sequence in the uniform metric $d$. Thus, the sequence will converge to a homeomorphism $\Phi$ and it will follow that $\Phi$ satisfies \eqref{eq3}.

Using the above idea,
we will construct a sequence of measure preserving homeomorphisms $\{\Phi_j\}_{j=0}^\infty$, $\Phi_j:\mathscr{Q}\to \mathscr{Q}$, $\Phi_j|_{\partial \mathscr{Q}}=\id$, that are approximately differentiable almost everywhere and for all $j\geq 1$ satisfy:
\begin{equation}
\label{eq8}
D_{\rm a}\Phi_j(x)\in \{\mathcal{I},\mathcal{R}\} =\{\lambda(x),\mathcal{R}\lambda(x)\}
\quad
\text{a.e. in $\mathscr{Q}$,}
\end{equation}
\begin{equation}
\label{eq11}
|B_j|\leq 2^{-j},
\end{equation}
\begin{equation}
\label{eq9}
d(\Phi_j,\Phi_{j+1})\leq 2^{-j+1},
\quad
|L_j|\leq 2^{-j+1},
\quad
\Big|B_j\setminus \bigcup_{\ell=j}^\infty L_\ell\Big|=0,
\end{equation}
where $d$ is the uniform metric,
$L_j:=\overline{\{x\in \mathscr{Q}:\, \Phi_j\neq\Phi_{j+1}\}}$, and the set
$$
B_j:=\{x\in \mathscr{Q}:\, D_{\rm a}\Phi_j(x)=\mathcal{R}\lambda(x)\}
$$
is the set of points where the approximate derivative of $\Phi_j$ is bad.

Before we construct such a sequence, we show that a sequence satisfying the above conditions converges to a homeomorphism $\Phi$ that has all the required properties, except for being a limit of volume preserving diffeomorphisms (we will take care of it at the end of the proof).

Clearly, \eqref{eq9} and Lemma~\ref{T6} imply convergence in the uniform metric to some homeomorphism $\Phi:\mathscr{Q}\to \mathscr{Q}$ that is identity on the boundary.

Note that by \eqref{eq9},
$$
A_j:=\mathscr{Q}\setminus \bigcup_{\ell=j}^\infty L_\ell.
$$
is an increasing sequence of measurable sets that exhaust $\mathscr{Q}$ up to a set of measure zero.

We have that $\Phi=\Phi_j$ on $A_j$, because $\Phi_j=\Phi_{j+1}=\Phi_{j+2}=\ldots$ on $A_j$, and hence
$$
D_{\rm a}\Phi=D_{\rm a}\Phi_j=\lambda
\quad
\text{almost everywhere in } A_j.
$$
The first equality follows from Lemma~\ref{T7} and
the last equality follows \eqref{eq8}, from the definition of $B_j$ and the fact that
$|A_j\cap B_j|=0$ (which is a consequence of \eqref{eq9} and the definition of $A_j$).
Since the sets $A_j$ exhaust $\mathscr{Q}$ up to a set of measure zero, $D_{\rm a}\Phi=\lambda$ a.e. in $\mathscr{Q}$.

Therefore, it remains to construct the sequence $\Phi_j$ with the properties described above and after the construction is completed to prove that the homeomorphisms $\Phi_j$ can be approximated by volume preserving diffeomorphisms.

We will construct a sequence $\Phi_j$ by induction as a sequence of measure preserving homeomorphisms that are identity on the boundary, and satisfy properties \eqref{eq8} and \eqref{eq11}. Then, the properties listed in \eqref{eq9} will be verified directly, but they are not needed to run the induction.

In the initial step we choose $\Phi_0=\id:\mathscr{Q}\to \mathscr{Q}$, so obviously $B_0=E$ (up to a set of measure zero) and conditions \eqref{eq8} and \eqref{eq11} are satisfied. Now suppose that we already constructed homeomorphisms $\Phi_\ell$ for $\ell\leq j$; we will describe construction of $\Phi_{j+1}$ as a modification of $\Phi_j$.

The homeomorphism $\Phi_j^{-1}$ is uniformly continuous on $\mathscr{Q}$, so let $\delta_j>0$ be such that $|\Phi_j^{-1}(x)-\Phi_j^{-1}(y)|\leq 2^{-j}$ whenever $|x-y|\leq \delta_j$.

Next, we choose a finite family of closed cubes $K^j_i$, $i=1,\ldots,m_j$, with pairwise disjoint interiors, with $\diam K_i^j\leq \min(\delta_j,2^{-j})$, and such that $\bigcup_{i=1}^{m_j}K_i^j$ well approximates the set $\Phi_j(B_j)$ measurewise:
\begin{equation}
\label{eq5}
\bigg|\Phi_j(B_j)\vartriangle \bigcup_{i=1}^{m_j}K_i^j\bigg|\leq 2^{-(j+1)}.
\end{equation}
We set $L_j=\Phi^{-1}_j(\bigcup_{i=1}^{m_j}K_i^j)$; then \eqref{eq5} and the fact that $\Phi_j$ is a measure preserving homeomorphism yields $|B_j\vartriangle L_j|\leq 2^{-(j+1)}$.

Note that with this choice
$$
|L_j|\leq |B_j|+|B_j\vartriangle L_j|\leq 2^{-j}+2^{-(j+1)}<2^{-j+1}.
$$
Let $\kappa_i^j:K_i^j\to \mathscr{Q}$ be the standard similarity maps between the cubes and define
$$
\Psi_j(y):=
\begin{cases}
(\kappa_i^j)^{-1}\circ\Psi\circ\kappa_i^j(y)& \text{ if } y\in K_i^j, \\
y&\ \text{ if } y\in \mathscr{Q}\setminus \bigcup_{i=1}^{m_j}K_i^j.
\end{cases}
$$
Clearly,
\begin{equation}
\label{eq6}
D_{\rm a}\Psi_j(y)=\mathcal{R}
\qquad
\text{for almost all } y\in\bigcup_{i=1}^{m_j} K_i^j=\Phi_j(L_j).
\end{equation}
Since the cubes $K_i^j$ have pairwise disjoint interiors, it follows that $\Psi_j$ is a measure preserving homeomorphism. Moreover, $\Psi_j|_{\partial \mathscr{Q}}=\id$. Then we define $\Phi_{j+1}=\Psi_j\circ\Phi_j$ and clearly $\Phi_{j+1}$ is measure preserving, too, with $\Phi_{j+1}|_{\partial \mathscr{Q}}=\id$. Note that $\Psi_j$ is approximately differentiable a.\,e.\ and  $\det D_{\rm a} \Psi = \pm 1$ a.\,e.\ on $\mathscr{Q}$ so by Lemma~\ref{T35}, $\Phi_{j+1}$ which is the composition of $\Psi_j$ with $\Phi_j$, is approximately differentiable a.\,e.\ on $\mathscr{Q}$.

The homeomorphisms $\Phi_j$ and $\Phi_{j+1}$ differ only on the set $L_j=\bigcup_i\Phi_j^{-1}(K_i^j)$. Since both $\diam \Phi_j^{-1}(K_i^j)$ and $\diam K_i^j$ are at most $2^{-j}$, we have $d(\Phi_j,\Phi_{j+1})\leq 2^{-j+1}$.

Also, the definition of the set $L_j$ and the construction of $\Phi_{j+1}$ shows that
$$
L_j=\overline{\{x\in \mathscr{Q}:\, \Phi_j(x)\neq\Phi_{j+1}(x)\}}.
$$
It follows from Lemma~\ref{T3} and \eqref{eq6} that
$$
D_{\rm a}\Phi_{j+1}(x)= D_{\rm a}\Psi_j(\Phi_j(x))\cdot D_{\rm a}\Phi_j(x)=
\mathcal{R}D_{\rm a}\Phi_j(x)
\quad
\text{for almost all } x\in L_j.
$$
Thus, $D_{\rm a}\Phi_{j+1}=\lambda$ on $L_j\cap B_j$. Also,
$$
D_{\rm a}\Phi_{j+1}(x)=D_{\rm a}\Phi_j(x)
\quad
\text{for almost all } x\in \mathscr{Q}\setminus L_j,
$$
so $D_{\rm a}\Phi_{j+1}=\lambda$  on $(\mathscr{Q}\setminus L_j)\setminus B_j$. This means that the set $B_{j+1}$ of points where $D_{\rm a}\Phi_{j+1}=\mathcal{R}\lambda$ is contained (modulo a set of measure zero) in the complement of the union of these two sets, which is $B_j\vartriangle L_j$, and hence
$$
|B_{j+1}|\leq |B_j\vartriangle L_j|\leq 2^{-(j+1)}.
$$
This completes the proof that $\Phi_{j+1}$ satisfies the induction hypothesis.
What is left to prove is the last part of \eqref{eq9}, $\Big|B_j\setminus \bigcup_{\ell=j}^\infty L_\ell\Big|=0$.

To see that, it suffices to realize that for all $k>j$ the mappings $\Phi_k$ have bad derivative
$D_{\rm a}\Phi_k=\mathcal{R}\lambda$
at almost all points of $B_j\setminus \bigcup_{\ell=j}^\infty L_\ell$,
because $\Phi_k$ is obtained from $\Phi_j$ by a sequence of modifications which happen only in
$\bigcup_{\ell=j}^\infty L_\ell$. Thus $B_j\setminus \bigcup_{\ell=j}^\infty L_\ell\subset \bigcap_{k\geq j}B_k$ (modulo a set of measure zero). Since $|B_k|\to 0$, it follows that $\Big|B_j\setminus \bigcup_{\ell=j}^\infty L_\ell\Big|=0$.

Now it remains to show that the homeomorphism $\Phi$ can be approximated in the metric $d$ by volume preserving diffeomorphisms that are identity in a neighborhood of $\partial Q$.

Let $\Psi_k\overset{d}{\to}\Psi$ be a sequence of volume preserving diffeomorphisms, $\Psi_k=\id$ near $\partial\mathscr{Q}$, from Theorem~\ref{T33}.
Note that the diffeomorphisms
$$
\Psi_{j,k}:=
\begin{cases}
(\kappa_i^j)^{-1}\circ\Psi_k\circ\kappa_i^j(y)& \text{ if } y\in K_i^j, \\
y&\ \text{ if } y\in \mathscr{Q}\setminus \bigcup_{i=1}^{m_j}K_i^j
\end{cases}
$$
are volume preserving and $\Psi_{j,k}=\id$ near $\partial\mathscr{Q}$. Clearly, $\Psi_{j,k}\overset{d}{\to}\Psi_j$ as $k\to\infty$.

An easy induction on $j$ shows that $\Phi_j$ can be approximated in the metric $d$ by $C^\infty$ volume preserving diffeomorphisms $\Phi_{j,k}$ that are identity near $\partial\mathscr{Q}$. If $j=0$, then $\Phi_0=\id$ and we can take $\Phi_{0,k}:=\Phi_0$.
If the claim is true for $j$, $\Phi_{j,k}\overset{d}{\to}\Phi_j$, then, it is easy to check that
$\Phi_{j+1,k}:=\Psi_{j,k}\circ\Phi_{j,k}\overset{d}{\to}\Psi_j\circ\Phi_j=\Phi_{j+1}$.
Therefore, for each $j$, we can find $k_j$ such that $d(\Phi_{j,k_j},\Phi_j)<2^{-j}$ and hence $\Phi_{j,k_j}\overset{d}{\to}\Phi$.
\end{proof}

\section{Proof of Theorem \ref{T36}}
\label{szesc}

\noindent
{\em Proof of Theorem~\ref{T36}.}
\subsection{Reduction to the case \texorpdfstring{$\det T > 0$}{TEXT}}\label{S7.1}
It suffices to prove Theorem \ref{T36} under an additional assumption that $\det T>0$ a.e. Indeed, assume that we have already proven Theorem \ref{T36} under this assumption and, for a general $T$, let
$$
\hat{T}(x)=
\begin{cases}
T(x) &\text{ if }\det T\geq 0,\\
\mathcal{R}T(x) &\text{ if }\det T<0,
\end{cases}
$$
where $\mathcal{R}$ is defined in \eqref{eq110}.
Then $\det \hat{T}=|\det T|>0$ a.e. Let $\hat{\Phi}$ be the almost everywhere approximately differentiable homeomorphism provided by Theorem~\ref{T36} with $\hat{T}$ in place of $T$, $D_{\rm a}\hat{\Phi}=\hat{T}$ a.e.

Let $E:=\{\det T<0\}$.
Note that since $\hat{\Phi}$ satisfies the Lusin (N) condition, the set $\hat{E}:=\hat{\Phi}(E)$ is measurable (because $E$ is the union of a Borel set and a set of measure zero). Theorem~\ref{T5} yields an a.e. approximately differentiable homeomorphism $\Phi'$ such that $\Phi'|_{\partial \mathscr{Q}}=\id$ and
$$
D_{\rm a}\Phi'(x)=
\begin{cases}
\mathcal{R} &\text{ for a.e. } x\in \hat{E},\\
\mathcal{I} &\text{ for a.e. }x\in \mathscr{Q}\setminus \hat{E}.
\end{cases}
$$
Then Lemma~\ref{T3} implies that the composition $\Phi:=\Phi'\circ\hat{\Phi}$ is a.e. approximately differentiable and satisfies the chain rule \eqref{eq1}, so
for a.e. $x\in \mathscr{Q}$ we have
\begin{equation}
\begin{split}
D_{\rm a}\Phi(x)=D_{\rm a} \Phi'(\hat{\Phi}(x))\,D_{\rm a}\hat{\Phi}(x)
&=
\left.\begin{cases}
\mathcal{R}\,\mathcal{R} T(x)  &\text{ for a.e. }x\in E,\\
\mathcal{I}\,T(x) &\text{ for a.e. }x\in \mathscr{Q}\setminus E\\
\end{cases}\right\}
=T(x).
\end{split}
\end{equation}
Obviously, $\Phi|_{\partial \mathscr{Q}}=\id$ and, since $\Phi'$ is measure preserving (see Remark \ref{R2}) and $\hat{\Phi}$ preserves sets of measure zero, $\Phi$ preserves sets of measure zero as well.

Since by Theorem~\ref{T5} and Theorem~\ref{T36} the homeomorphisms $\Phi'$ and $\hat{\Phi}$ can be approximated in the uniform metric $d$ by $C^\infty$-diffeomorphisms $\Phi'_k$ and $\hat{\Phi}_k$, $\Phi'_k=\hat{\Phi}_k=\id$ near $\partial\mathscr{Q}$, one can easily check that
$$
\Phi'_k\circ\hat{\Phi}_k\overset{d}{\to} \Phi'\circ\hat{\Phi}=\Phi.
$$

It remains to check that $\Phi^{-1}$ is approximately differentiable a.\,e.\ on $\mathscr{Q}$ and $D_{\rm a}\Phi^{-1}(y)=T^{-1}(\Phi^{-1})(y)$ for a.\,e.\ $y\in \mathscr{Q}$.
Indeed, $\Phi^{-1} = \hat{\Phi}^{-1} \circ \Phi'^{-1}$ and both $\hat{\Phi}^{-1}$ (by Theorem~\ref{T36}) and $\Phi'^{-1}$ (see Remark~\ref{R2})
are approximately differentiable a.\,e.\ on $\mathscr{Q}$. Since $\det D_{\rm a} (\Phi'^{-1})\neq 0$ a.\,e.\ (Remark~\ref{R2}), $\Phi^{-1}$ is approximately differentiable a.\,e.\ by Lemma~\ref{T3}. Then applying the chain rule~\eqref{eq1} to $\Phi^{-1}\circ\Phi=\mathcal{I}$ yields  that the derivative of $\Phi^{-1}$ has the required form.

This concludes the proof of Theorem \ref{T36} in the general case, provided we can prove it for $T$ such that $\det T>0$ a.e. in $\mathscr{Q}$.

Thus, from now on, we assume that $\det T>0$ a.e.

\subsection{General outline of the construction}
\label{gen}
The general plan of the proof is as follows: To construct a~homeomorphism $\Phi$ such that $\Phi|_{\partial \mathscr{Q}} = \id$ and $D_{\rm a} \Phi = T$ a.\,e., we shall iterate the construction from Theorem~\ref{T31} on the smaller and smaller subsets of $\mathscr{Q}$ on which the derivative is not yet as required.

We inductively show that there exists a~family of orientation preserving $C^1$-diffeomorphisms $\Phi_k$ of $\mathscr{Q}$ and Borel sets $C_k \subset \mathscr{Q}$ with the following properties for $k \geq 1$:
\begin{enumerate}[(i)]
\item $\Phi_k = \mathrm{id}$ near $\partial \mathscr{Q}$;
\item $\Phi_{k+1} = \Phi_k$ on $C_k$;
\item $D\Phi_k = T$  on $C_k$;
\item $C_k$ is an increasing family of sets, $C_1\subset C_2\subset\cdots$,  with $\lim_{k \to \infty} |C_k| = 1$;
\item $d(\Phi_k,\Phi_{k+1}) < 2^{-(k-1)}$ for $k \geq 2$.
\end{enumerate}
The limit map $\Phi := \lim_{k \to \infty} \Phi_k$ is the required homeomorphism. Indeed, property (v) implies that $(\Phi_k)$ is a~Cauchy sequence in the uniform metric $d$, hence its limit is a~homeomorphism as shown in Lemma~\ref{T6}. By (i), $\Phi = \mathrm{id}$ on $\partial \mathscr{Q}$.
Note that $\Phi=\Phi_k$ on $C_k$ by (ii) and (iv). Therefore, Lemma~\ref{T7} and (iii) imply that
$$
D_{\rm a} \Phi = T \text{ a.\,e.\ on } \bigcup_{k=1}^\infty C_k,
$$
and hence $D_{\rm a}\Phi=T$ a.e. on $\mathscr{Q}$, because $|C_k|\to 1$.

At this point, let us stress that Lemma~\ref{T35} implies that $\Phi$ and $\Phi^{-1}$ satisfy the Lusin (N) condition which is part (b) of Theorem~\ref{T36}
and property (a) follows directly from Lemma~\ref{T44}.

Finally, $\Phi_k$ are $C^1$-diffeomorphisms, identity near $\partial\mathscr{Q}$, that converge to $\Phi$ in the uniform metric $d$, but according to Lemma~\ref{T46}, each $\Phi_k$ can be approximated in the metric $d$ by $C^\infty$-diffeomorphisms and (c) follows.

This completes the proof of Theorem~\ref{T36} and it remains to construct diffeomorphisms $\Phi_k$ and Borel sets $C_k$ satisfying (i)-(v).

The construction of the family $\Phi_k$ and of the sets $C_k$ is complicated. In fact, the sets $C_k$ are not constructed inductively, but defined only at the end, when all the steps of the inductive construction are concluded.

The actual inductive construction provides a quadruple $(\Phi_k,\mathscr{P}_k,E_k,L_k)_{k=1}^\infty$:
\begin{itemize}
\item a~diffeomorphism $\Phi_k$ of $\mathscr{Q}$ onto itself; the diffeomorphism $\Phi_k$ is constructed by a modification of $\Phi_{k-1}$,
\item a~partition $\mathscr{P}_k$ of the unit cube; the corrections leading from $\Phi_{k-1}$ to $\Phi_k$ are done at a small scale, i.e., within the elements of the partition $\mathscr{P}_{k-1}$.
\item a large set $E_k$, on which $D\Phi_k=T$,
\item a small set $L_k\subset E_{k-1}$ such that $\Phi_k=\Phi_{k-1}$ on $E_{k-1}\setminus L_k$
(although $D\Phi_{k-1}=T$ on $E_{k-1}$, for technical reasons in the construction of $\Phi_k$ from $\Phi_{k-1}$ we alter $\Phi_k$ on the subset $L_k$ of $E_{k-1}$),
\item the sets $C_k$ are constructed at the very end by \eqref{eq87}.
\end{itemize}
    To be more precise, we have, for $k \geq 1$,
\begin{enumerate}[(a)]
\item a $C^1$-diffeomorphism $\Phi_k: \mathscr{Q} \to \mathscr{Q}$, $\Phi_k = \mathrm{id}$ near $\partial \mathscr{Q}$;
\item a compact set $E_k \subset \mathscr{Q}$ such that $D\Phi_k = T$  on $E_k$;
\item $2^{-(k+1)}< |\mathscr{Q} \setminus E_k| < 2^{-k}$;
\item a Borel set $L_k\subset E_{k-1}$ for $k\geq 2$, such that $\Phi_k = \Phi_{k-1}$ on $E_{k-1} \setminus L_k$ and $|L_k| < 2^{-k}$;
\item a partition $\mathscr{P}_k$ of the unit cube, $\mathscr{P}_k = \{ P_{ki}\}_{i=1}^{M_k}$ for $k\geq 2$, such that
\begin{equation}
\label{eq112}
\Phi_k(P_{k-1,i})=\Phi_{k-1}(P_{k-1,i})
\quad
\text{ for } k\geq 3,
\end{equation}
\begin{equation}
\label{eq37}
        |\Phi_k(P_{ki})| = \int_{P_{ki}} \det T(x) \, dx
        \quad
        \text{ for } k\geq 2,
\end{equation}
\begin{equation}
\label{eq38}
        \diam{P_{ki}} < 2^{-k}, \quad \diam{\Phi_k(P_{ki})} < 2^{-k}
        \quad
        \text{ for } k\geq 2.
\end{equation}
\end{enumerate}

We will show that the family of diffeomorphisms $(\Phi_k)$ with properties (a)-(e) satisfies conditions (i)-(v) for
\begin{equation} \label{eq87}
C_k := \bigcap_{j = k}^\infty \left( E_j \setminus L_{j+1} \right).
\end{equation}
Conditions (i) and  (a) are the same. Clearly, $C_k \subset E_k$, which means that condition (iii) is satisfied. Applying (d) with $k+1$ in place of $k$, we get $\Phi_{k+1} = \Phi_k$ on $E_k \setminus L_{k+1}$ and since $C_k \subset E_k \setminus L_{k+1}$, condition (ii) also holds. Since the sets $C_k$ form an increasing family of sets, in order to show (iv), it remains to check that $\lim_{k \to \infty} |C_k| = 1$.

For the sake of this calculation, set $A_j := E_j \setminus L_{j+1}$. Since $L_{j+1} \subset E_j$ and in view of (c) and (d),
\begin{equation}
\label{eq111}
|\mathscr{Q} \setminus A_j| = |\mathscr{Q}\setminus E_j| + |L_{j+1}| < 2^{-j} + 2^{-(j+1)} = 3 \cdot 2^{-(j+1)}.
\end{equation}
Therefore,
$$
|\mathscr{Q}\setminus C_k|=
\Big| \mathscr{Q} \setminus \bigcap_{j=k}^\infty A_j \Big| = \Big|\bigcup_{j=k}^\infty (\mathscr{Q} \setminus A_j)\Big| \leq \sum_{j=k}^\infty |\mathscr{Q} \setminus A_j| < 3 \sum_{j=k}^\infty 2^{-(j+1)}=3\cdot 2^{-k}.
$$
This implies that
$$
|C_k| > 1 - 3\cdot 2^{-k}
$$
and since $|C_k| \leq 1$, this shows that $|C_k| \to 1$ and finishes the proof of (iv). It remains to show (v).

It follows from condition (e) that for $k\geq 3$
the diffeomorphism $\Phi_{k}$ is a modification of $\Phi_{k-1}$ in each of the sets $P_{k-1,i}$, i.e., $\Phi_{k}(P_{k-1,i})=\Phi_{k-1}(P_{k-1,i})$.
Hence
$$
\Vert\Phi_k-\Phi_{k-1}\Vert_\infty\leq \max_i\big\{\diam\Phi_{k-1}(P_{k-1,i})\big\}<2^{-(k-1)},
$$
$$
\Vert\Phi_k^{-1}-\Phi_{k-1}^{-1}\Vert_\infty\leq \max_i\{\diam P_{k-1,i}\}<2^{-(k-1)},
$$
so $d(\Phi_k,\Phi_{k-1})<2\cdot 2^{-(k-1)}=2^{-(k-2)}$ and (v) follows.
The proof of properties (i)-(v) is complete.\footnote{In the construction of $\Phi_2$ from $\Phi_1$, we modify $\Phi_1$ in $\mathscr{Q}$, so in this step we use the trivial partition $\mathscr{P}_1=\{\mathscr{Q}\}$.}

\subsection{Construction of \texorpdfstring{$\Phi_1$}{TEXT} and \texorpdfstring{$\Phi_2$}{TEXT}.} By a direct application of Theorem~\ref{T31}, we obtain a~diffeomorphism $\Phi_1$ of the unit cube, $\Phi_1 = \mathrm{id}$ near $\partial \mathscr{Q}$, and a~compact set $E_1$ such that  $D\Phi_1 = T$ on $E_1$ and $1/4 < |\mathscr{Q} \setminus E_1| < 1/2$.
Note that $\Phi_1$ and $E_1$ satisfy conditions (a)-(e), because conditions (d) and (e) do not apply to $k=1$.

We shall now describe in detail the construction of $\Phi_2$, which demonstrates all the crucial aspects of the construction of $\Phi_{k}$ based on $\Phi_{k-1}$. The induction step for general $k$ will be described later.

In the course of the proof we use two numbers $\alpha = 1/2$, $\beta = 3/4$. We write $\alpha$, $\beta$ instead of the actual fractions as we believe it makes it easier to transfer this argument to the proof for arbitrary $k$.
Note that $|E_1|>1-\alpha$.

We begin with correcting the way $\Phi_1$ distributes measure so that we are later able to apply Theorem~\ref{T31} in sets of smaller diameter. To this end, we use Proposition~\ref{T43} for the compact set $\Phi_1(E_1)$ and functions
$$
f(y) = \det T(\Phi_1^{-1}(y)) \det D\Phi_1^{-1}(y) \quad \text{and} \quad g(y) = 1.
$$

We check that~\eqref{eq88} is satisfied. By change of variables,
$$
\int_\mathscr{Q} \det T(\Phi_1^{-1}(y)) \det D\Phi_1^{-1}(y) \, dy = \int_{\Phi_1^{-1}(\mathscr{Q})} \det T (x) \, dx = \int_\mathscr{Q} \det T(x) \, dx = \int_\mathscr{Q} 1 \, dx.
$$
Moreover, since $D\Phi_1 = T$ on $E_1$,  we have $f(y) = g(y)$ for all $y\in\Phi_1(E_1)$. Therefore, the assumptions of Proposition~\ref{T43} are satisfied.

By Proposition~\ref{T43}, we get a~diffeomorphism $\Psi: \mathscr{Q} \to \mathscr{Q}$, a~partition $\mathscr{R} = \{ R_{2i} \}_{i=1}^{2^{nN_2}}$ of $\mathscr{Q}$ and a~compact set $\widetilde{K} \subset \Phi_1(E_1)$ with properties described below.

Partition $\mathscr{R}$ is a diffeomorphic dyadic partition and satisfies
\begin{equation} \label{eq90}
\diam \Psi(R_{2i}) < 1/4 \text{ and } \diam (\Phi_1^{-1}(R_{2i})) < 1/4
\end{equation}
(we use here the uniform continuity of $\Phi_1^{-1}$) and
$$
\int_{R_{2i}} f(y) \, dy = |\Psi(R_{2i})|,
$$
which after a change of variables in the integral becomes
\begin{equation} \label{eq91}
\int_{\Phi_1^{-1}(R_{2i})} \det T(x) \, dx = |\Psi(R_{2i})|.
\end{equation}
We find $\eta >0$ such that for any measurable set $A \subset \mathscr{Q}$ with $|A| < \eta$, $\int_A \det \, D\Phi^{-1}_1 < \alpha^2/2$. We choose $\widetilde{K} \subset \Phi_1(E_1)$ so that
$$
|\Phi_1(E_1) \setminus \widetilde{K} |  < \eta.
$$
Moreover, $\Psi = \mathrm{id}$ near $\widetilde{K}\cup\partial \mathscr{Q}$. In view of the inequality above, setting $K := \Phi_1^{-1}(\widetilde{K}) \subset E_1$, we have
\begin{equation}
\label{eq89}
|E_1 \setminus K| = |\Phi_1^{-1} \, ( \Phi_1(E_1) \setminus \widetilde{K} ) | < \alpha^2/2.
\end{equation}

Eventually, we set
$$
\widetilde{\Phi}_2 := \Psi \circ \Phi_1 \quad \text{and} \quad \mathscr{P}_2 := \{ P_{2i} \}_{i=1}^{M_2}, \text{ where } M_2=2^{nN_2}, \text{ and }P_{2i} := \Phi^{-1}_1(R_{2i}).
$$
Then~\eqref{eq90} and~\eqref{eq91} become
\begin{equation} \label{eq92}
\diam{\widetilde{\Phi}_2 (P_{2i})} < 1/4, \quad \diam ({P_{2i}}) < 1/4, \quad \int_{P_{2i}} \det T(x) \, dx = |\widetilde{\Phi}_2(P_{2i})|,
\end{equation}
i.\,e., diffeomorphism $\widetilde{\Phi}_2$ satisfies (e) for $k=2$ (condition \eqref{eq112} does not apply to $k=2$). Observe also that $\widetilde{\Phi}_2 = \Phi_1$ near $K\cup\partial \mathscr{Q}$. Consequently,
\begin{equation}
\label{eq98}
D\widetilde{\Phi}_2 = T \text{ on } K, \quad \widetilde{\Phi}_2 = \mathrm{id}  \text{ near } \partial \mathscr{Q} \quad \text{and} \quad \widetilde{\Phi}_2 = \Phi_1 \text{ on } K \subset E_1.
\end{equation}
We construct $\Phi_2$ by replacing $\widetilde{\Phi}_2$ inside each $P_{2i}$ with a diffeomorphism $\Phi_{2i}$ which has correct derivative $D\Phi_{2i}=T$ on a  larger set. To this end, we would like to keep $\widetilde{\Phi}_2$ on $K$ unchanged and apply Theorem~\ref{T31} to the open set $\mathring{P}_{2i}\setminus K$. Unfortunately, this set need not be connected and Theorem~\ref{T31} cannot be applied. To overcome this difficulty we use Lemma~\ref{T40} to find a compact set $\widetilde{E}_{2i} \subset \mathring{P}_{2i} \cap K$ so that $\mathring{P}_{2i} \setminus \widetilde{E}_{2i}$ is connected and
\begin{equation}
\label{eq93}
| (\mathring{P}_{2i} \cap K) \setminus \widetilde{E}_{2i}| < \alpha^2/(2M_2).
\end{equation}
Set
$$
\widetilde{E}_2 := \bigcup_{i=1}^{M_2} \widetilde{E}_{2i} \subset E_1 \quad \text{and} \quad L_2 := E_1 \setminus \widetilde{E}_{2}.
$$
Clearly, $\widetilde{E}_2$ is compact and summing~\eqref{eq93} over $i = 1, \ldots, M_2$ yields
$$
|K \setminus \widetilde{E}_2 | < \alpha^2/2.
$$
By~\eqref{eq89}, we arrive at
\begin{equation}
\label{eq119}
|L_2| = |E_1 \setminus \widetilde{E}_2|=|E_1\setminus K|+|K\setminus\widetilde{E}_2| < \alpha^2.
\end{equation}

Consequently,
\begin{equation} 
\label{eq96}
|\widetilde{E}_2 | = |E_1| - |L_2| > 1 - \alpha - \alpha^2.
\end{equation}
Let us stress that $\widetilde{E}_2 \subset K$ so \eqref{eq98} yields
\begin{equation} 
\label{eq97}
D\widetilde{\Phi}_2 = T \text{ on } \widetilde{E}_{2}
\end{equation}
and
\begin{equation} 
\label{eq99}
\widetilde{\Phi}_2 = \Phi_1 \text{ on } \widetilde{E}_2 = E_1 \setminus L_2.
\end{equation}

Although $D\Phi_1=T$ on $E_1$,
as explained earlier, the set $L_2$ is the set on which we will have spoiled the already prescribed derivative of $\Phi_1$. This is the cost we bear in order to be able to prescribe the derivative farther. One can also think of the set $\widetilde{E}_2$ as the set of points in which the prescribed derivative \textit{survives} the transition from $\Phi_1$ to $\Phi_2$.

Let us now focus on applying Theorem~\ref{T31}, i.\,e.\,, correcting the derivative of $\widetilde{\Phi}_2$. The set
$$
\Omega_{2i} := \mathring{P}_{2i} \setminus \widetilde{E}_{2i}
$$
is open and connected. Observe that, setting $\Omega_2 := \bigcup_i \Omega_{2i}$,
$$
|\Omega_2| = |\mathscr{Q} \setminus \widetilde{E}_2|,
$$
since the set $\Omega_2$ coincides with $\mathscr{Q} \setminus \widetilde{E}_2$ up to a~set of measure zero which consists of boundaries of the diffeomorphic cubes $P_{2i}$. We check that $\Omega_{2i}$ satisfies
\begin{align}
\label{eq100}
\begin{split}
| \widetilde{\Phi}_2(\Omega_{2i}) | &= |\widetilde{\Phi}_2(P_{2i})| - | \widetilde{\Phi}_2 (\widetilde{E}_{2i})| \overset{\eqref{eq92}}{=} \int_{P_{2i}} \det T(x) \, dx - \int_{\widetilde{E}_{2i}} \det D \widetilde{\Phi}_2(x) \, dx \\
&\overset{\eqref{eq97}}{=} \int_{P_{2i}} \det T(x) \, dx - \int_{\widetilde{E}_{2i}} \det T(x) \, dx = \int_{\Omega_{2i}} \det T(x) \, dx.
\end{split}
\end{align}

We are now in position to use Theorem~\ref{T31} for the domain $\Omega_{2i}$ and the diffeo\-morphism $\widetilde{\Phi}_2$. We find a~diffeomorphism $\Phi_{2i}: \Omega_{2i} \to \widetilde{\Phi}_2(\Omega_{2i})$ and a~compact set $E'_{2i} \subset \Omega_{2i}$ such that
\begin{equation}
\label{eq95}
D\Phi_{2i} = T \text{  on } E'_{2i}, \quad \Phi_{2i} = \widetilde{\Phi}_2 \text{ near } \partial \Omega_{2i} \quad \text{and} \quad |E'_{2i}| > \beta |\Omega_{2i}|.
\end{equation}
Let
$$
E'_2 := \bigcup_{i=1}^{M_2} E'_{2i}.
$$
Clearly, $E'_2$ is compact. By the third condition in~\eqref{eq95}, we have
$$
|E'_2| > \beta |\Omega_2| = \beta |\mathscr{Q} \setminus \widetilde{E}_2 |.
$$

We replace the diffeomorphism $\widetilde{\Phi}_2$ with $\Phi_{2i}$ inside each $\Omega_{2i}$, setting
$$
\Phi_2 := \begin{cases}
    \Phi_{2i} &\text{on } \Omega_{2i},\\
    \widetilde{\Phi}_2 &\text{on } \mathscr{Q} \setminus \Omega_2.
\end{cases}
$$
Observe that thanks to the second condition in~\eqref{eq95}, $\Phi_2$ is indeed a~diffeomorphism. Moreover, $\Phi_2 = \widetilde{\Phi}_2$ near $\bigcup_i \partial P_{2i}$. Since $\partial \mathscr{Q}\subset\bigcup_i \partial P_{2i}$, $\Phi_2=\id$ near $\partial \mathscr{Q}$, whence (a) holds.

Note that $\widetilde{E}_2\subset\mathscr{Q}\setminus\Omega_2$ and $\Phi_2=\widetilde{\Phi}_2$ in a neighborhood of $\widetilde{E}_2$. Hence,
the first condition in~\eqref{eq95} together with~\eqref{eq97} imply that
$$
D\Phi_2 = T \text{  on } \hat{E}_2:=E'_2 \cup \widetilde{E}_2.
$$
We calculate
\begin{align*}
\begin{split}
|\hat{E}_2| &= |E'_2| + |\widetilde{E}_2| > \beta | \mathscr{Q} \setminus \widetilde{E}_2| + |\widetilde{E}_2| = \beta|\mathscr{Q}| + (1 - \beta)|\widetilde{E}_2| \\[2mm]
&\overset{\eqref{eq96}}{>} \beta + (1- \beta)(1 - \alpha - \alpha^2) = 1 - \alpha(1 + \alpha)(1 - \beta) \\[2mm]
&> 1 - \alpha^2,
\end{split}
\end{align*}
since $\beta > 1/(1 + \alpha)$.
Therefore, $|\mathscr{Q}\setminus\hat{E}_2|<\alpha^2=1/4$. By discarding some points from $\hat{E}_2$ if necessary, we obtain a compact set $E_2\subset\hat{E}_2$ such that $1/8<|\mathscr{Q}\setminus E_2|<1/4$. Clearly, $D\Phi_2=T$ on $E_2$.

Let us check if $\Phi_2$ satisfies properties (a)-(e). We have already verified conditions (a), (b) and (c).
Moreover, in view of \eqref{eq92}, (e) holds as well, because $\Phi_2(P_{2i})=\widetilde{\Phi}_2(P_{2i})$ (\eqref{eq112} does not apply to $k=2$).
Since $\Phi_2 = \widetilde{\Phi}_2$ on $\widetilde{E}_2 = E_1 \setminus L_2$, recalling \eqref{eq99} and \eqref{eq119}, we see that (d) also holds.

\subsection{Construction of \texorpdfstring{$\Phi_{k}$}{TEXT} given \texorpdfstring{$\Phi_{k-1}$}{TEXT}.}
Let $k\geq 3$.
As in the construction of $\Phi_2$, we use $\alpha = 1/2$ and $\beta = 3/4$. We assume that we have found $\Phi_{k-1}$ satisfying properties (a)-(e) for $k-1$ instead of $k$ and we show how to construct $\Phi_k$.
In fact, we only need the following properties from the previous step:
\begin{itemize}
\item $\Phi_{k-1}:\mathscr{Q}\to \mathscr{Q}$, $\Phi_{k-1}=\id$ near $\partial \mathscr{Q}$,
\item $D\Phi_{k-1}=T$ on a compact set $E_{k-1}$, $\alpha^k<|\mathscr{Q}\setminus E_{k-1}|<\alpha^{k-1}$,
\item $\Phi_{k-1}(P_{k-1,i})=\int_{P_{k-1,i}}\det T(x)\, dx$.
\end{itemize}
In the construction of $\Phi_2$ from $\Phi_1$, we alter $\Phi_1$ inside $\mathscr{Q}$ while keeping $\Phi_2=\Phi_1$ near $\partial \mathscr{Q}$. In the construction of $\Phi_k$ from $\Phi_{k-1}$ we repeat the same construction, but we alter $\Phi_{k-1}$ inside each diffeomorphic closed cube $P_{k-1,i}$ while keeping $\Phi_k=\Phi_{k-1}$ near $\partial P_{k-1,i}$.
Therefore, the crucial though technical difference between the construction for $k \geq 3$ is that we do not use Proposition~\ref{T43} (stated for a~closed cube) but Proposition~\ref{T32} (stated for a~diffeomorphic closed cube).

Define $E_{k-1, i} := E_{k-1} \cap P_{k-1,i}$ for $i = 1, \ldots, M_{k-1}$. Note that $E_{k-1, i}$ is compact and that $|\bigcup_i E_{k-1, i}| = |E_{k-1}|$.

Let us firstly show that it suffices to construct for $i = 1, \ldots, M_{k-1}$
\begin{itemize}
    \item a~family of diffeomorphisms $\Phi_{ki}: P_{k-1, i} \to \Phi_{k-1}(P_{k-1,i})$,
    \item compact sets $\hat{E}_{ki} \subset \mathring{P}_{k-1, i}$,
    \item Borel sets $L_{ki} \subset E_{k-1, i}$,
    \item a~partition $\mathscr{P}_{ki} = \{ P_{kij}\}_{j=1}^{2^{nN_{ki}}}$ of ${P}_{k-1, i}$
\end{itemize}
such that
\begin{enumerate}
\item $\Phi_{ki} = \Phi_{k-1}$ near $\partial P_{k-1, i}$;
\item $D\Phi_{ki}  = T$ on $\hat{E}_{ki}$;
\item $|\hat{E}_{ki}| > \beta |P_{k-1,i}| + (1 - \beta) |E_{k-1,i} \setminus L_{ki}|$;
\item $\Phi_{ki} = \Phi_{k-1}$ on $E_{k-1,i} \setminus L_{ki}$ and $|L_{ki}| < \alpha^k \, M_{k-1}^{-1}$;
\item partition $\mathscr{P}_{ki}$ is a diffeomorphic dyadic partition and satisfies for $j = 1, \ldots, 2^{nN_{ki}}$
$$
|\Phi_{ki}(P_{kij})| = \int_{P_{kij}} \det T(x) \, dx
$$
and
$$
\diam P_{kij} < 2^{-k}, \quad \diam \Phi_{ki}(P_{kij}) < 2^{-k}.
$$
\end{enumerate}

We set
$$
\Phi_k := \Phi_{ki} \text{ on } P_{k-1, i}.
$$
By condition (1), $\Phi_k$ is indeed a~diffeomorphism and $\Phi_k = \mathrm{id}$ near $\partial \mathscr{Q}$, which is (a). Next, we set
$$
 \hat{E}_k := \bigcup_{i=1}^{M_{k-1}} \hat{E}_{ki} \quad \text{and} \quad L_{k} := \bigcup_{i=1}^{M_{k-1}} L_{ki}\subset E_{k-1}.
$$
Since $D\Phi_k = D\Phi_{ki}=T$ on $\hat{E}_{ki}$, we have that $D\Phi_k=T$ on $\hat{E}_k$.

Summing (4) over $i = 1, \dots, M_{k-1}$, we readily see that  $\Phi_k=\Phi_{k-1}$ on $E_{k-1}\setminus L_k$ and
$|L_k| < \alpha^k$, i.\,e., (d) holds. Moreover, summing (3) over $i = 1, \ldots, M_{k-1}$ and recalling that $|E_{k-1}| > 1 - \alpha^{k-1}$ , we get
\begin{equation*}
|\hat{E}_k| > \beta |\mathscr{Q}| + (1 - \beta) |E_{k-1} \setminus L_k| > \beta + (1-\beta)(1 - \alpha^{k-1} - \alpha^k) > 1 - \alpha^k,
\end{equation*}
so $|\mathscr{Q}\setminus\hat{E}_k|<\alpha^k$. By discarding some points from $\hat{E}_k$ if necessary, we can find another compact set $E_k\subset\hat{E}_k$ such that $\alpha^{k+1}<|\mathscr{Q}\setminus E_k|<\alpha^k$. Clearly, $D\Phi_k=T$ on $E_k$ i.\,e., conditions (b) and (c) are satisfied.  It remains to verify (e).

We define the partition\footnote{While $\mathscr{P}_{ki}$ are a diffeomorphic dyadic partitions, $\mathscr{P}_k$ need not be, because we might divide each of the diffeomorphic cubes $P_{k-1,i}$ into a different number of diffeomorphic dyadic cubes.}  $\mathscr{P}_k$ as the union of partitions $\mathscr{P}_{ki}$
$$
\mathscr{P}_k=\bigcup_{i=1}^{M_{k-1}}\bigcup_{j=1}^{2^{nN_{ki}}}\{ P_{kij}\}.
$$
After re-enumeration of the diffeomorphic cubes $P_{kij}$, we can write
$$
\mathscr{P}_k = \bigcup_{\ell=1}^{M_k}\{ P_{k\ell}\},
\quad
\text{where}
\quad
M_k=\sum_{i=1}^{M_{k-1}} 2^{nN_{ki}}.
$$
Since $\Phi_k=\Phi_{ki}$ on $P_{k-1,i}$, and $\Phi_{ki}=\Phi_{k-1}$ near $\partial P_{k-1,i}$, it follows that $\Phi_k(P_{k-1,i})=\Phi_{k-1}(P_{k-1,i})$, which is \eqref{eq112}. Since $P_{k\ell}=P_{kij}$ for some $i,j$ and $\Phi_k=\Phi_{ki}$ on $P_{k\ell}=P_{kij}\subset P_{k-1,i}$, condition (5) yields
$$
|\Phi_k(P_{k\ell})|=|\Phi_{ki}(P_{kij})|=\int_{P_{kij}}\det T(x)\, dx=\int_{P_{k\ell}} \det T(x)\, dx,
$$
which is \eqref{eq37}. Also
$$
\diam P_{k\ell}=\diam P_{kij}<2^{-k}
\quad
\text{and}
\quad
\diam\Phi_k(P_{k\ell})=\diam \Phi_{ki}(P_{kij})<2^{-k},
$$
which is \eqref{eq38}. This completes the proof of (e) and hence that of (a)-(e).

Fix any $i = 1, \ldots, M_{k-1}$. We shall now show that given $\Phi_{k-1}$, we can construct $\Phi_{ki}$ as described above. As before, we begin by correcting the way $\Phi_{k-1}$ distributes measure so that the subsequent corrections from $\Phi_{k-1}$ to $\Phi_k$ are made at a~smaller scale. To this end, we use Proposition~\ref{T32} for the diffeomorphic closed cube $\Phi_{k-1}(P_{k-1,i})$, the compact set $\Phi_{k-1}(E_{k-1, i})$ and functions
$$
f(y) = \det T(\Phi_{k-1}^{-1}(y)) \det D\Phi_{k-1}^{-1}(y) \quad \text{and} \quad g(y) = 1.
$$
We check that~\eqref{eq120} is satisfied by change of variables and the inductive assumption~\eqref{eq37} for $\Phi_{k-1}$,
\begin{align*}
\int_{\Phi_{k-1}(P_{k-1,i})} \det T(\Phi_{k-1}^{-1}(y)) \det D\Phi_{k-1}^{-1}(y) \, dy &= \int_{P_{k-1,i}} \det T (x) \, dx \overset{\eqref{eq37}}{=} |\Phi_{k-1}(P_{k-1, i})|\\
& = \int_{\Phi_{k-1}(P_{k-1,i})} 1 \, dx.
\end{align*}
Moreover, since $D\Phi_{k-1} = T$ on $E_{k-1, i}$, we have $f(y) = g(y)$ for all $y\in\Phi_{k-1}(E_{k-1,i})$. Therefore, assumptions of Proposition~\ref{T32} are satisfied.

By Proposition~\ref{T32}, we get a~diffeomorphism $\Psi_i: \Phi_{k-1}(P_{k-1,i})\to\Phi_{k-1}(P_{k-1,i})$, a~partition $\mathscr{R}_{ki} = \{ R_{kij} \}_{j=1}^{2^{nN_{ki}}}$ of $\Phi_{k-1}(P_{k-1,i})$ and a~compact set $\widetilde{K}_i \subset \Phi_{k-1}(E_{k-1, i})$ with properties described below.

Partition $\mathscr{R}_{ki}$ is a diffeomorphic dyadic partition and satisfies an analogue of \eqref{eq90} and \eqref{eq91}, namely
\begin{equation}
\label{eq103a}
\diam \Psi_i(R_{kij}) < 2^{-k}, \quad \diam(\Phi_{k-1}^{-1}(R_{kij})) < 2^{-k},
\end{equation}
and
\begin{equation}
\label{103b}
\int_{\Phi^{-1}_{k-1}(R_{kij})} \det T(x) \, dx = |\Psi_i (R_{kij}) |.
\end{equation}

We find $\eta >0$ such that for any measurable set $A \subset \mathscr{Q}$,
$$
|A| < \eta
\quad
\Rightarrow
\quad
\int_A \det D\Phi_{k-1}^{-1} < \frac{\alpha^k}{2M_{k-1}} \,.
$$
We choose a~compact set $\widetilde{K}_i\subset\Phi_{k-1}(E_{k-1,i})$ so that
$$
|\Phi_{k-1}(E_{k-1, i}) \setminus \widetilde{K}_i | < \eta .
$$
Moreover, $\Psi_i = \mathrm{id}$ near $\widetilde{K}_i\cup\partial(\Phi_{k-1}(P_{k-1, i}))$. In view of the inequality above, setting $K_i := \Phi_{k-1}^{-1}(\widetilde{K}_i) \subset E_{k-1, i}$, we have
\begin{equation}
\label{eq89k}
|E_{k-1, i} \setminus K_i| < \frac{\alpha^k}{2M_{k-1}} \,.
\end{equation}

Eventually, we set
$$
\widetilde{\Phi}_{ki} := \Psi_i \circ \Phi_{k-1} \text{ in } P_{k-1,i} \quad \text{and} \quad \mathscr{P}_{ki} := \{ P_{kij} \}_{j=1}^{2^{nN_{ki}}}, \text{ where } P_{kij} := \Phi^{-1}_{k-1}(R_{kij}).
$$
Then~\eqref{eq103a} and~\eqref{103b} become
\begin{equation} \label{eq92k}
\diam{\widetilde{\Phi}_{ki} (P_{kij})} < 2^{-k}, \quad \diam ({P_{kij}}) < 2^{-k}, \quad \int_{P_{kij}} \det T(x) \, dx = |\widetilde{\Phi}_{ki}(P_{kij})|,
\end{equation}
i.\,e., diffeomorphism $\widetilde{\Phi}_{ki}$ satisfies (5). Observe also that
\begin{equation}
\label{eq98k}
\widetilde{\Phi}_{ki} = \Phi_{k-1}
\text{ near } K_i\cup\partial P_{k-1, i}
\quad
\text{so}
\quad
D\widetilde{\Phi}_{ki} = T \text{ on } K_i\subset E_{k-1,i}.
\end{equation}
We will now replace $\widetilde{\Phi}_{ki}$ inside each $P_{kij}$ with a~diffeomorphism $\Phi_{kij}$ which has correct derivative on a~larger set.

As explained earlier, above inequality \eqref{eq93}, we need to use Lemma~\ref{T40} to get a~compact set $\widetilde{E}_{kij} \subset \mathring{P}_{kij} \cap K_i$ such that $\mathring{P}_{kij} \setminus \widetilde{E}_{kij}$ is connected and
\begin{equation}
\label{eq93k}
| ( \mathring{P}_{kij} \cap K_i) \setminus \widetilde{E}_{kij}| < \alpha^k \, 2^{-n N_{ki} - 1}M_{k-1}^{-1}.
\end{equation}
Set
$$
\widetilde{E}_{ki} := \bigcup_{j=1}^{2^{nN_{ki}}} \widetilde{E}_{kij} \subset K_i \quad \text{and} \quad L_{ki} := E_{k-1, i} \setminus \widetilde{E}_{ki}.
$$
Clearly, $\widetilde{E}_{ki}\subset\mathring{P}_{k-1,i}$.
Summing~\eqref{eq93k} over $j = 1, \ldots, 2^{nN_{ki}}$ yields
$$
|K_i \setminus \widetilde{E}_{ki} |  < \frac{\alpha^k}{2M_{k-1}} \,.
$$
By~\eqref{eq89k}, we arrive at
\begin{equation} \label{eq121}
|L_{ki}| = |E_{k-1, i} \setminus K_i|+|K_i\setminus\widetilde{E}_{ki}| < \alpha^k M_{k-1}^{-1}.
\end{equation}
Consequently,
\begin{equation} \label{eq96k}
|\widetilde{E}_{ki} | = |E_{k-1, i}| - |L_{ki}| > |E_{k-1, i}| - \alpha^k \, M_{k-1}^{-1}.
\end{equation}

Let us stress that $\widetilde{E}_{ki} \subset K_i$ so \eqref{eq98k} yields
\begin{equation} 
\label{eq97k}
D\widetilde{\Phi}_{ki} = T \text{ on } \widetilde{E}_{ki}.
\end{equation}
and
\begin{equation} 
\label{eq99k}
\widetilde{\Phi}_{ki} = \Phi_{k-1} \text{ on } \widetilde{E}_{ki} = E_{k-1, i} \setminus L_{ki}.
\end{equation}

Although $D\Phi_{k-1}=T$ on $E_{k-1,i}$,
the set $L_{ki}$ is the set on which we will have spoiled the already prescribed derivative of $\Phi_{k-1}$. On the other hand, the set $\widetilde{E}_{ki}$ consists of points in which the prescribed derivative \textit{survives} the transition from $\Phi_{k-1}$ to $\Phi_{ki}$ and, as we will soon see, to $\Phi_k$.

We now correct the derivative of $\widetilde{\Phi}_{ki}$. The set
$$
\Omega_{kij} := \mathring{P}_{kij} \setminus \widetilde{E}_{kij}
$$
is open and connected. Observe that, setting $\Omega_{ki} := \bigcup_j \Omega_{kij}$,
$$
|\Omega_{ki}| = |P_{k-1, i} \setminus \widetilde{E}_{ki}|.
$$
Exactly as in~\eqref{eq100}, invoking \eqref{eq92k} and \eqref{eq97k} instead of \eqref{eq92} and \eqref{eq97}, we check that $\Omega_{kij}$ satisfies
$$
| \widetilde{\Phi}_{ki} (\Omega_{kij})| = \int_{\Omega_{kij}} \det T(x) \, dx.
$$

We are now in the position to use Theorem~\ref{T31} for the domain $\Omega_{kij}$ and the diffeomorphism $\widetilde{\Phi}_{ki}$. We find a~diffeomorphism $\Phi_{kij}: \Omega_{kij} \to \widetilde{\Phi}_{ki}(\Omega_{kij})$ and a~compact set $E'_{kij} \subset \Omega_{kij}$ such that
\begin{equation} 
\label{eq95k}
D\Phi_{kij} = T \text{ on } E'_{kij}, \quad \Phi_{kij} = \widetilde{\Phi}_{ki} \text{ near } \partial \Omega_{kij} \quad \text{and} \quad |E'_{kij}| > \beta |\Omega_{kij}|.
\end{equation}
and thus $\Phi_{kij}(\Omega_{kij}) = \widetilde{\Phi}_{ki}(\Omega_{kij})$.
Let
$$
E'_{ki} := \bigcup_{j=1}^{2^{nN_{ki}}} E'_{kij}\subset\mathring{P}_{k-1,i}.
$$
By the third condition in~\eqref{eq95k}, we have
$$
|E'_{ki}| > \beta |\Omega_{ki}| = \beta |P_{k-1, i} \setminus \widetilde{E}_{ki}|.
$$

We replace the diffeomorphism $\widetilde{\Phi}_{ki}$ with $\Phi_{kij}$ inside each $\Omega_{kij}$, setting
$$
\Phi_{ki} := \begin{cases}
    \Phi_{kij} &\text{on } \Omega_{kij},\\
    \widetilde{\Phi}_{ki} &\text{on } P_{k-1,i} \setminus \Omega_{ki}.
\end{cases}
$$
Observe that thanks to the second condition in~\eqref{eq95k}, $\Phi_{ki}$ is indeed a~diffeomorphism.

Moreover, $\Phi_{ki} = \widetilde{\Phi}_{ki}$ near $ \partial P_{k-1, i}\subset \bigcup_j \partial P_{kij}$. Since $\widetilde{\Phi}_{ki}=\Phi_{k-1}$ near $\partial P_{k-1, i}$, property (1) follows.

Since $\Phi_{ki}=\widetilde{\Phi}_{ki}$ near $\bigcup_j\partial P_{kij}$, $\Phi_{ki}(P_{kij})=\widetilde{\Phi}_{ki}(P_{kij})$, so \eqref{eq92k} proves property (5).

Note that $\widetilde{E}_{ki}\subset \mathring{P}_{k-1,i}\setminus\Omega_{ki}$ and $\Phi_{ki}=\widetilde{\Phi}_{ki}$ in a neighborhood of $\widetilde{E}_{ki}$, so the first condition in \eqref{eq95k} together with \eqref{eq97k} imply that
$$
D\Phi_{ki} = T \text{ on } \hat{E}_{ki}:=E'_{ki} \cup \widetilde{E}_{ki}\subset\mathring{P}_{k-1,i}.
$$
Note that $\hat{E}_{ki}$ is compact as a~finite sum of compact sets. This proves property (2). We calculate
\begin{equation*}
|\hat{E}_{ki}| = |E'_{ki}| + |\widetilde{E}_{ki}| > \beta |P_{k-1, i} \setminus \widetilde{E}_{ki}| + |\widetilde{E}_{ki}| = \beta|P_{k-1, i} | + (1 - \beta) |E_{k-1, i} \setminus L_{ki}|,
\end{equation*}
where the last equality follows from $\widetilde{E}_{ki}=E_{k-1,i}\setminus L_{ki}$. This proves (3).

Since $\Phi_{ki} = \widetilde{\Phi}_{ki}$ on $\widetilde{E}_{ki} = E_{k-1, i} \setminus L_{ki}$, recalling \eqref{eq99k} and \eqref{eq121}, we see that (4) also holds.

We verified conditions (1)-(5) and that completes the proof of the theorem.
\qed

\section{Appendix}
\begin{lemma}
\label{T48}
Let $E\subset\bbbr^n$, $n\geq 1$, be a measurable set of finite measure. Then the space of measurable functions $f:E\to\bbbr$ is complete with respect to the Lusin metric $d_L$.
\end{lemma}
\begin{proof}
Let $\{ f_k\}_k$ be a Cauchy sequence. It suffices to show that it has a convergent subsequence in the metric $d_L$. To this end, we will show that a subsequence $\{f_{k_\ell}\}_\ell$ such that
$$
d_L(f_{k_\ell},f_{k_{\ell+1}})=|\{f_{k_\ell}\neq f_{k_{\ell+1}}\}|<2^{-(\ell+1)}
$$
is convergent. Let
$$
A_\ell:=\{f_{k_\ell}=f_{k_{\ell+1}}\}
\quad
\text{and}
\quad
C:=\bigcup_{\ell=1}^\infty\bigcap_{i=\ell}^\infty A_i.
$$
Note that for any $\ell$,
$$
|E\setminus C|\leq\Big|E\setminus\bigcap_{i=\ell}^\infty A_i\Big|\leq
\sum_{i=\ell}^\infty |E\setminus A_i|<2^{-\ell},
\quad
\text{so}
\quad
|E\setminus C|=0.
$$
If $x\in C$, then $x\in\bigcap_{i=\ell}^\infty A_i$ for some $\ell$ and hence
$$
f_{k_\ell}(x)=f_{k_{\ell+1}}(x)=f_{k_{\ell+2}}(x)=\ldots
\quad
\text{so}
\quad
f(x):=\lim_{\ell\to\infty} f_{k_\ell}(x)
$$
exists for all $x\in C$ and hence for almost all $x\in E$. In fact, $f(x)=f_{k_\ell}(x)$ for $x\in\bigcap_{i=\ell}^\infty A_i$ and hence
$$
d_L(f,f_{k_\ell})\leq \Big|E\setminus\bigcap_{i=\ell}^\infty A_i\Big|<2^{-\ell},
$$
proving that $f_{k_\ell}\to f$ in the metric $d_L$.
\end{proof}

\end{document}